%% file: main.tex
\documentclass[11pt]{article}
\newcommand\arxiv{}

\ifdefined\arxiv

\usepackage[margin=15mm]{geometry}
\input{sharedarxiv}

\title{A recursive butterfly factorization with optimality guarantees}
\author{David Persson\thanks{New York University \& Flatiron Institute \href{mailto:dup210@nyu.edu}{dup210@nyu.edu}, \href{mailto:dpersson@flatironinstitute.org}{dpersson@flatironinstitute.org}}
\and Paul G. Beckman\thanks{Oden Institute, University of Texas at Austin \href{mailto:paul.beckman@austin.utexas.edu}{paul.beckman@austin.utexas.edu}}
\and Tyler Chen\thanks{New York University, \href{mailto:tyler.chen@nyu.edu}{tyler.chen@nyu.edu}}
\and Diana Halikias\thanks{New York University, \href{mailto:diana.halikias@nyu.edu}{diana.halikias@nyu.edu}}
\and Christopher Musco\thanks{New York University, \href{mailto:cmusco@nyu.edu}{cmusco@nyu.edu} }}

\begin{document}
\maketitle
\begin{abstract}
We formalize a recursive format for representing a butterfly matrix.  
This new format naturally leads to a  simple recursive algorithm for computing a quasi-optimal butterfly approximation to an arbitrary $N \times N$ matrix $\bm{A}$. 
When the entries of $\bm{A}$ are explicitly available, we show that the algorithm computes a butterfly matrix $\bm{B}$ in $O(N^2)$ operations with approximation error 
$\|\bm{A} - \bm{B}\|_\F$ at most a $O(\sqrt{\log(N)})$ factor away from that of the best possible approximation by a butterfly matrix. 
We also develop a matrix-free variant of the method, which uses $\widetilde{O}(\sqrt{N})$ matrix-vector products and $\widetilde{O}(N)$ working memory and, with high probability, returns a butterfly approximation with Frobenius norm error within a $O(N^{1/4})$-factor of the optimal error. 
We show that the algorithm is a reformulation of the hybrid butterfly factorization approach  presented in [Liu et. al.; SISC, 43 (2021)]. Our paper therefore provides the first theoretical quasi-optimality guarantee for that algorithm. 
\end{abstract}

\else
\usepackage[numbers]{natbib}

\input{shared}

\begin{document}
\let\WriteBookmarks\relax
\def\floatpagepagefraction{1}
\def\textpagefraction{.001}
\shorttitle{Recursive butterfly factorization}
\shortauthors{Persson, Beckman, Chen, Halikias, Musco}

\title [mode = title]{A recursive butterfly factorization with optimality guarantees}

\tnotetext[1]{DH and CM were partially supported by NSF Award \#2427363.}

\author[1,2]{David Persson}[orcid=0009-0006-9980-5854]
\cormark[1]
\ead{dup210@nyu.edu}
\affiliation[1]{organization={New York University},
                city={New York},
                state={NY},
                country={USA}}
\affiliation[2]{organization={Flatiron Institute},
                city={New York},
                state={NY},
                country={USA}}

\author[3]{Paul G. Beckman}[]
\ead{paul.beckman@austin.utexas.edu}

\affiliation[3]{organization={Oden Institute, University of Texas at Austin},
                city={Austin},
                state={TX},
                country={USA}}

\author[1]{Tyler Chen}[orcid=0000-0002-1187-1026]
\ead{tyler.chen@nyu.edu}

\author[1]{Diana Halikias}[orcid={0000-0001-7159-7770
}]
\ead{diana.halikias@nyu.edu}

\author[1]{Christopher Musco}[]
\ead{cmusco@nyu.edu}

\cortext[cor1]{Corresponding author}

\begin{abstract}
We formalize a recursive format for representing a butterfly matrix.  
This new format naturally leads to a  simple recursive algorithm for computing a quasi-optimal butterfly approximation to an arbitrary $N \times N$ matrix $\bm{A}$. 
When the entries of $\bm{A}$ are explicitly available, we show that the algorithm computes a butterfly matrix $\bm{B}$ in $O(N^2)$ operations with approximation error 
$\|\bm{A} - \bm{B}\|_\F$ at most a $O(\sqrt{\log(N)})$ factor away from that of the best possible approximation by a butterfly matrix. 
We also develop a matrix-free variant of the method, which uses $\widetilde{O}(\sqrt{N})$ matrix-vector products and $\widetilde{O}(N)$ working memory and, with high probability, returns a butterfly approximation with Frobenius norm error within a $O(N^{1/4})$-factor of the optimal error. 
We show that the algorithm is a reformulation of the hybrid butterfly factorization approach  presented in [Liu et. al.; SISC, 43 (2021)]. Our paper therefore provides the first theoretical quasi-optimality guarantee for that algorithm. 
\end{abstract}

\begin{highlights}
\item Recursive data structure for expressing butterfly matrices
\item Near-optimal butterfly approximations
\item Near-optimality guarantees on computing butterfly approximations from matrix-vector products
\end{highlights}

\begin{keywords}
butterfly matrices, hierarchical matrices, rank-structured matrices, randomized algorithms, matrix-vector products
\end{keywords}

\maketitle

\fi

\input{intro}

\input{prelims}
\input{meta_algorithm}
\input{matvec_algorithm}
\input{numerical}
\input{conclusion}
\input{ai}
\appendix
\input{recursive}
\input{absolute}
\input{memory_efficient_implementation}
\input{naive_matvec_algorithm}
\input{mapbetweenformats}

\bibliographystyle{siam}
\bibliography{bibliography}

\end{document}

%% file: sharedarxiv.tex
\usepackage{lipsum}
\usepackage{stmaryrd}
\usepackage{amsfonts}
\usepackage{graphicx}
\usepackage{epstopdf}
\usepackage{amsmath}
\usepackage{amsopn}
\usepackage{amsthm}
\usepackage{hyperref}
\hypersetup{
    colorlinks,
    linkcolor=c0,
    citecolor=c1,
    urlcolor=c0
}
\usepackage{bm}
\usepackage{amssymb}
\usepackage{xcolor}
\usepackage{overpic}
\usepackage{algorithm}
\usepackage{algpseudocode}
\algrenewcommand\algorithmiccomment[1]{\hfill{\color{gray}$\triangleright$~#1}}
\algrenewcommand\algorithmicindent{1em}
\usepackage{tikz}
\usetikzlibrary{arrows.meta,positioning,fit,backgrounds,calc,patterns}
\usepackage{tikz}
\usetikzlibrary{matrix}
\usepackage{subcaption}
\usepackage[capitalize,nameinlink,noabbrev]{cleveref}
\usepackage{array}
\usepackage{colortbl}

\crefformat{equation}{#2(#1)#3}
\crefmultiformat{equation}{{}#2(#1)#3}{ and~#2(#1)#3}{, #2(#1)#3}{ and~#2(#1)#3}
\Crefformat{equation}{#2(#1)#3}
\Crefmultiformat{equation}{{}#2(#1)#3}{ and~#2(#1)#3}{, #2(#1)#3}{ and~#2(#1)#3}

\crefname{assumption}{Assumption}{Assumptions}
\crefname{problem}{Problem}{Problems}
\crefname{section}{Section}{Sections}
\crefname{subsection}{Subsection}{Subsections}

\usepackage{xcolor}
\definecolor{c0}{HTML}{641a80}
\definecolor{c1}{HTML}{b73779}
\definecolor{c2}{HTML}{1f77b4}
\definecolor{c3}{HTML}{ff7f0e}
\definecolor{c4}{HTML}{2ca02c}

\DeclareMathOperator{\tr}{tr}
\DeclareMathOperator{\Var}{Var}
\DeclareMathOperator{\rank}{rank}
\DeclareMathOperator{\diag}{diag}
\DeclareMathOperator{\range}{range}
\DeclareMathOperator*{\argmin}{arg\,min}
\newcommand{\blockdiag}{\operatorname{blockdiag}}
\newcommand{\blockblr}{\operatorname{blockBLR}^2}
\newcommand{\butterfly}{\operatorname{BF}}
\newcommand{\SSS}{\operatorname{SSS}}
\newcommand{\BLR}{\operatorname{BLR}}
\newcommand{\bsep}{\operatorname{BF-out}}
\newcommand{\butterflymid}{\operatorname{BF-mid}}
\newcommand{\T}{\mathsf{T}}
\newcommand{\F}{\mathsf{F}}
\newcommand{\R}{\mathbb{R}}
\newcommand{\C}{\mathbb{C}}
\newcommand{\bO}{\mathcal{O}}
\newcommand{\mA}{\bm{A}}
\newcommand{\mB}{\bm{B}}
\newcommand{\mU}{\bm{U}}
\newcommand{\mLambda}{\bm{Lambda}}
\newcommand{\mI}{\bm{I}}
\newcommand{\leftfold}{\operatorname{leftfold}}
\newcommand{\side}{\operatorname{side}}
\newcommand{\memory}{\operatorname{Memory}}
\newcommand{\gaussian}{\operatorname{Gaussian}}
\newcommand{\complexity}{\operatorname{Complexity}}
\newcommand{\halving}{\operatorname{halving}}
\newcommand{\domain}{\operatorname{domain}}
\newcommand{\OPT}{\mathrm{OPT}}

\usepackage[labelfont=bf]{caption}

\usepackage[shortlabels]{enumitem}

\usepackage[dvipsnames]{xcolor}
\newcommand{\David}[1]{\textcolor{blue}{\bf David: #1}}
\newcommand{\Tyler}[1]{\textcolor{ForestGreen}{\bf Tyler: #1}}
\newcommand{\diana}[1]{\textcolor{purple}{\bf Diana: #1}}
\newcommand{\paul}[1]{\textcolor{teal}{\bf Paul: #1}}
\newcommand{\chris}[1]{\textcolor{blue}{\bf Chris: #1}}
\newcommand{\TODO}[1]{\textcolor{red}{\bf TODO: #1}}

\newcommand\numberthis{\addtocounter{equation}{1}\tag{\theequation}}

\ifpdf
  \DeclareGraphicsExtensions{.eps,.pdf,.png,.jpg}
\else
  \DeclareGraphicsExtensions{.eps}
\fi

\newcommand{\creflastconjunction}{, and~}

\newtheorem{theorem}{Theorem}[section]
\newtheorem*{theorem*}{Theorem}
\newtheorem{lemma}[theorem]{Lemma}

\newtheorem{corollary}[theorem]{Corollary}
\theoremstyle{definition}
\newtheorem{remark}[theorem]{Remark}
\newtheorem{definition}[theorem]{Definition}

%% file: shared.tex
\usepackage{lipsum}
\usepackage{stmaryrd}
\usepackage{amsfonts}
\usepackage{graphicx}
\usepackage{epstopdf}
\usepackage{amsmath}
\usepackage{amsopn}
\usepackage{amsthm}
\usepackage{hyperref}
\hypersetup{
    colorlinks,
    linkcolor=c0,
    citecolor=c1,
    urlcolor=c0
}
\usepackage{bm}
\usepackage{amssymb}
\usepackage{xcolor}
\usepackage{overpic}
\usepackage{algorithm}
\usepackage{algpseudocode}
\algrenewcommand\algorithmiccomment[1]{\hfill{\color{gray}$\triangleright$~#1}}
\algrenewcommand\algorithmicindent{1em}
\usepackage{tikz}
\usetikzlibrary{arrows.meta,positioning,fit,backgrounds,calc,patterns}
\usepackage{tikz}
\usetikzlibrary{matrix}
\usepackage{subcaption}
\usepackage[capitalize,nameinlink,noabbrev]{cleveref}

\crefformat{equation}{#2(#1)#3}
\crefmultiformat{equation}{{}#2(#1)#3}{ and~#2(#1)#3}{, #2(#1)#3}{ and~#2(#1)#3}
\Crefformat{equation}{#2(#1)#3}
\Crefmultiformat{equation}{{}#2(#1)#3}{ and~#2(#1)#3}{, #2(#1)#3}{ and~#2(#1)#3}

\crefname{assumption}{Assumption}{Assumptions}
\crefname{problem}{Problem}{Problems}
\crefname{section}{Section}{Sections}
\crefname{subsection}{Subsection}{Subsections}

\usepackage{xcolor}
\definecolor{c0}{HTML}{641a80}
\definecolor{c1}{HTML}{b73779}
\definecolor{c2}{HTML}{1f77b4}
\definecolor{c3}{HTML}{ff7f0e}
\definecolor{c4}{HTML}{2ca02c}

\DeclareMathOperator{\tr}{tr}
\DeclareMathOperator{\Var}{Var}
\DeclareMathOperator{\rank}{rank}
\DeclareMathOperator{\diag}{diag}
\DeclareMathOperator{\range}{range}
\DeclareMathOperator*{\argmin}{arg\,min}
\newcommand{\blockdiag}{\operatorname{blockdiag}}
\newcommand{\blockblr}{\operatorname{blockBLR}^2}
\newcommand{\butterfly}{\operatorname{BF}}
\newcommand{\SSS}{\operatorname{SSS}}
\newcommand{\BLR}{\operatorname{BLR}}
\newcommand{\bsep}{\operatorname{BF-out}}
\newcommand{\butterflymid}{\operatorname{BF-mid}}
\newcommand{\T}{\mathsf{T}}
\newcommand{\F}{\mathsf{F}}
\newcommand{\R}{\mathbb{R}}
\newcommand{\C}{\mathbb{C}}
\newcommand{\bO}{\mathcal{O}}
\newcommand{\mA}{\bm{A}}
\newcommand{\mB}{\bm{B}}
\newcommand{\mU}{\bm{U}}
\newcommand{\mLambda}{\bm{Lambda}}
\newcommand{\mI}{\bm{I}}
\newcommand{\leftfold}{\operatorname{leftfold}}
\newcommand{\side}{\operatorname{side}}
\newcommand{\memory}{\operatorname{Memory}}
\newcommand{\gaussian}{\operatorname{Gaussian}}
\newcommand{\complexity}{\operatorname{Complexity}}
\newcommand{\halving}{\operatorname{halving}}
\newcommand{\domain}{\operatorname{domain}}
\newcommand{\OPT}{\mathrm{OPT}}

\usepackage[labelfont=bf]{caption}

\usepackage[shortlabels]{enumitem}

\usepackage{lineno}
\linenumbers

\usepackage[dvipsnames]{xcolor}
\newcommand{\David}[1]{\textcolor{blue}{\bf David: #1}}
\newcommand{\Tyler}[1]{\textcolor{ForestGreen}{\bf Tyler: #1}}
\newcommand{\diana}[1]{\textcolor{purple}{\bf Diana: #1}}
\newcommand{\paul}[1]{\textcolor{teal}{\bf Paul: #1}}
\newcommand{\chris}[1]{\textcolor{blue}{\bf Chris: #1}}
\newcommand{\TODO}[1]{\textcolor{red}{\bf TODO: #1}}

\newcommand\numberthis{\addtocounter{equation}{1}\tag{\theequation}}

\ifpdf
  \DeclareGraphicsExtensions{.eps,.pdf,.png,.jpg}
\else
  \DeclareGraphicsExtensions{.eps}
\fi

\newcommand{\creflastconjunction}{, and~}


\newtheorem{definition}{Definition}
\newtheorem{theorem}{Theorem}
\newtheorem{lemma}[theorem]{Lemma}
\newtheorem{corollary}[theorem]{Corollary}
\newdefinition{remark}{Remark}
\newproof{pf}{Proof}
\newproof{pot}{Proof of Theorem \ref{thm}}

%% file: intro.tex
\section{Introduction}
Butterfly matrices are a structured matrix format arising throughout scientific computing. They generalize the underlying algebraic structure of the Fast Fourier Transform and are widely used to approximate oscillatory operators, including Fourier integral operators \cite{candes2009fast}, high-frequency acoustic and electromagnetic scattering operators \cite{michielssen1996multilevel,Guo2017,LiuXingGuo:2021}, and Hankel, spherical harmonic, and other special function transforms \cite{oneil2010algorithm,tygert2010fast,seljebotn2012wavemoth,beckman2026butterfly}.
All of these operators are dense, but have so-called \emph{complementary low-rank structure}, meaning that certain submatrices determined by a hierarchical partition of the row and column indices are approximately low-rank \cite{li2015butterfly}. A butterfly approximation takes advantage of that structure to obtain a more compact and efficient representation. In particular, when the rank parameter is fixed, an $N\times N$ butterfly matrix can be represented as a product of $O(\log(N))$ sparse factors, each with $O(N)$ nonzero entries. As such, it takes just $O(N\log(N))$ storage and can be rapidly applied to a vector in only $O(N \log(N))$ operations.

\subsection{Butterfly Approximation}
In some applications, problem-specific knowledge can be used to construct an accurate butterfly approximation for a given target matrix. In other settings, however, such information is unavailable, incomplete, or too costly to exploit. One instead seeks to find a butterfly approximation for a given matrix $\bm{A}$ using only black-box access to its entries or its action on vectors. This agnostic approximation problem arises throughout scientific computing, and in recent applications of butterfly approximation in other areas, like neural network compression and fine-tuning~\cite{dao2019learning,liu2024parameter}. 
Entry access is appropriate when $\bm{A}$ is stored as a dense matrix, or when its entries can be calculated efficiently using some formula~\cite{oneil2010algorithm,tygert2010fast,seljebotn2012wavemoth,li2015butterfly,beckman2026butterfly}.
Matrix-vector product access, where the algorithm interacts with $\bm{A}$ via black-box ``matvec queries'' of the form $\bm{x}\mapsto \bm{A}\bm{x}$ and $\bm{x} \mapsto \bm{A}^\T\bm{x}$, is natural when the action of $\bm{A}$ and $\bm{A}^\T$ can be quickly computed using an existing analysis-based fast algorithm, or when the matrix is already represented in a structured format~\cite{li2015butterfly,Guo2017,LiuXingGuo:2021}.

No matter the access model, the central goal is to develop algorithms for constructing a butterfly approximation for a matrix $\bm{A}$ that do not rely on the matrix's origin or structure. Concretely, in this work, we are interested in bounding the error of the approximation returned by such an algorithm \emph{relative to the best butterfly approximation of the matrix}. 

Despite a long line of work on butterfly approximation, until very recently, existing methods failed to provide such a strong theoretical guarantee \cite{oneil2010algorithm,li2015butterfly,Pang2020,LiuXingGuo:2021,zheng2023efficient}. This limitation was resolved in a breakthrough result of Le, Zheng, Riccietti, and Gribonval \cite{Le2025}, which studies a broader class of ``deformable'' butterfly matrices. That work provides the first efficient algorithm that, given entry access to a matrix $\bm{A}$, returns a butterfly approximation whose error is within a multiplicative factor of the best possible -- specifically, within $O(\sqrt{\log(N)})$ of optimal.

\subsection{Our Contributions}
Our work builds on the progress in \cite{Le2025} by developing and analyzing an alternative class of butterfly approximation algorithms that also enjoy  strong multiplicative error guarantees. We do so by revisiting the butterfly approximation problem from a new perspective, presenting a recursive characterization of butterfly matrices in \Cref{section:butterfly_def}. While the recursive nature of butterfly matrices has been noted in prior literature \cite{li2017interpolative}, our treatment goes a step further. We show that butterfly matrices can be stored as recursive data structures, rather than the standard product of sparse factors. While equivalent in representative power, our recursive presentation significantly simplifies the description and theoretical analysis of butterfly algorithms. In particular, we leverage it to obtain two algorithmic results:
\begin{itemize}
    \item In \cref{sec:greedy_meta}, we describe a simple recursive algorithm that accesses entries of $\bm{A}$ directly and uses $O(N^2)$ computation to return a butterfly approximation $\bm{B}$ to $\bm{A}$ whose error $\|\bm{A} - \bm{B}\|_\F$ is at most $O(\sqrt{\log(N)})$ times worse than the best possible butterfly approximation to $\bm{A}$ (see \cref{theorem:meta_butterfly}). 
    This matches the runtime and accuracy guarantee of the algorithm described in \cite{Le2025}.
    \item In \cref{section:matvec}, we describe an algorithm
    that accesses $\bm{A}$ via $\widetilde{O}(\sqrt{N})$ matvec queries, uses $\widetilde{O}(N^{3/2})$ additional computation, and $\widetilde{O}(N)$ working memory, and returns a butterfly approximation whose error is at most $O(N^{1/4})$ times the best possible (see \cref{theorem:matvecmain}). The $\widetilde{O}(\cdot)$ notation suppresses polynomial factors in $L = O(\log(N))$. The precise hidden factors, including their dependence on the failure probability, are controlled by an oversampling parameter and are made explicit in \Cref{section:matvec}. Our algorithm can be viewed as a reformulation of the popular algorithm developed in \cite{Guo2017,LiuXingGuo:2021}.\footnote{We remark that any entry-access algorithm for computing quasi-optimal butterfly approximation can be converted into an algorithm using $O(\sqrt{N})$ matvecs. The resulting reduction, however, requires $O(N^{3/2})$ working memory; see \Cref{remark:reduction}.}
\end{itemize}
Beyond these main results, we provide an open source Julia package\footnote{\label{fn:software}\tt\url{https://github.com/pbeckman/RecursiveButterfly.jl}} which includes implementations of both algorithms, as well as rank-adaptive generalizations to rectangular matrices $\bm{A} \in \C^{M \times N}$. This package demonstrates the simplicity of working with butterfly matrices in our recursive format. Further details and numerical demonstrations are provided in \cref{sec:experiments}.

\subsection{Past work}
The butterfly factorization and its underlying complementary low-rank property were formalized in \cite{oneil2010algorithm,li2015butterfly}, building on earlier butterfly algorithms for oscillatory operators \cite{michielssen1996multilevel,candes2009fast}. \cite{li2015butterfly} introduced a randomized algorithm for butterfly matrix approximation based on matrix-vector products that uses $O(\sqrt{N})$ matvecs and $O(N^{3/2})$ working memory. Subsequently, \cite{Guo2017} introduced an algorithm that also uses $O(\sqrt{N})$ matvecs, but requires only $O(N\log(N))$ working memory.
In~\cite{LiuXingGuo:2021} an adaptive, parallel implementation is described. As noted above, and expanded on in \cref{sec:mapbetweenformats}, our memory-efficient matvec algorithm is essentially equivalent to the algorithm of \cite{LiuXingGuo:2021}. However, while that work provides some theoretical analysis that relates the accuracy of the butterfly approximation produced to the error of  low-rank approximations used within the algorithm, it falls short of providing a strong guarantee comparing the overall approximation error to that of the \emph{best possible} butterfly approximation. Our work fills this gap, providing the first theoretical quasi-optimality guarantee for the algorithm of \cite{LiuXingGuo:2021}.

Also related to our work is \cite{Pang2020}, which describes a random access butterfly algorithm based on the interpolative decomposition.
This algorithm runs in $O(N\log (N))$ time, but cannot be guaranteed to obtain a good approximation in the worst-case, even when the input $\bm{A}$ is itself a butterfly matrix. Notably however, like our work, that algorithm exploits the recursive nature of butterfly matrices.

Our work also builds on the recent paper \cite{amsel2025quasi} which studies algorithms for hierarchical semi-separable (HSS) matrices.
In particular, we make use of bounds for a class of matrices, called $\BLR^2$ in \cite{amsel2025quasi}.
These matrices are a superset of butterfly (and HSS) matrices; see \cref{remark:blr2}. 
More broadly, our work is part of an ongoing line of research which provides similar approximation guarantees for structured matrix approximation algorithms, particularly in the matvec query model.
In this context, the most widely studied family of matrices are low-rank matrices \cite{HalkoMartinssonTropp:2011, BakshiClarksonWoodruff:2022,BakshiNarayanan:2023,SimchowitzElAlaouiRecht:2018}.
Recent work has also considered approximation algorithms for classes such as sparse matrices \cite{CurtisPowellReid:1974,ColemanMore:1983, ColemanCai:1986,WimalajeewaEldarVarshney:2013,DasarathyShahBhaskarNowak:2015,BekasKokiopoulouSaad:2007,TangSaad:2011,BastonNakatsukasa:2022,DharangutteMusco:2023,park_nakatsukasa_23,AmselChenDumanKelesHalikiasMuscoMusco:2026,musco2026nearlyinstanceoptimalsparse}, 
hierarchical matrices \cite{ChenDumanKelesHalikiasMuscoMuscoPersson:2025,amsel2025quasi}, and more \cite{WatersSankaranarayananBaraniuk:2011,SchaferKatzfussOwhadi:2021,LiuXingGuo:2021, kapralov2023toeplitz,Amsel2026Query}.

%% file: prelims.tex
\section{Background and notation}
This section introduces the necessary notation and provides definitions of the matrix classes used in this work.

\subsection{Notation}
For $m,n \in \mathbb{N}$ with $m \geq n$, we define $[n;m] := \{n,n+1,\ldots,m-1,m\}$. If $n = 1$, we write $[1;m] = [m]$. For two index sets $I$ and $J$ and a matrix $\bm{M}\in \mathbb{R}^{m \times n}$, we denote by $\bm{M}(I,J)$ the submatrix of $\bm{M}$ corresponding to the row and column indices in $I$ and $J$. We define $\bm{M}(:,J) := \bm{M}([m],J)$ and $\bm{M}(I,:) = \bm{M}(I,[n])$. 
The Frobenius norm is denoted by $\| \bm M \|_\F$. 
We denote the best rank-$k$ approximation to a matrix $\bm M$ in any unitarily invariant norm as $\llbracket\bm{M}\rrbracket_k$.
This approximation can be obtained by the truncated SVD as described by the Eckart--Young Theorem.  
The optimal rank-$k$ approximation error in the Frobenius norm is denoted by $\mathcal{E}_k(\bm{B}) := \|\bm{B} - \llbracket\bm{B}\rrbracket_k\|_\F$.
We write the transpose of $\bm{M}$  as $\bm M ^\T $ and the Moore-Penrose pseudoinverse as $\bm{M}^{\dagger}$. 
Throughout, $\gaussian(m, n)$ denotes the distribution of $m \times n$ matrices whose entries are independent standard normal random variables.
We denote by $\bm{I}_{n}$ the $n \times n$ identity matrix. We denote by $\blockdiag(\bm{B}_1,\ldots,\bm{B}_m)$ the block-diagonal matrix with diagonal blocks $\bm{B}_1,\ldots,\bm{B}_m$. For a block-diagonal matrix $\bm{B} \in \mathbb{R}^{m \times n}$, where $m$ and $n$ are even, we denote by $\bm{B}_{[1]}, \bm{B}_{[2]} \in \mathbb{R}^{m/2 \times n/2}$ the upper and lower $m/2 \times n/2$ diagonal blocks of $\bm{B}$, respectively. In other words, $\bm{B} = \blockdiag(\bm{B}_{[1]},\bm{B}_{[2]})$.

\subsection{Binary partition trees}
To define butterfly matrices, we define a collection of subsets of $[2^{L+1}k]$ to refer to certain submatrices.
\begin{definition}
\label[definition]{def:index_sets}
    Let $L,k\in\mathbb{N}$ be fixed.
    For any $\ell=0,\ldots, L$ and $i=1, \ldots, 2^\ell$, define the index set
    \begin{equation*}
        I_{i,\ell} := [(i-1)2^{L+1-\ell}k + 1; i2^{L+1-\ell}k].
    \end{equation*}
\end{definition}
Observe that 
\[
I_{i,\ell-1} = I_{2i-1,\ell} \cup I_{2i,\ell},
\]
and hence these index sets correspond to a tree where the node $I_{i,\ell-1}$ has children $I_{2i-1,\ell}$ and $I_{2i,\ell}$. 
We use $\mathcal{T} = \{I_{i,\ell}\}$ to denote the set of these nodes. 
This tree is visualized in \Cref{fig:tree}.

\begin{figure}
    \centering
    \input{tree}
    \caption{A perfect dyadic partition tree corresponding to the index sets defined in \cref{def:index_sets} for $L = 3$.}
    \label{fig:tree}
\end{figure}
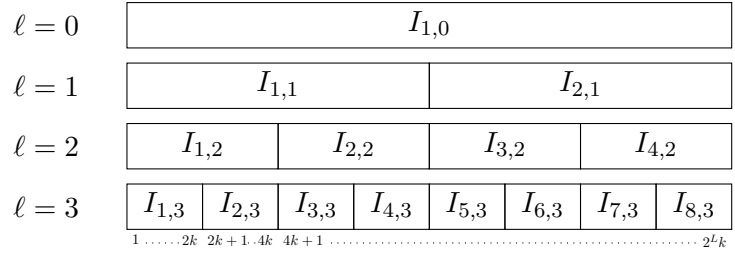

\subsection{Butterfly matrices}
\label{section:butterfly_def}
We define butterfly matrices using the so-called \emph{complementary low-rank property}; see for example \cite{li2015butterfly}.
\begin{definition}[Butterfly matrix]\label[definition]{def:butterfly}
    Let $L, k\in \mathbb{N}$ be fixed, and assume that $L$ is even. 
    Define the index sets $I_{i,\ell}$ as in \cref{def:index_sets}.
    A matrix $\bm{B} \in \mathbb{R}^{2^{L+1}k  \times 2^{L+1}k}$ is said to be a level-$L$ butterfly matrix with butterfly rank $k$ if for all $\ell = 0,\ldots,L$, $i = 1,\ldots,2^{\ell}$, and $j = 1,\ldots,2^{L-\ell}$, we have 
    \[
        \rank(\bm{B}(I_{i,\ell},I_{j,L-\ell})) \leq k.
    \]    
    The set of such matrices is denoted as $\butterfly(L,k)$. The pair of index sets $(I_{i,\ell},I_{j,L-\ell})$ are called \emph{complementary pairs} and the submatrix $\bm{B}(I_{i,\ell},I_{j,L-\ell})$ is called a \emph{complementary low-rank block}. Note that in the base case of $L = 0$, where $\bm {B} \in \R^{2k \times 2k}$, $\bm {B}$ is a butterfly matrix if and only if it is rank at most $k$. 
\end{definition}

We note that we assume the blocks are exactly low-rank, rather than numerically as is done in~\cite{li2015butterfly}. 
This is necessary for a true equivalence between the complementary low-rank property and other definitions of butterfly matrices, including the sparse factorization often used in practice. 

We illustrate the complementary low-rank structure in \cref{fig:CLR}. 
While \cref{def:butterfly} assumes that the binary partition tree is a perfect dyadic tree with leaf-nodes of cardinality $2k$, the definition can be trivially extended to butterfly matrices with respect to general binary partition trees.

\begin{remark}[Why $N=2^{L+1}k$?]
    For notational simplicity, we take the matrix dimension to be $N = 2^{L+1} k$ for some fixed $L$ and $k$ and use a perfect dyadic partition tree, where the leaf nodes contain $2k$ contiguous indices.
    The algorithms and analysis developed in this paper extend to  more general settings, which we describe in \cref{sec:experiments}.
\end{remark}

\begin{figure}
    \centering
    \input{imgs/CLR}
    \caption{Illustration of the submatrices $\bm{B}(I_{i,\ell},I_{j,L-\ell})$ from \cref{def:butterfly}.
    The complementary low-rank property asserts these submatrices are all of rank at most $k$.
    Here $L = 4$, and the boxes in the light grey grid are of size $2k\times 2k$.}
    \label{fig:CLR}
\end{figure}
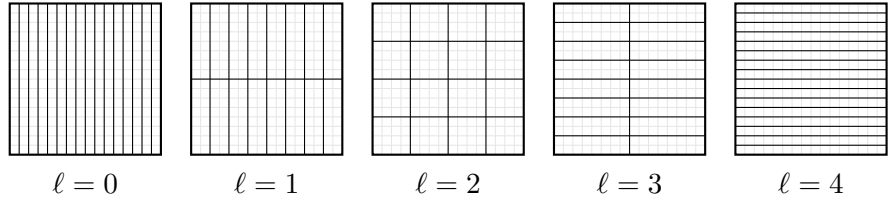

\subsection{A recursive characterization of butterfly matrices}

In this paper, we use an equivalent recursive definition of butterfly matrices.
While the recursive structure of butterfly matrices is implicit in the literature \cite{li2015butterfly,Pang2020}, we are unaware of any past work that makes the following observation explicit.

\begin{theorem}[Recursive factorization of butterfly matrices]\label{theorem:recursive}
    Let $L, k\in \mathbb{N}$ be fixed, and assume that $L$ is even. Consider a matrix $\bm{B} \in \mathbb{R}^{2^{L+1}k \times 2^{L+1}k}$. Then, $\bm{B} \in \butterfly(L,k)$ if and only if
    \begin{enumerate}
        \item when $L = 0$,  $\rank(\bm{B}) \leq k$, and can thus be written as $\bm{B} = \bm{U} \bm{X} \bm{V}^\T$, where $\bm{U}, \bm{V} \in \mathbb{R}^{2k \times k}$ and $\bm{X} \in \mathbb{R}^{k \times k}$,
        \item when $L > 0$, 
        \begin{equation*}
            \bm{B} = \bm{U} \bm{X} \bm{V}^\T,
        \end{equation*}
        where 
        \begin{align*}
                \bm{U} &= \blockdiag(\bm{U}_1,\ldots,\bm{U}_{2^L}), \quad \bm{U}_i \in \mathbb{R}^{2k \times k} \text{ has orthonormal columns},\\
            \bm{V} &= \blockdiag(\bm{V}_1,\ldots,\bm{V}_{2^L}), \quad \bm{V}_i \in \mathbb{R}^{2k \times k} \text{ has orthonormal columns},
            \end{align*}
            and $\bm{X} \in \mathbb{R}^{2^{L}k \times 2^L k}$ can be partitioned as
            \begin{equation*}
                \bm{X} = \begin{bmatrix} \bm{X}_{1,1} & \bm{X}_{1,2} \\ 
                \bm{X}_{2,1} & \bm{X}_{2,2} \end{bmatrix},
            \end{equation*}
            where for $s,t = 1,2$, $\bm{X}_{s,t} \in \butterfly(L-2,k)$.
        \end{enumerate}
    In particular, since this result can be applied at every level $L$, it results in a recursive representation of butterfly matrices.
\end{theorem}

The basic intuition for the proof of \cref{theorem:recursive} is easily understood and  explained visually in \cref{fig:recursive_def}. 
The complete proof of \cref{theorem:recursive} is somewhat tedious, so we defer it to \cref{section:recursive}.
One of the main advantages of this formulation is that it makes the nested structure explicit. A butterfly matrix can be represented as a recursive data structure, which significantly simplifies the implementation of algorithms.

\begin{figure}
    \centering
    \input{imgs/recursive_def}
    \caption{Let $\bm{B} \in \butterfly(L,k)$. The complementary low-rank property (\Cref{def:butterfly}) at the extreme levels ($\ell = 0$ and $\ell = L$) implies $\bm{B} \in \bsep(L,k)$ (see \Cref{def:bsep}), and hence gives a factorization $\bm{B} = \bm{U}\bm{X}\bm{V}^\T$, where $\bm{U}$ and $\bm{V}$ are block-diagonal matrices with orthonormal columns.
    \Cref{theorem:recursive} asserts that $\bm{X}$ is a $2\times 2$ block-matrix, where each block is in $\butterfly(L-2,k)$.
    That is, each block in the $2 \times 2$ structure of $\bm{X}$ satisfies the complementary low-rank property. 
    Choose any complementary low-rank block of one of the blocks of $\bm{X}$.
    Examples of some submatrices that must be low-rank are highlighted.
    Observe that, due to the block-diagonal structure of $\bm{U}$ and $\bm{V}$, submatrices in $\bm{X}$ are mapped to submatrices in $\bm{B}$, which by the complementary low-rank property on $\bm{B}$, have rank at most $k$.
    Moreover, denoting the colored submatrices of $\bm{B}$ and $\bm{X}$ respectively by $\widehat{\bm{B}}$ and $\widehat{\bm{X}}$, we see that, $\widehat{\bm{B}} = \widehat{\bm{U}}\widehat{\bm{X}}\widehat{\bm{V}}^\T$, where $\widehat{\bm{U}}$ and $\widehat{\bm{V}}$ have orthonormal columns.
    This implies that $\rank(\widehat{\bm{X}}) = \rank(\widehat{\bm{B}}) \leq k$, thus giving that the colored submatrices of $\bm{X}$ have rank at most $k$. }
    \label{fig:recursive_def}
\end{figure}
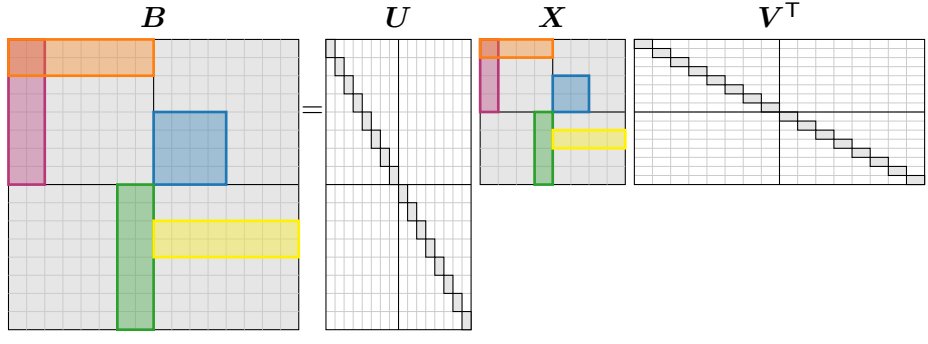

\begin{remark}[Relation to sparse factorization]
    Butterfly matrices can also be defined in terms of an equivalent ``product of sparse factors'' characterization; see e.g. \cite{li2015butterfly}. 
    This  definition naturally corresponds to an efficient representation of butterfly matrices that is used in many existing algorithms for butterfly matrices.
    The recursive definition described above simplifies some of the book-keeping necessary for the sparse representation.
    We discuss the conversion between the recursive characterization from \cref{theorem:recursive} and the product of sparse factors characterization in \cref{sec:mapbetweenformats}.
\end{remark}

\begin{remark}[Memory and complexity requirements]
    Storing a matrix using the recursive format outlined in \Cref{theorem:recursive} requires $O(Nk\log(N/k))$ units of memory. Furthermore, one can apply a vector to a matrix in $\butterfly(L,k)$ using a recursive algorithm requiring $O(Nk\log(N/k))$ operations. 
    This coincides with complexities of butterfly matrices stored using a product of sparse factors. 
\end{remark} 

Our algorithm for butterfly approximation builds on an algorithm for matrix approximation from a class of matrices, which we denote by $\bsep(L,k)$. This class consists of all matrices satisfying the complementary low-rank property at the extreme levels $\ell = 0$ and $\ell = L$, where the blocks in the partition are either as long and skinny as possible, or as short and wide as possible (see, e.g.,  $\ell = 0$ and $\ell = 4$ in \Cref{fig:CLR}). Our algorithm will recursively approximate a given matrix using matrices from $\bsep(L,k)$.
\begin{definition}[$\bsep$]\label[definition]{def:bsep}
Let $L, k\in \mathbb{N}$ be fixed and assume that $L$ is even. Consider a matrix $\bm{B} \in \mathbb{R}^{2^{L+1}k \times 2^{L+1}k}$. We say $\bm{B} \in \bsep(L,k)$ if $\bm{B} = \bm{U} \bm{X} \bm{V}^\T$
where 
        \begin{align*}
                \bm{U} &= \blockdiag(\bm{U}_1,\ldots,\bm{U}_{2^L}), \quad \bm{U}_i \in \mathbb{R}^{2k \times k} \text{ has orthonormal columns},\\
            \bm{V} &= \blockdiag(\bm{V}_1,\ldots,\bm{V}_{2^L}), \quad \bm{V}_i \in \mathbb{R}^{2k \times k} \text{ has orthonormal columns},
            \end{align*}
            and $\bm{X} \in \mathbb{R}^{2^{L}k \times 2^L k}$. Equivalently, using the notation from \Cref{def:index_sets}, $\bm{B} \in \bsep(L,k)$ if and only if $\rank(\bm{B}(I_{i,L},I_{1,0})), \rank(\bm{B}(I_{1,0},I_{i,L})) \leq k$ for all $i = 1,\ldots,2^{L}$.
\end{definition}

\begin{remark}
\label[remark]{remark:blr2}
    The set of $\bsep$ matrices is a special case of the set of $\BLR^2$ matrices, which appear in the context of numerical solutions of integral equations \cite{ChandrasekaranDewildeGuPalsvanderVeen:2002,AshcraftButtariMary:2021,gillman2012direct,martinsson2005fast}. In particular, using the notation of \cite[Appendix C]{amsel2025quasi}, $\bsep(L,k) = \BLR^2(2^L,2k,k,\varnothing)$. Furthermore, note that $\butterfly(L,k) \subseteq \bsep(L,k)$.
\end{remark}

%% file: tree.tex
\begin{tikzpicture}

\def\L{3}          
\def\k{2}          
\def\boxheight{.6}
\def\rowsep{.8}

\pgfmathtruncatemacro{\N}{2^\L}

\foreach \l in {0,...,\L} {
    \pgfmathtruncatemacro{\numintervals}{2^\l}
    \pgfmathtruncatemacro{\width}{2^(\L-\l)}
    \pgfmathsetmacro{\y}{(\L-\l)*\rowsep}

    \foreach \i in {1,...,\numintervals} {

        \pgfmathsetmacro{\xl}{(\i-1)*\width}
        \pgfmathsetmacro{\xr}{\i*\width}

        \draw (\xl,\y) rectangle (\xr,\y+\boxheight);

        \node at ({(\xl+\xr)/2},{\y+0.5*\boxheight})
            {$I_{\i,\l}$};
    }

    \node[left] at (-.5,\y+0.5*\boxheight)
        {$\ell=\l$};
}

\node[below,anchor=north west,scale=.5] at (0,0) {$1$};
\node[below,anchor=north east,scale=.5] at (1,0) {$2k$};

\node[below,anchor=north west,scale=.5] at (1,0) {$2k+1$};
\node[below,anchor=north east,scale=.5] at (2,0) {$4k$};

\node[below,anchor=north west,scale=.5] at (2,0) {$4k+1$};

\node[below,anchor=north east,scale=.5] at ({2^\L},0) {$2^Lk$};

\draw[dotted] (.25,-.18) -- (.75,-.18);
\draw[dotted] (1.6,-.18) -- (1.7,-.18);
\draw[dotted] (2.7,-.18) -- ({2^\L-.45},-.18);

\end{tikzpicture}

%% file: imgs/CLR.tex
\begin{tikzpicture}
  \def\w{2}
  \def\h{2}
  \def\gap{0.4}
  \def\nboxes{5}
  \def\startcols{16}
  \def\startrows{1}
  \foreach \b in {0,...,\numexpr\nboxes-1\relax}{
    \pgfmathsetmacro{\xoff}{\b*(\w+\gap)}
    \pgfmathtruncatemacro{\ncols}{\startcols/2^\b}
    \pgfmathtruncatemacro{\nrows}{\startrows*2^\b}
    \foreach \i in {1,...,15}{
      \draw[line width=.2pt,black!10] (\xoff+\i*\w/16, 0) -- (\xoff+\i*\w/16, \h);
      \draw[line width=.2pt,black!10] (\xoff, \i*\h/16) -- (\xoff+\w, \i*\h/16);
    }
    \draw[thick] (\xoff,0) rectangle (\xoff+\w,\h);
    \ifnum\ncols>1
      \foreach \i in {1,...,\numexpr\ncols-1\relax}{
        \draw (\xoff+\i*\w/\ncols, 0) -- (\xoff+\i*\w/\ncols, \h);
      }
    \fi
    \ifnum\nrows>1
      \foreach \j in {1,...,\numexpr\nrows-1\relax}{
        \draw (\xoff, \j*\h/\nrows) -- (\xoff+\w, \j*\h/\nrows);
      }
    \fi
    \node[below] at (\xoff+\w/2, -0.1) {$\ell = \b$};
4   }
\end{tikzpicture}

%% file: imgs/recursive_def.tex
\begin{tikzpicture}[scale=.12]
\def\K{16}
\def\Km{15}
\draw[fill=black!10] (0,0) rectangle ({2*\K},{2*\K});
\foreach \x in {1,...,\Km}{
    \draw[line width=.2px,black!20] (0,{2*\x}) -- ({2*\K},{2*\x});
    \draw[line width=.2px,black!20] ({2*\x},0) -- ({2*\x},{2*\K});
}
\draw[] (0,\K) -- ({2*\K},\K);
\draw[] (\K,0) -- (\K,2*\K);
\node[above] at (\K, 2*\K+.2) {$\bm{B}$};
\node[] at ({2*\K+1.5},{\K+\K/2}) [] {=};
\draw[color=c1,line width=1,fill=c1,fill opacity=.35] (0,\K)  rectangle (4,{2*\K});
\draw[color=c2,line width=1,fill=c2,fill opacity=.35] (16,\K)  rectangle (24,{2*\K-8});
\draw[color=c3,line width=1,fill=c3,fill opacity=.35] (0,2*\K-4)  rectangle (16,{2*\K});
\draw[color=c4,line width=1,fill=c4,fill opacity=.35] (12,\K-16)  rectangle (16,{\K});
\draw[color=yellow,line width=1,fill=yellow,fill opacity=.35] (16,8)  rectangle (32,12);

\begin{scope}[shift = {(2*\K+3,0)}]
\foreach \x in {1,...,\Km}{
    \draw[line width=.2px,black!20] (0,{2*\x}) -- (\K,{2*\x});
    \draw[line width=.2px,black!20] (\x,0) -- (\x,{2*\K});
}
\draw[] (0,0) rectangle (\K,{2*\K});
\foreach \x in {1,...,\K}{
    \draw[fill=black!10] (\x-1,{2*\K-2*(\x-1)}) rectangle (\x,{2*\K-2*\x});
}
\draw[] (0,\K) -- ({\K},\K);
\draw[] (\K/2,0) -- (\K/2,2*\K);
\node[above] at (\K/2, 2*\K+.2) {$\bm{U}$};
\begin{scope}[shift = {(\K+1,\K)}]
\draw[fill=black!10] (0,0) rectangle (\K,\K);
\foreach \x in {2,4,...,\Km}{
    \draw[line width=.2px,black!20] (0,\x) -- (\K,\x);
    \draw[line width=.2px,black!20] (\x,0) -- (\x,\K);
}
\draw[] (0,\K/2) -- ({\K},\K/2);
\draw[] (\K/2,0) -- (\K/2,\K);
\draw[color=c1,line width=1,fill=c1,fill opacity=.35] (0,\K/2)  rectangle (2,{\K});
\draw[color=c2,line width=1,fill=c2,fill opacity=.35] (8,\K/2)  rectangle (12,\K-4);
\draw[color=c3,line width=1,fill=c3,fill opacity=.35] (0,\K-2)  rectangle (8,\K);
\draw[color=c4,line width=1,fill=c4,fill opacity=.35] (6,\K-16)  rectangle (8,{\K-8});
\draw[color=yellow,line width=1,fill=yellow,fill opacity=.35] (8,4)  rectangle (16,6);
\node[above] at (\K/2, \K+.2) {$\bm{X}$};
\end{scope}
\begin{scope}[shift = {(2*\K+2,\K)}]
\foreach \x in {1,...,\Km}{
    \draw[line width=.2px,black!20] (0,\x) -- ({2*\K},\x);
    \draw[line width=.2px,black!20] ({2*\x},0) -- ({2*\x},\K);
}
\draw[] (0,0) rectangle ({2*\K},\K);
\foreach \x in {1,...,\K}{
    \draw[fill=black!10] ({2*\K-2*(\x-1)},\x-1) rectangle ({2*\K-2*\x,\x});
}
\draw[] (0,\K/2) -- ({2*\K},\K/2);
\draw[] (\K,0) -- (\K,\K);
\node[above] at (\K,\K+.2) {$\bm{V}^\T$};
\end{scope}
\end{scope}
\end{tikzpicture}

%% file: meta_algorithm.tex
\section{Near-optimal greedy butterfly approximation}
\label{sec:greedy_meta}

This section is devoted to analyzing a greedy algorithm for butterfly approximation (\Cref{alg:butterfly_greedy_meta}). 
This algorithm makes use of the recursive characterization of butterfly matrices given in \Cref{theorem:recursive}.
The algorithm does a recursive approximation with $\bsep(L,k)$ matrices and is guaranteed to return an approximation with error that is  within a  $O(\sqrt{L})$ factor from the optimal error, where $L = O(\log(N))$ is the number of levels. Along the way, we establish approximation bounds of our algorithm based on the low-rank approximability of the complementary low-rank blocks (\Cref{fig:CLR}). 
We also rely on approximation results for the closely related $\BLR^2$ matrix class established in~\cite{amsel2025quasi}. 
Finally, we note that an $O(\sqrt{L})$ quasi-optimality guarantee was also recently established in \cite{Le2025} using a fundamentally different algorithmic approach. 
\subsection{Approximation with $\bsep$-matrices}\label{section:bsep}
In this section, we present an algorithm for computing near-optimal $\bsep$ approximations; see \Cref{alg:bsep_greedy_meta}. This algorithm will serve as a building block and will be called at every level of our butterfly approximation algorithm introduced in \Cref{section:metaalg}. 
\begin{algorithm}
\caption{Greedy $\bsep$ approximation (meta-algorithm)}
\label{alg:bsep_greedy_meta}
\textbf{input:} Matrix $\bm{A} \in \mathbb{R}^{2^{L+1}k \times 2^{L+1}k}$, rank parameter $k$.\\
\textbf{output:} An approximation $\bm{B} \in \bsep(L,k)$ with factorization as in \Cref{def:bsep}.
\begin{algorithmic}[1]
\For{$i = 1,\ldots,2^{L}$}
\State Compute $\bm{U}_{i}, \bm{V}_i \in \mathbb{R}^{2k \times k}$ with orthonormal columns approximately minimizing 
\begin{equation*}
    \|\bm{A}(I_{i,L},:) - \bm{U}_i \bm{U}_i^\T\bm{A}(I_{i,L},:)\|_\F, \quad \|\bm{A}(:,I_{i,L}) - \bm{A}(:,I_{i,L})\bm{V}_i \bm{V}_i^\T\|_\F.
\end{equation*}
\EndFor
\State Set 
\begin{equation*}
    \bm{U} = \blockdiag(\bm{U}_1,\ldots,\bm{U}_{2^L}), \quad \bm{V} = \blockdiag(\bm{V}_1,\ldots,\bm{V}_{2^L}).
\end{equation*}
\State Define $\bm{X} = \bm{U}^\T \bm{A} \bm{V}$.
\State \Return $\bm{B} = \bm{U} \bm{X} \bm{V}^\T$ as defined in  \Cref{def:bsep}. 
\end{algorithmic}
\end{algorithm}
Noting that $\bsep(L,k) = \BLR^2(2^L,2k,k,\varnothing)$, an application \cite[Theorem C.3 and Remark 3.6]{amsel2025quasi} immediately yields the following guarantee.
\begin{theorem}\label{theorem:bsep}
    Suppose for $i = 1,\ldots,2^L$ the orthonormal bases $\bm{U}_i, \bm{V}_i \in \mathbb{R}^{2k \times k}$ from \Cref{alg:bsep_greedy_meta} satisfy
    \begin{align*}
        &\|\bm{A}(I_{i,L},:)-\bm{U}_i \bm{U}_i^\T\bm{A}(I_{i,L},:)\|_\F^2 \leq \Gamma \|\bm{A}(I_{i,L},:) - \llbracket \bm{A}(I_{i,L},:)\rrbracket_k\|_\F^2,\\
        &\|\bm{A}(:,I_{i,L})-\bm{A}(:,I_{i,L})\bm{V}_i \bm{V}_i^\T\|_\F^2 \leq \Gamma \|\bm{A}(:,I_{i,L}) - \llbracket \bm{A}(:,I_{i,L})\rrbracket_k\|_\F^2. 
    \end{align*}
    Then the output $\bm{B} \in \bsep(L,k)$ satisfies
    \begin{equation*}
        \|\bm{A} - \bm{B}\|_\F^2 \leq 2\Gamma \min\limits_{\bm{C} \in \bsep(L,k)}\|\bm{A} - \bm{C}\|_\F^2.
    \end{equation*}
    Furthermore, if
    \begin{equation*}
         \varepsilon_{i,\mathrm{r}} := \|\bm{A}(I_{i,L},:)-\llbracket \bm{A}(I_{i,L},:)\rrbracket_k\|_\F  \quad \text{and} \quad
        \varepsilon_{i,\mathrm{c}} := \|\bm{A}(:,I_{i,L})-\llbracket \bm{A}(:,I_{i,L})\rrbracket_k\|_\F
     \end{equation*}
     then
     \begin{equation*}
         \frac{1}{2}\sum\limits_{i=1}^{2^L} \left[\varepsilon_{i,\mathrm{r}}^2 + \varepsilon_{i,\mathrm{c}}^2\right] \leq \min\limits_{\bm{C} \in \bsep(L,k)} \|\bm{A} - \bm{C}\|_\F^2 \leq \|\bm{A} - \bm{B}\|_\F^2 \leq \Gamma \sum\limits_{i=1}^{2^L} \left[\varepsilon_{i,\mathrm{r}}^2 + \varepsilon_{i,\mathrm{c}}^2\right].
     \end{equation*}
\end{theorem}

\subsection{Bounding the optimum at different levels}\label{section:optbound}
In this section, we establish a key intermediate result. Let $\bm{U}$ and $\bm{V}$ are block-diagonal orthonormal bases as in \Cref{theorem:recursive} and define
\begin{equation*}
    \bm{U}^\T \bm{A} \bm{V} =: \bm{X} = \begin{bmatrix} \bm{X}_{1,1} & \bm{X}_{1,2} \\ \bm{X}_{2,1} & \bm{X}_{2,2} \end{bmatrix}
    , \quad \bm{X}_{s,t} \in \mathbb{R}^{2^{L-1}k \times 2^{L-1}k}.
\end{equation*}
We show that the optimal butterfly approximation error of $\bm{A}$ is always larger than the sum of the optimal butterfly approximation errors of $\bm{X}_{1,1},\bm{X}_{1,2},\bm{X}_{2,1},$ and $\bm{X}_{2,2}$. This result will be used to argue that if $\bm{U} \bm{X} \bm{V}^\T$ is a near-optimal $\bsep(L,k)$ approximation to $\bm{A}$, then by recursively computing near-optimal $\bsep(L-2,k)$ approximations of $\bm{X}_{1,1},\bm{X}_{1,2},\bm{X}_{2,1},$ and $\bm{X}_{2,2}$, then the error of the overall approximation is at most a $O(L)$ factor times the  optimum, yielding our main result.
\begin{theorem}\label{theorem:contraction}
    Let $\bm{A} \in \mathbb{R}^{2^{L+1} k \times 2^{L+1}k}$. Consider any matrices
     \begin{align*}
                \bm{U} &= \blockdiag(\bm{U}_1,\ldots,\bm{U}_{2^L}), \quad \bm{U}_i \in \mathbb{R}^{2k \times k} \text{ has orthonormal columns},\\
            \bm{V} &= \blockdiag(\bm{V}_1,\ldots,\bm{V}_{2^L}), \quad \bm{V}_i \in \mathbb{R}^{2k \times k} \text{ has orthonormal columns}.
            \end{align*}
Define $\bm{X} = \bm{U}^\T \bm{A} \bm{V}$ and partition
 \begin{equation*}
                \bm{X} = \begin{bmatrix} \bm{X}_{1,1} & \bm{X}_{1,2} \\ 
                \bm{X}_{2,1} & \bm{X}_{2,2} \end{bmatrix}, \quad \bm{X}_{s,t} \in \mathbb{R}^{2^{L-1}k \times 2^{L-1}k}.
\end{equation*}
Then,
\begin{equation*}
    \sum\limits_{s,t=1}^2 \min\limits_{\bm{C} \in \butterfly(L-2,k)}\|\bm{X}_{s,t} - \bm{C}\|_\F^2 \leq \min\limits_{\bm{C}\in \butterfly(L,k)}\|\bm{A} - \bm{C}\|_\F^2. 
\end{equation*}
\end{theorem}
To prove \Cref{theorem:contraction}, we first show that block-diagonal transformations of butterfly matrices are also butterfly matrices. 
\begin{lemma}\label[lemma]{lemma:preservation}
    If $\bm{B} \in \butterfly(L,k)$ and 
    \begin{align*}
        \bm{D} &= \blockdiag(\bm{D}_1,\ldots,\bm{D}_{2^L}), \quad \bm{D}_i \in \mathbb{R}^{2k \times 2k}\\
        \bm{S} &= \blockdiag(\bm{S}_1,\ldots,\bm{S}_{2^L}), \quad \bm{S}_i \in \mathbb{R}^{2k \times 2k}.
    \end{align*}
    Then $\bm{C} = \bm{D} \bm{B} \bm{S} \in \butterfly(L,k)$.
\end{lemma}
\begin{proof}
    Pick any complementary pair $(I_{i,\ell}, I_{j,L-\ell})$ as per \Cref{def:butterfly}.
    Then, $\bm{C}(I_{i,\ell},I_{j,L-\ell}) = \bm{D}(I_{i,\ell},:) \bm{B} \bm{S}(:,I_{j,L-\ell}) = \bm{D}(I_{i,\ell},I_{i,\ell}) \bm{B}(I_{i,\ell}, I_{j,L-\ell}) \bm{S}(I_{j,L-\ell},I_{j,L-\ell})$. Hence, $\rank(\bm{C}(I_{i,\ell},I_{j,L-\ell})) \leq \rank(\bm{B}(I_{i,\ell},I_{j,L-\ell}))\leq k$.
\end{proof}
With \Cref{lemma:preservation} at hand, we can proceed with the proof of \Cref{theorem:contraction}.
\begin{proof}[Proof of \Cref{theorem:contraction}]
    The proof is similar to \cite[Lemma 3.4]{amsel2025quasi}, which establishes a similar property for HSS matrices. Let $\bm{B}^{\star} \in \argmin\limits_{\bm{C} \in \butterfly(L,k)}\|\bm{A} - \bm{C}\|_\F^2$ have decomposition $\bm{B}^{\star} = \bm{U}^{\star} \bm{X}^{\star} (\bm{V}^{\star})^\T$  as per \Cref{theorem:recursive}.\footnote{Using identical arguments as in \cite[Theorem D.1]{amsel2025quasi} one can show that $\butterfly(L,k)$ is a closed set. Hence, an optimal approximation always exists.} Partition
    \begin{equation*}
        \bm{B}^{\star} = \begin{bmatrix} \bm{B}_{1,1} & \bm{B}_{1,2} \\ 
                \bm{B}_{2,1} & \bm{B}_{2,2} \end{bmatrix} = \begin{bmatrix} \bm{U}_{[1]}^{\star} \bm{X}^{\star}_{1,1}(\bm{V}_{[1]}^{\star})^\T & \bm{U}_{[1]}^{\star} \bm{X}^{\star}_{1,2}(\bm{V}_{[2]}^{\star})^\T \\ 
                \bm{U}_{[2]}^{\star} \bm{X}^{\star}_{2,1}(\bm{V}_{[1]}^{\star})^\T & \bm{U}_{[2]}^{\star} \bm{X}^{\star}_{2,2}(\bm{V}_{[2]}^{\star})^\T \end{bmatrix}
    \end{equation*}
    Hence, if we partition $\bm{U} = \blockdiag(\bm{U}_{[1]},\bm{U}_{[2]})$ and $\bm{V} = \blockdiag(\bm{V}_{[1]},\bm{V}_{[2]})$ into diagonal blocks of equal size, then 
    \begin{equation*}
        \bm{U}^\T \bm{B}^{\star} \bm{V} = \begin{bmatrix} \bm{U}_{[1]}^{\T}\bm{U}_{[1]}^{\star} \bm{X}^{\star}_{1,1}(\bm{V}_{[1]}^{\T}\bm{V}_{[1]}^{\star})^\T & \bm{U}_{[1]}^{\T}\bm{U}_{[1]}^{\star} \bm{X}^{\star}_{1,2}(\bm{V}_{[2]}^{\T}\bm{V}_{[2]}^{\star})^\T \\ 
                \bm{U}_{[2]}^{\T}\bm{U}_{[2]}^{\star} \bm{X}^{\star}_{2,1}(\bm{V}_{[1]}^{\T}\bm{V}_{[1]}^{\star})^\T & \bm{U}_{[2]}^{\T}\bm{U}_{[2]}^{\star} \bm{X}^{\star}_{2,2}(\bm{V}_{[2]}^{\T}\bm{V}_{[2]}^{\star})^\T \end{bmatrix} :=\begin{bmatrix} \bm{C}_{1,1} & \bm{C}_{1,2} \\ 
                \bm{C}_{2,1} & \bm{C}_{2,2} \end{bmatrix}.
    \end{equation*}
    Note that $\bm{U}_{[s]}^{\T}\bm{U}_{[s]}^{\star}$ and $\bm{V}_{[s]}^{\T}\bm{V}_{[s]}^{\star}$ can be viewed as a block-diagonal matrix with $2^{L-2}$ diagonal blocks of size $2k \times 2k$. Furthermore, by \Cref{theorem:recursive} we have $\bm{X}_{s,t}^{\star} \in \butterfly(L-2,k)$, so by \Cref{lemma:preservation} $\bm{C}_{s,t}\in \butterfly(L-2,k)$. Hence,
    \begin{align*}
        \sum\limits_{s,t=1}^2 \min\limits_{\bm{C} \in \butterfly(L-2,k)}\|\bm{X}_{s,t} - \bm{C}\|_\F^2 &\leq  \sum\limits_{s,t=1}^2 \|\bm{X}_{s,t} - \bm{C}_{s,t}\|_\F^2=\|\bm{X} - \bm{U}^\T\bm{B}^{\star}\bm{V}\|_\F^2\leq \|\bm{A} -\bm{B}^{\star}\|_\F^2= \min\limits_{\bm{C} \in \butterfly(L,k)}\|\bm{A} - \bm{C}\|_\F^2,
    \end{align*}
    as required. 
\end{proof}
\subsection{A greedy meta-algorithm for butterfly approximation}\label{section:metaalg}
We now show how to combine the results from \Cref{section:bsep,section:optbound} to arrive at an algorithm to compute butterfly approximations. The algorithm is described in \Cref{alg:butterfly_greedy_meta}, which is a generic greedy algorithm that proceeds by recursively computing near-optimal $\bsep$ approximations using \Cref{alg:bsep_greedy_meta}. We show that \Cref{alg:butterfly_greedy_meta} returns a quasi-optimal butterfly approximation provided that \Cref{alg:bsep_greedy_meta} is implemented so that the assumptions in \Cref{theorem:bsep} hold.

\begin{algorithm}
\caption{Greedy $\butterfly$ approximation (meta-algorithm)}
\label{alg:butterfly_greedy_meta}
\textbf{input:} Matrix $\bm{A} \in \mathbb{R}^{2^{L+1}k \times 2^{L+1}k}$, rank parameter $k$ and even level $L$.\\
\textbf{output:} An approximation $\bm{B} \in \butterfly(L,k)$ with factorization as in \Cref{theorem:recursive}.
\begin{algorithmic}[1]
\State Apply \Cref{alg:bsep_greedy_meta} to to obtain $\bm{U} \widehat{\bm{X}} \bm{V}^\T \in \bsep(L,k)$ where $\widehat{\bm{X}} = \bm{U}^\T \bm{A} \bm{V}$. 
\If{$L = 0$}
\State $\bm{X} = \widehat{\bm{X}}$.
\Else
\State Partition 
\begin{equation*}
    \widehat{\bm{X}} = \begin{bmatrix} \widehat{\bm{X}}_{1,1} & \widehat{\bm{X}}_{1,2} \\ 
                \widehat{\bm{X}}_{2,1} & \widehat{\bm{X}}_{2,2} \end{bmatrix}, \quad \widehat{\bm{X}}_{s,t} \in \mathbb{R}^{2^{L-1}k \times 2^{L-1}k}.
\end{equation*}
\For{$s,t = 1,2$}
\State Recursively apply \Cref{alg:butterfly_greedy_meta} to $\widehat{\bm{X}}_{s,t}$ to obtain $\bm{X}_{s,t} \in \butterfly(L-2,k)$.
\EndFor
\State Define
\begin{equation*}
    \bm{X} := \begin{bmatrix} \bm{X}_{1,1} & \bm{X}_{1,2} \\ 
                \bm{X}_{2,1} & \bm{X}_{2,2} \end{bmatrix}.
\end{equation*}
\EndIf
\State \textbf{return} $\bm{B} = \bm{U} \bm{X} \bm{V}^\T$.
\end{algorithmic}
\end{algorithm}
\begin{theorem}[Quasi-optimal butterfly approximation]\label{theorem:meta_butterfly}
    Suppose that \Cref{alg:bsep_greedy_meta} is implemented so the conditions in \Cref{theorem:bsep} hold. Then the output of \Cref{alg:butterfly_greedy_meta} satisfies
    \begin{equation*}
        \|\bm{A} - \bm{B}\|_\F^2 \leq \Gamma(L+2)\min\limits_{\bm{C} \in \butterfly(L,k)}\|\bm{A} - \bm{C}\|_\F^2.
    \end{equation*}
\end{theorem}

\begin{proof}
    The proof is by induction on $h$, where $L = 2h$. The proof for $L = 0$ is immediate due to \Cref{theorem:bsep}, since $\butterfly(0,k) = \bsep(0,k)$. Suppose the statement is true for $L-2$. Using the notation of \Cref{alg:butterfly_greedy_meta}, we obtain
    \begin{align*}
        \|\bm{A} - \bm{B}\|_\F^2 &= \|\bm{A} - \bm{U} \widehat{\bm{X}}\bm{V}^\T\|_\F^2 + \sum\limits_{s,t=1}^2\|\widehat{\bm{X}}_{s,t} - \bm{X}_{s,t}\|_\F^2 \tag{Pythagoras' theorem}\\
        &\leq2\Gamma \min\limits_{\bm{C}\in \bsep(L,k)} \|\bm{A} - \bm{C}\|_\F^2 + \sum\limits_{s,t=1}^2\|\widehat{\bm{X}}_{s,t} - \bm{X}_{s,t}\|_\F^2 \tag{Assumption on \Cref{alg:bsep_greedy_meta}}\\
        &\leq 2\Gamma \min\limits_{\bm{C}\in \butterfly(L,k)} \|\bm{A} - \bm{C}\|_\F^2 + \sum\limits_{s,t=1}^2\|\widehat{\bm{X}}_{s,t} - \bm{X}_{s,t}\|_\F^2 \tag{$\butterfly(L,k) \subseteq \bsep(L,k)$}\\
        &\leq 2\Gamma  \min\limits_{\bm{C}\in \butterfly(L,k)} \|\bm{A} - \bm{C}\|_\F^2 + \Gamma L\sum\limits_{s,t=1}^2\min\limits_{\bm{C} \in \butterfly(L-2,k)}\|\widehat{\bm{X}}_{s,t} - \bm{C}\|_\F^2 \tag{Inductive hypothesis}\\
        &\leq 2\Gamma  \min\limits_{\bm{C}\in \butterfly(L,k)} \|\bm{A} - \bm{C}\|_\F^2 + \Gamma L \min\limits_{\bm{C}\in \butterfly(L,k)} \|\bm{A} - \bm{C}\|_\F^2 \tag{\Cref{theorem:contraction}}\\
        &= \Gamma(L+2) \min\limits_{\bm{C}\in \butterfly(L,k)} \|\bm{A} - \bm{C}\|_\F^2,
        \end{align*}
        as required.
\end{proof}
\Cref{alg:butterfly_greedy_meta} can be implemented by approximating each complementary low-rank block using the rank-$k$ truncated SVD. 
For this implementation of \Cref{alg:bsep_greedy_meta}, we have the following result. 

\begin{corollary}
    Let $\bm{A} \in \mathbb{R}^{2^{L+1}k \times 2^{L+1}k}$ and let $N = 2^{L+1}k$ denote the matrix size. Suppose that \Cref{alg:bsep_greedy_meta} is implemented with the truncated SVD. Then, the algorithm requires $O(N^2 k)$ operations and returns an approximation $\bm{B}\in\butterfly(L,k)$ satisfying
    \begin{equation*}
        \|\bm{A} - \bm{B}\|_\F^2 \leq (L+2) \min\limits_{\bm{B}\in \butterfly(L,k)} \|\bm{A} - \bm{C}\|_\F^2.
    \end{equation*}
\end{corollary}
\begin{proof}
    If \cref{alg:bsep_greedy_meta} is implemented using the truncated SVD, then $\Gamma = 1$. Thus, the accuracy statement is an immediate corollary of \Cref{theorem:meta_butterfly}.
    We now proceed with computing the complexity. Let $C(L)$ denote the complexity of applying \Cref{alg:butterfly_greedy_meta} to a matrix of size $2^{L+1}k \times 2^{L+1}k$. Then $C(0) = O(k^3)$. Since \Cref{alg:bsep_greedy_meta} requires $O(4^{L+1}k^3)$, the complexity $C(L)$ satisfies the following recurrence relation
    \begin{equation*}
        C(L) = O(4^{L+1}k^3) + 4C(L-2) = O(4^{L+1}k^3) = O(N^2 k),
    \end{equation*}
    where we used $N = 2^{L+1}k$.
\end{proof}

\begin{remark}
One may also use other methods for low-rank approximation.
Possibilities include column pivoted QR and the interpolative decomposition, which can produce a low-rank approximation of a matrix without observing all of the entries \cite{Pang2020}.
However, the theoretical guarantees for such methods are complicated and depend on various properties of the matrix in question.
\end{remark}

The proof of \Cref{theorem:meta_butterfly} can easily be modified to obtain absolute error bounds, which are useful in practical implementations. 
In particular, we have the following absolute error bounds. We note that \Cref{theorem:absolute} ii) is similar to \cite[Theorem 5.3]{LiuXingGuo:2021}.

\begin{theorem}\label{theorem:absolute}
    Consider a matrix $\bm{A} \in \mathbb{R}^{2^{L+1} k \times 2^{L+1} k}$ and let $N = 2^{L+1}k$. Let $I_{i,\ell}$ be as in \Cref{def:butterfly}. For each $\ell = 0,1,\ldots,L$ and $i,j$ suppose 
     \begin{align*}
          \|\bm{A}(I_{i,\ell},I_{j,L-\ell})-\llbracket \bm{A}(I_{i,\ell},I_{j,L-\ell})\rrbracket_k\|_\F \leq \varepsilon(I_{i,\ell},I_{j,L-\ell}).
     \end{align*}
     Then, under the assumptions of \Cref{theorem:bsep}, if $\bm{B}$ is the approximation from \Cref{alg:butterfly_greedy_meta}
    we have
    \begin{equation}\label{eq:absolute_main}
        \|\bm{A} - \bm{B}\|_\F^2 \leq \Gamma \sum\limits_{\ell = 0}^{L/2} \sum\limits_{i=1}^{2^{\ell}} \sum\limits_{j=1}^{2^{L-\ell}} \left[\varepsilon(I_{i,\ell},I_{j,L-\ell})^2 + \varepsilon(I_{j,L-\ell},I_{i,\ell})^2\right].
    \end{equation}
    Hence, 
    \begin{enumerate}[label=\roman*)]
        \item if $\varepsilon(I_{i,\ell},I_{j,L-\ell}) \leq \epsilon$, then
        \begin{equation}\label{eq:absolute1}
            \|\bm{A} - \bm{B}\|_\F \leq \sqrt{\Gamma \frac{N}{2k}(L+2)} \cdot \epsilon.
        \end{equation}
        \item\label{theorem:absolute:ii} 
        if $\varepsilon(I_{i,\ell},I_{j,L-\ell}) \leq \varepsilon \|\bm{A}(I_{i,\ell},I_{j,L-\ell})\|_\F$, then
         \begin{equation}\label{eq:absolute2}
            \|\bm{A} - \bm{B}\|_\F \leq   \sqrt{\Gamma (L+2)} \cdot \varepsilon \|\bm{A}\|_\F.
        \end{equation}
    \end{enumerate}
\end{theorem}
While the proof is similar in spirit to \Cref{theorem:meta_butterfly}, it relies on notation introduced in \Cref{section:recursive}. We therefore defer the proof to \Cref{section:absolute}. 
We conclude this section with a useful remark. 
\begin{remark}
    Suppose that every matrix in \Cref{alg:butterfly_greedy_meta} that is sent to \Cref{alg:bsep_greedy_meta} is approximated with relative error $\varepsilon$, then the proof of \Cref{theorem:absolute} implies that final output has relative error at most $\varepsilon\sqrt{L+2}$. This may be useful for rank adaptive algorithms.
\end{remark}

%% file: matvec_algorithm.tex
\section{A memory-efficient matrix-vector product query algorithm}\label{section:matvec}

We now shift focus to the setting when $\bm{A}$ can only be accessed through matvecs  $\bm{x} \mapsto \bm{A} \bm{x}$ and $\bm{y} \mapsto \bm{A}^\T \bm{y}$.
We design and analyze an algorithm that returns a butterfly matrix whose error is, with high probability, within a factor $O(N^{1/4})$ of the optimal error using $\widetilde{O}(N)$ units of memory, where $N$ is the matrix size. 
Here, $\widetilde{O}(\cdot)$ suppresses polynomial factors in the number of levels $L = O(\log(N))$; these factors are made explicit later in this section.

To highlight the main ideas, we first present a simple implementation in \Cref{section:memoryeff_alg} and establish a quasi-optimality guarantee in \cref{section:memoryeff_proof}. 
In this form, the algorithm is conceptually simpler but requires all $\widetilde{O}(\sqrt{N})$ matvecs to be computed and stored up front, leading to a memory cost of $\widetilde{O}(N^{3/2})$. 
In \Cref{section:memoryeff}, we show how to implement the algorithm in a memory-efficient manner by computing matvecs when needed rather than upfront, thereby reducing the memory requirements to $\widetilde{O}(N)$, while preserving the quasi-optimality guarantees. 
The algorithm we present can be viewed as an instance of the black-box hybrid butterfly factorization approach developed in \cite{Guo2017,LiuXingGuo:2021}.

\begin{remark}\label[remark]{remark:reduction}
    Before proceeding, we note that given \emph{any} algorithm to compute quasi-optimal butterfly approximations, one can always obtain an algorithm using $O(\sqrt{N})$ matvecs with $\bm{A}$ and $\bm{A}^\T$ that returns a quasi-optimal butterfly approximation. Consider the set $\butterflymid(L,k)$ of matrices satisfying $\rank(\bm{B}(I_{i,L/2},I_{j,L/2})) \leq k$. I.e., the set of matrices satisfying the complementary low-rank property at level $\ell = L/2$ (see $\ell = 2$ in \Cref{fig:CLR}). Using $O(\sqrt{N})$ matvecs we can always find a near-optimal approximation $\widehat{\bm{B}}$ from this class. Since $\butterfly(L,k)$ is a subset of $\butterflymid(L,k)$, if $\bm{B}$ is a quasi-optimal butterfly approximation to $\widehat{\bm{B}}$, then by a simple triangle inequality argument $\bm{B}$ is a quasi-optimal butterfly approximation to $\bm{A}$. 
    However, storing any matrix in $\butterflymid(L,k)$ in factored form requires $O(N^{3/2})$ units of memory, and such an algorithm can therefore never be memory-efficient. 
    See \Cref{sec:matvec_strongeraccuracy} for further details.
\end{remark}

\subsection{Basic idea and simple implementation}\label{section:memoryeff_alg}
To introduce the algorithm, we first consider the idealized setting in which $\bm{A} \in \butterfly(L,k)$. In other words, $\bm{A}$ is exactly butterfly. 
By \Cref{theorem:recursive}, $\bm{A}$ admits a factorization of the form
\begin{equation*}
    \bm{A} = \bm{U} \bm{X} \bm{V}^\T,
\end{equation*}
where
\begin{align*}
    \bm{U} &= \blockdiag(\bm{U}_1,\ldots,\bm{U}_{2^L}), \quad \bm{U}_i \in \mathbb{R}^{2k \times k} \text{ has orthonormal columns},\\
    \bm{V} &= \blockdiag(\bm{V}_1,\ldots,\bm{V}_{2^L}), \quad \bm{V}_i \in \mathbb{R}^{2k \times k} \text{ has orthonormal columns}.
\end{align*}
Consequently, if $I_{i,L}$ is a leaf-node as in \Cref{def:index_sets}, then
\begin{equation*}
    \rank(\bm{A}(I_{i,L},:)), \rank(\bm{A}(:,I_{i,L}))\leq k,\quad i = 1,\ldots,2^L.
\end{equation*}
Hence, we can recover $\range(\bm{A}(I_{i,L},:))$ and $\range(\bm{A}(:,I_{i,L}))$ using randomized methods. 
If $\bm{\Omega}^{(L)} \sim \gaussian(2^{L+1}k,\alpha k)$, where $\alpha \geq 1$, then
\begin{equation*}
    \bm{Y}^{(L)} = \bm{A} \bm{\Omega}^{(L)} \Rightarrow \range(\bm{Y}^{(L)}(I_{i,L},:)) = \range(\bm{A}(I_{i,L},:)), \quad i = 1,\ldots,2^L,
\end{equation*}
almost surely. Hence, if $\bm{U}_i$ are the top $k$ left singular vectors of $\bm{Y}^{(L)}(I_{i,L},:)$, and $\bm{U} = \blockdiag(\bm{U}_1,\ldots,\bm{U}_{2^L})$ we have $\bm{A} = \bm{U}\bm{U}^\T \bm{A}$.
By an entirely similar argument, we can compute an orthonormal basis $\bm{V} = \blockdiag(\bm{V}_1,\ldots,\bm{V}_{2^L})$ from $\bm{Z}^{(L)} = \bm{A}^\T \bm{\Psi}^{(L)}$, where $\bm{\Psi}^{(L)} \sim \gaussian(2^{L+1}k,\alpha k)$, so that
\begin{equation*}
    \bm{A} = \bm{U}\bm{U}^\T\bm{A} \bm{V}\bm{V}^\T = \bm{U} \bm{X} \bm{V}^\T, \quad \bm{X} := \bm{U}^\T\bm{A} \bm{V}.
\end{equation*}
Since $\bm{A} \in \butterfly(L,k)$, by \Cref{theorem:recursive} $\bm{X}$ can be partitioned into a $2 \times 2$ block-matrix with blocks of equal size
\begin{equation*}
    \bm{X} = \begin{bmatrix} \bm{X}_{1,1} & \bm{X}_{1,2} \\ 
    \bm{X}_{2,1} & \bm{X}_{2,2} \end{bmatrix},
\end{equation*}
where for $s,t = 1,2$, $\bm{X}_{s,t} \in \butterfly(L-2,k)$. 
We can, once again, recover the column and row bases of the blocks $\bm{X}_{s,t}$ for $s,t = 1,2$ from randomized sketches. 
Draw two independent Gaussian random matrices $\bm{\Omega}_{[1]}^{(L-2)}, \bm{\Omega}_{[2]}^{(L-2)} \sim \gaussian(2^{L}k,\alpha k)$. 
Then, if $\bm{V} = \blockdiag(\bm{V}_{[1]},\bm{V}_{[2]})$ we have
\begin{equation*}
    \bm{U}^\T \bm{A} \begin{bmatrix} \bm{\Omega}_{[1]} & \\ & \bm{\Omega}_{[2]} \end{bmatrix} = \begin{bmatrix} \bm{X}_{1,1} \left(\bm{V}_{[1]}^\T \bm{\Omega}_{[1]}\right) & \bm{X}_{1,2} \left(\bm{V}_{[2]}^\T \bm{\Omega}_{[2]}\right) \\ 
    \bm{X}_{2,1} \left(\bm{V}_{[1]}^\T \bm{\Omega}_{[1]}\right) & \bm{X}_{2,2} \left(\bm{V}_{[2]}^\T \bm{\Omega}_{[2]}\right) \end{bmatrix}. 
\end{equation*}
By the unitary invariance of standard Gaussian matrices, conditioned on $\bm{V}$, $\bm{V}_{[t]}^\T\bm{\Omega}_{[t]}$ is a standard Gaussian matrix. 
Hence, $\bm{X}_{s,t} \left(\bm{V}_{[t]}^\T \bm{\Omega}_{[t]}\right)$ is a sketch of $\bm{X}_{s,t} \in \butterfly(L-2,k)$ and we can recover its columns space from this sketch using the method described above. 
By an entirely analogous argument, we can recover the row-space of $\bm{X}_{s,t}$:  if $\bm{\Psi}_{[1]}^{(L-2)}, \bm{\Psi}_{[2]}^{(L-2)} \sim \gaussian(2^{L}k,\alpha k)$ are two independent standard Gaussian random matrices, then
\begin{equation*}
    \bm{V}^\T \bm{A}^\T \begin{bmatrix} \bm{\Psi}_{[1]}^{(L-2)} & \\ & \bm{\Psi}_{[2]}^{(L-2)} \end{bmatrix} = \begin{bmatrix} \bm{X}_{1,1}^\T \left(\bm{U}_{[1]}^\T \bm{\Psi}_{[1]}^{(L-2)}\right) & \bm{X}_{2,1}^\T \left(\bm{U}_{[2]}^\T \bm{\Psi}_{[2]}^{(L-2)}\right) \\ 
    \bm{X}_{1,2}^\T \left(\bm{U}_{[1]}^\T \bm{\Psi}_{[1]}^{(L-2)}\right) & \bm{X}_{2,2}^\T \left(\bm{U}_{[2]}^\T \bm{\Psi}_{[2]}^{(L-2)}\right) \end{bmatrix}, 
\end{equation*}
where $\bm{X}_{s,t}^\T\left(\bm{U}_{[s]}^\T \bm{\Psi}_{[s]}^{(L-2)}\right)$ is a sketch of $\bm{X}_{s,t}^\T$.

The derivation above illustrates the recursive structure of the recovery procedure. 
Repeating the same argument recursively yields a recovery procedure for every level of the butterfly structure. 
For each level $\ell = L,L-2,\ldots,2,0$ define the block-diagonal sketch matrices
\begin{equation}
\begin{aligned}
\bm{\Omega}^{(\ell)} &= \blockdiag\left(\bm{\Omega}_1^{(\ell)},\ldots,\bm{\Omega}^{(\ell)}_{2^{\frac{L-\ell}{2}}}\right) = \blockdiag(\bm{\Omega}_{[1]}^{(\ell)}, \bm{\Omega}_{[2]}^{(\ell)}),\\
    \bm{\Psi}^{(\ell)} &= \blockdiag\left(\bm{\Psi}_1^{(\ell)},\ldots,\bm{\Psi}^{(\ell)}_{2^{\frac{L-\ell}{2}}}\right) = \blockdiag(\bm{\Psi}_{[1]}^{(\ell)}, \bm{\Psi}_{[2]}^{(\ell)})
\end{aligned}
\label{eqn:sketch_def}
\end{equation}
where $\bm{\Omega}_i^{(\ell)}, \bm{\Psi}_i^{(\ell)} \stackrel{i.i.d.}{\sim} \gaussian(2^{\frac{L+\ell}{2}+1}k, \alpha k),$ and
\begin{equation}\label{eq:matvecs}
    \bm{Y}^{(\ell)} = \bm{A} \bm{\Omega}^{(\ell)}, \quad \bm{Z}^{(\ell)} = \bm{A}^\T \bm{\Psi}^{(\ell)}.
\end{equation}

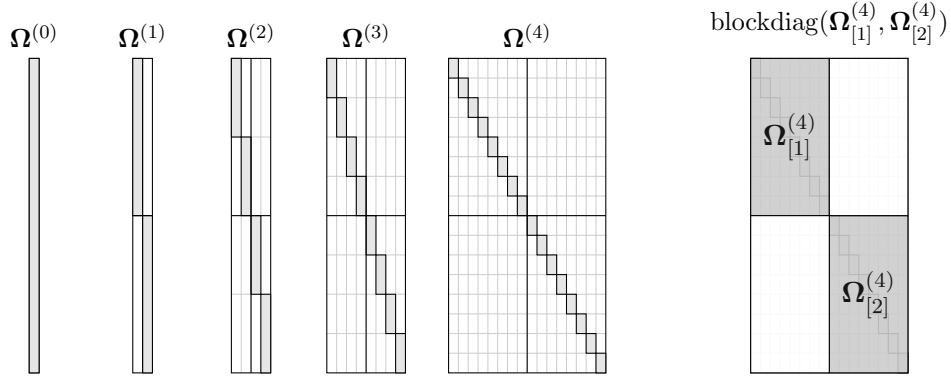
\begin{figure}
    \centering
    \input{imgs/sketch_ex}
    \caption{Visualization of the nonzero entries of the sketch matrices $\bm{\Omega}^{(\ell)}$ defined in \cref{eqn:sketch_def} for $L=4$.
    On the right, we illustrate the partitioning of $\bm{\Omega}^{(\ell)}$into $\bm{\Omega}_{[1]}^{(\ell)}$ and $\bm{\Omega}_{[2]}^{(\ell)}$ for $\ell = 4$. 
    The same block-diagonal structure is also used for the other $\ell$.}
    \label{fig:sketch_ex}
\end{figure}
The block-diagonal structure of $\bm{\Omega}^{(\ell)}$ and $\bm{\Psi}^{(\ell)}$ are shown in \Cref{fig:sketch_ex}. The sketches $\bm{Y}^{(\ell)}$ and $\bm{Z}^{(\ell)}$ are used to recover the bases at level $\ell$ of the recursion for $\ell = L,L-2,\ldots,2$. 
At the final level $\ell = 0$, recovering a matrix in $\butterfly(0,k)$ is the same as recovering a $2k \times 2k$ matrix of rank $k$, which can be done using the generalized Nyström approximation \cite{Nakatsukasa:2020}. 
In particular, if $\bm{A} \in \butterfly(0,k)$, and $\bm{U} \in \mathbb{R}^{2k \times k}$ are the top $k$ left singular vectors of $\bm{A}\bm{\Omega}$, where $\bm{\Omega} \sim \gaussian(2k,\alpha k)$, then if $\bm{\Psi} \sim \gaussian(2k,\alpha k)$ we have
\begin{equation*}
    \bm{A} = \bm{U} (\bm{\Psi}^\T \bm{U})^{\dagger} \bm{\Psi}^\T\bm{A}
\end{equation*}
almost surely. 
With these ideas at hand, we arrive at \Cref{alg:butterfly_greedy_matvec}, which requires $O(\alpha \sqrt{Nk})$ matvecs with $\bm{A}$ and $\bm{A}^\T$ and $O(\alpha (Nk)^{3/2})$ additional operations, where $N = 2^{L+1}k$ is the matrix-size. Note that \Cref{alg:butterfly_greedy_matvec} is a special instantiation of \Cref{alg:butterfly_greedy_meta}: lines~\ref{line:6}-\ref{line:10} correspond to the call of \Cref{alg:bsep_greedy_meta} in \Cref{alg:butterfly_greedy_meta}; lines~\ref{line:11}-\ref{line:16} correspond to processing the sketches of $\bm{A}$ to obtain (approximate) sketches with the blocks of $\bm{X} = \bm{U}^\T \bm{A} \bm{V}$; finally, lines~\ref{line:17}-\ref{line:19} correspond to the recursion in \Cref{alg:butterfly_greedy_meta}.
Before proceeding, we make three remarks:
\begin{enumerate}
    \item[i)] The sketches $\bm{Y}^{(\ell)},\bm{Z}^{(\ell)}$ are computed and stored upfront and used as inputs in \Cref{alg:butterfly_greedy_matvec}. 
    Consequently, the memory complexity is therefore $O(\alpha N \sqrt{Nk})$. In \Cref{section:memoryeff} we show how the matvecs can be computed only when needed and discarded immediately after. This reduces the memory complexity to $O(Nk\log(N/k) + \alpha Nk)$.
    \item[ii)] The preceding derivation assumed that $\bm{A}$ was exactly butterfly. In practice, however, $\bm{A}$ need not be exactly butterfly, but we have $\bm{A} = \bm{B} + \bm{E}$, where $\bm{B} 
    \in \butterfly(L,k)$ is an unknown approximation. 
    This introduces error in the algorithm, and in \Cref{section:memoryeff_proof} we bound the error to prove that the approximation error of the output of the algorithm is, with high probability, at most $O(N^{1/4})$ away from the optimal approximation error.
    \item[iii)] The parameter $\alpha$ controls the amount of oversampling and therefore the matvec and memory cost of the algorithm. For general $\alpha$, the algorithm uses $O(\alpha \sqrt{Nk})$ matvecs and $O(Nk\log(N/k) + \alpha Nk)$ working memory. For the practically common choice $\alpha = O(1)$ yields $O(\sqrt{Nk})$ matvecs and $O(Nk\log(N/k))$ memory, which is optimal. The high-probability quasi-optimality guarantee established in the following section require $\alpha$ to grow polylogarithmically with $N/k$, resulting in polylogarithmic overheads in the matvec and memory complexities. 
\end{enumerate}

\begin{algorithm}
\caption{Greedy $\butterfly$ approximation (memory-inefficient matvec algorithm)}
\label{alg:butterfly_greedy_matvec}
\textbf{input:} Block-diagonal sketch matrix $\bm{\Psi}^{(0)}$  and sketches $\{\bm{Y}^{(\ell)}, \bm{Z}^{(\ell)} \}_{\ell = 0,2\ldots,L}$. Rank parameter $k$. Level $L$. \\
\textbf{output:} An approximation $\bm{B} \in \butterfly(L,k)$ with factorization as in \Cref{theorem:recursive}.
\begin{algorithmic}[1]
\If{$L = 0$}\label{line:1}
\State Compute the top $k$ left singular vectors $\bm{U}$ for $\bm{Y}^{(0)}$. 
\State $\bm{M} = (\bm{\Psi}^{(0)\T} \bm{U})^{\dagger} \bm{Z}^{(0)\T}$.
\State Compute the (rank reduced) SVD of $\bm{M}$ to obtain $\bm{W} \bm{S} \bm{V}^\T$.
\State Let $\bm{X} =\bm{W} \bm{S}$. \label{line:5}
\Else\label{line:6}
\State Partition 
\begin{align*}
    \bm{Y}^{(L)} = \begin{bmatrix} \bm{Y}^{(L,1)}\\\vdots\\\bm{Y}^{(L,2^{L})} \end{bmatrix}, \quad \bm{Y}^{(L,i)} \in \mathbb{R}^{2k \times \alpha k},\quad
     \bm{Z}^{(L)} = \begin{bmatrix} \bm{Z}^{(L,1)}\\\vdots\\\bm{Z}^{(L,2^{L})} \end{bmatrix}, \quad \bm{Z}^{(L,i)} \in \mathbb{R}^{2k \times \alpha k}.
\end{align*}
\For{$i = 1,\ldots,2^L$}
 \State Compute the top $k$ left singular vectors $\bm{U}_i$ and $\bm{V}_i$ for $\bm{Y}^{(L,i)}$ and $\bm{Z}^{(L,i)}$, respectively.
\EndFor\label{line:10}
\State Let $\bm{U} = \blockdiag(\bm{U}_1,\ldots,\bm{U}_{2^L})$ and $\bm{V} = \blockdiag(\bm{V}_1,\ldots,\bm{V}_{2^L})$.
\State Compute $\widehat{\bm{\Psi}}^{(0)} = \bm{U}^\T \bm{\Psi}^{(0)}$ and partition in half:
\begin{equation*}
    \widehat{\bm{\Psi}}^{(0)} = \begin{bmatrix} \widehat{\bm{\Psi}}_{[1]}^{(0)} & \\ & \widehat{\bm{\Psi}}_{[2]}^{(0)} \end{bmatrix}.
\end{equation*}\label{line:11}
\For{$\ell = L-2,\ldots,2,0$}\label{line:12}
\State Compute $\widehat{\bm{Y}}^{(\ell)} = \bm{U}^\T \bm{Y}^{(\ell)}, \widehat{\bm{Z}}^{(\ell)} = \bm{V}^\T \bm{Z}^{(\ell)}$.
\State Partition 
\begin{equation*}
    \widehat{\bm{Y}}^{(\ell)} = \begin{bmatrix} \widehat{\bm{Y}}^{(\ell)}_{1,1} & \widehat{\bm{Y}}^{(\ell)}_{1,2} \\
    \widehat{\bm{Y}}^{(\ell)}_{2,1} & \widehat{\bm{Y}}^{(\ell)}_{2,2} \end{bmatrix}, \quad \widehat{\bm{Z}}^{(\ell)} = \begin{bmatrix} \widehat{\bm{Z}}^{(\ell)}_{1,1} & \widehat{\bm{Z}}^{(\ell)}_{2,1} \\
    \widehat{\bm{Z}}^{(\ell)}_{1,2} & \widehat{\bm{Z}}^{(\ell)}_{2,2} \end{bmatrix},
\end{equation*}
where the partition is exactly half in both axis.
\EndFor\label{line:16}
\For{$s,t = 1,2$}\label{line:17}
\State Recursively apply \Cref{alg:butterfly_greedy_matvec} with inputs
\begin{align*}
    \text{Sketch matrix:}& \quad \widehat{\bm{\Psi}}_{[s]}^{(0)},\\
    \text{Sketches:}& \quad \{\widehat{\bm{Y}}_{s,t}^{(\ell)},\widehat{\bm{Z}}_{s,t}^{(\ell)}\}_{\ell = 0,2,\ldots,L-2},\\
    \text{Rank parameter:}& \quad k,\\
    \text{Level:} & \quad L-2,
\end{align*}
to obtain $\bm{X}_{s,t} \in \butterfly(L-2,k)$.
\EndFor\label{line:19}
\State Define
\begin{equation*}
    \bm{X} := \begin{bmatrix} \bm{X}_{1,1} & \bm{X}_{1,2} \\ 
                \bm{X}_{2,1} & \bm{X}_{2,2} \end{bmatrix}.
\end{equation*}
\EndIf
\State \text{return} $\bm{B} = \bm{U} \bm{X} \bm{V}^\T$.
\end{algorithmic}
\end{algorithm}

\subsection{Analysis}
\label{section:memoryeff_proof}
This section is dedicated to an analysis of the approximation error of the a butterfly approximation returned by \Cref{alg:butterfly_greedy_matvec}. 
Recall that the derivation of the algorithm in the previous section assumed that the matrix of interest $\bm{A}$ is exactly butterfly. 
In general, this is not true. 
As a consequence, the algorithm will introduce errors that will corrupt each of the block-diagonal orthonormal bases. 
We now provide a guarantee on the approximation quality for each complementary low-rank block under a noisy matvec model, where the noisy term has a particular structure due to the recursive nature of~\cref{alg:butterfly_greedy_matvec}.
\begin{lemma}\label[lemma]{thm:perturbation1}
     Let $\bm{B} \in \mathbb{R}^{m \times n}$, $\bm{E} \in \mathbb{R}^{m \times q}$ be fixed and suppose $\bm{\Omega} \sim \gaussian(n,\alpha k)$ and $\widetilde{\bm{\Omega}} \sim \gaussian(q,\alpha k)$ are independent.  Let $\bm{U}$ be the top $k$ left singular vectors of $\bm{B} \bm{\Omega} + \bm{E} \widetilde{\bm{\Omega}}$. Let $\delta \in (0,1/2]$ and $\varepsilon \in (0,1]$. Then, if $\alpha = O\left( \frac{\log(1/\delta)}{\varepsilon^2}\right)$, we have with probability at least $1-\delta$
    \begin{equation*}
        \|\bm{B} - \bm{U}\bm{U}^\T \bm{B}\|_\F^2 \leq (1+\varepsilon)  [\|\bm{B} - \llbracket \bm{B} \rrbracket_k\|_\F^2 + \|\bm{E}\|_\F^2].
    \end{equation*}
\end{lemma}
\begin{proof}
    Note that
    \begin{equation*}
        \bm{B} \bm{\Omega} + \bm{E} \widetilde{\bm{\Omega}} = \begin{bmatrix} \bm{B} & \bm{E} \end{bmatrix} \begin{bmatrix} \bm{\Omega} \\ \widetilde{\bm{\Omega}} \end{bmatrix} = \bm{C} \bm{\Psi},
     \end{equation*}
     where $\bm{C} \in \mathbb{R}^{m \times (n+q)}$ and $\bm{\Psi} \sim \gaussian(n+q,\alpha k)$. Then, by \cite[Claim 1 and Corollary 7]{pcps}\footnote{Here we are using that $O(\frac{k+\log(1/\delta)}{\varepsilon^2}) \leq O(\frac{k\log(1/\delta)}{\varepsilon^2})$ and that if $\varepsilon' = \frac{1}{3} \varepsilon$, then $\frac{1+\varepsilon'}{1-\varepsilon'} \leq 1+\varepsilon$.} we have with probability at least $1-\delta$ 
     \begin{align*}
         \|\bm{B} - \bm{U}\bm{U}^\T \bm{B}\|_\F^2 \leq \|\bm{C} - \bm{U}\bm{U}^\T \bm{C}\|_\F^2 \leq (1+\varepsilon) \,\mathcal{E}_k(\bm{C})^2.
     \end{align*}
     The proof is concluded by noting
     \begin{equation*}
         \mathcal{E}_k(\bm{C})^2 \leq \left\|\bm{C}-\begin{bmatrix} \llbracket \bm{B}\rrbracket_k & \bm{0}\end{bmatrix} \right\|_\F^2 \leq \|\bm{B}  - \llbracket \bm{B}\rrbracket_k \|_\F^2 + \|\bm{E}\|_\F^2. 
     \end{equation*}
\end{proof}
With \Cref{thm:perturbation1} at hand, we are ready to prove an approximation guarantee with $\bsep$ matrices. 
Since \Cref{alg:butterfly_greedy_matvec} computes a $\butterfly$ approximation through recursive $\bsep$ approximation, this lemma will be essential for our final approximation guarantee. 
\begin{lemma}\label[lemma]{lemma:perturbation2}
    Let $\bm{A} \in \mathbb{R}^{2^{L+1} k \times 2^{L+1}k}$ and $\bm{E}, \bm{D}\in \mathbb{R}^{2^{L+1}k \times q}$ be fixed. 
    Draw independent $\bm{\Omega}, \bm{\Psi} \sim \gaussian(2^{L+1}k, \alpha k)$ and $\widetilde{\bm{\Omega}}, \widetilde{\bm{\Psi}} \sim \gaussian(q, \alpha k)$. Let $\bm{Y} = \bm{A} \bm{\Omega} + \bm{E}\widetilde{\bm{\Omega}}$ and $\bm{Z} = \bm{A}^\T\bm{\Psi} + \bm{D}^\T\widetilde{\bm{\Psi}}$. Partition $\bm{Y}$ and $\bm{Z}$ into blocks of size $2k \times \alpha k$:
    \begin{equation*}
        \bm{Y} = \begin{bmatrix}\bm{Y}_1 \\ \vdots \\ \bm{Y}_{2^L} \end{bmatrix}, \quad \bm{Z} = \begin{bmatrix}\bm{Z}_1 \\ \vdots \\ \bm{Z}_{2^L} \end{bmatrix}, \quad \bm{Y}_i, \bm{Z}_i \in \mathbb{R}^{2k \times  \alpha k}.
    \end{equation*}
    For $i = 1,\ldots,2^L$, let $\bm{U}_i$ and $\bm{V}_i$ be the top $k$ left singular vectors of $\bm{Y}_i$ and $\bm{Z}_i$, respectively, and define $\bm{U} = \blockdiag(\bm{U}_1,\ldots,\bm{U}_{2^L})$ and $\bm{V} = \blockdiag(\bm{V}_1,\ldots,\bm{V}_{2^L})$. 
    Let $\delta \in (0,1/2]$ and $\varepsilon \in (0,1]$. Then, if $\alpha  = O\left(\frac{L+\log(1/\delta)}{\varepsilon^2} \right)$, we have with probability at least $1-\delta$
    \begin{equation*}
        \|\bm{A} - \bm{U}\bm{U}^\T\bm{A} \bm{V}\bm{V}^\T\|_\F^2 \leq 2(1+\varepsilon) \min\limits_{\bm{C} \in \bsep(L,k)}\|\bm{A} - \bm{C}\|_\F^2 + (1+\varepsilon)\left[\|\bm{E}\|_\F^2 + \|\bm{D}\|_\F^2\right].
    \end{equation*}
\end{lemma}
\begin{proof}
    Let $I_{i,L}$ be as in \cref{def:index_sets}. We have
    \begin{align*}
        \|\bm{A} - \bm{U}\bm{U}^\T\bm{A} \bm{V}\bm{V}^\T\|_\F^2 &= \|\bm{A} - \bm{U}\bm{U}^\T \bm{A}\|_\F^2 + \|\bm{U}\bm{U}^\T \bm{A} - \bm{U}\bm{U}^\T\bm{A} \bm{V}\bm{V}^\T\|_\F^2 \\
        & \leq \|\bm{A} - \bm{U}\bm{U}^\T \bm{A}\|_\F^2 + \|\bm{A} - \bm{A} \bm{V}\bm{V}^\T\|_\F^2 \\
        & = \sum\limits_{i=1}^{2^L}\left[\|\bm{A}(I_{i,L},:) - \bm{U}_i\bm{U}_i^\T \bm{A}(I_{i,L},:)\|_\F^2 + \|\bm{A}(:,I_{i,L}) - \bm{A}(:,I_{i,L}) \bm{V}_i\bm{V}_i^\T\|_\F^2\right].
    \end{align*}
    Since $\alpha = O\left(\frac{L+ \log(1/\delta)}{\varepsilon^2}\right) = O\left(\frac{L+1+ \log(1/\delta)}{\varepsilon^2}\right) = O\left(\frac{ \log(2^{L+1}/\delta)}{\varepsilon^2}\right)$, by \Cref{thm:perturbation1} we have for each $i=1,\ldots,2^L$ with probability at least $1-2\delta/2^{L+1} = 1-\delta/2^L$,
    \begin{align}
        \begin{split}\label{eq:nearopt}
        \hspace{3em}&\hspace{-3em}\|\bm{A}(I_{i,L},:) - \bm{U}_i\bm{U}_i^\T \bm{A}(I_{i,L},:)\|_\F^2 + \|\bm{A}(:,I_{i,L}) - \bm{A}(:,I_{i,L}) \bm{V}_i\bm{V}_i^\T\|_\F^2\\
        &\leq (1+\varepsilon)\left[ \mathcal{E}_k(\bm{A}(I_{i,L},:))^2 + \mathcal{E}_k(\bm{A}(:,I_{i,L}))^2\right] + (1+\varepsilon)\left[\|\bm{E}(I_{i,L},:)\|_\F^2 + \|\bm{D}(:,I_{i,L})\|_\F^2\right].
    \end{split}
    \end{align}
    By a union bound over all $2^L$ blocks, \cref{eq:nearopt} holds for all $i = 1,\ldots,2^L$ with probability at least $1-\delta$. 
    Hence, 
    \begin{align*}
        \|\bm{A} - \bm{U}\bm{U}^\T\bm{A} \bm{V}\bm{V}^\T\|_\F^2 
        &\leq \|\bm{A} - \bm{U}\bm{U}^\T \bm{A}\|_\F^2 + \|\bm{A} - \bm{A} \bm{V}\bm{V}^\T\|_\F^2\\
        &=\sum\limits_{i=1}^{2^L}\left[\|\bm{A}(I_{i,L},:) - \bm{U}_i\bm{U}_i^\T \bm{A}(I_{i,L},:)\|_\F^2 + \|\bm{A}(:,I_{i,L}) - \bm{A}(:,I_{i,L}) \bm{V}_i\bm{V}_i^\T\|_\F^2\right]\\
        & \leq (1+\varepsilon) \sum\limits_{i=1}^{2^L}\left[\mathcal{E}_k(\bm{A}(I_{i,L},:))^2 + \mathcal{E}_k(\bm{A}(:,I_{i,L}))^2\right] \\
        &\hspace{3em}+ (1+\varepsilon) \sum\limits_{i=1}^{2^L} \left[\|\bm{E}(I_{i,L},:)\|_\F^2 + \|\bm{D}(:,I_{i,L})\|_\F^2\right]\\
        &\leq 2(1+\varepsilon) \min\limits_{\bm{C} \in \bsep(L,k)}\|\bm{A} - \bm{C}\|_\F^2 + (1+\varepsilon) \left[\|\bm{E}\|_\F^2 + \|\bm{D}\|_\F^2\right],
    \end{align*}
    holds with probability $1-\delta$, as required, where the final inequality uses \Cref{theorem:bsep}.
\end{proof}
Having established an approximation guarantee using the $\bsep(L, k)$ family, we can apply it across all levels to get a guarantee on the butterfly approximation produced by~\cref{alg:butterfly_greedy_matvec}. 
At every level of the recursion in \Cref{alg:butterfly_greedy_matvec}, the sketches will inherit errors due to $\bm{A}$ not being exactly butterfly. 
However, a key property that allows us to control these errors is due to that they are additive and independent. 
The following lemma gives an approximation guarantee of the approximation returned by \Cref{alg:butterfly_greedy_matvec} if the sketches are corrupted by additive and independent errors, and our main result will be a simple corollary.
Although the top-level sketches are exact in the algorithm, we state the lemma with arbitrary additive perturbations at every level so that the errors generated earlier in the recursion can be absorbed into the inductive hypothesis. 

\begin{lemma}\label[lemma]{lemma:preparation}
    Let $\bm{A} \in \mathbb{R}^{2^{L+1}k \times 2^{L+1}k}$ where $L$ is even. 
    Let $\alpha \geq 1$ and for $\ell = L,L-2\ldots,2,0$ define $\bm{\Omega}^{(\ell)}$ and $\bm{\Psi}^{(\ell)}$ as in \cref{eqn:sketch_def} and for $c = 1,\ldots,C_{\ell}$, where $\{C_{\ell}\}$ are fixed given numbers, let $\widetilde{\bm{\Omega}}^{(\ell,c)}$ and $\widetilde{\bm{\Psi}}^{(\ell,c)}$ be independent copies of $\bm{\Omega}^{(\ell)}$ and $\bm{\Psi}^{(\ell)}$, respectively. 
    Let  $\bm{E}^{(\ell,c)},\bm{D}^{(\ell,c)} \in \mathbb{R}^{2^{L+1}k \times 2^{L+1}k}$ be fixed matrices. 
    Define
\begin{equation*}
    \bm{Y}^{(\ell)} = \bm{A} \bm{\Omega}^{(\ell)} + \sum\limits_{c=1}^{C_{\ell}} \bm{E}^{(\ell,c)}\widetilde{\bm{\Omega}}^{(\ell,c)}, \quad \bm{Z}^{(\ell)} = \bm{A}^\T \bm{\Psi}^{(\ell)} + \sum\limits_{c=1}^{C_{\ell}} \bm{D}^{(\ell,c)\T}\widetilde{\bm{\Psi}}^{(\ell,c)}.
\end{equation*}
For all fixed choices of $\delta \in (0,1/2]$, $\varepsilon \in (0,1]$, $\{C_{\ell}\},\{\bm{E}^{(\ell,c)}\},\{\bm{D}^{(\ell,c)}\}$, if $\alpha = O\left(\frac{L +\log(1/\delta)}{\varepsilon^2} \right)$ then the output $\bm B$ of \Cref{alg:butterfly_greedy_matvec} satisfies with probability at least $1-\frac{L+2}{2}\delta$
\begin{equation*}
    \|\bm{A} - \bm{B}\|_\F^2 \leq 2((2+\varepsilon)^{\frac{L+2}{2}}-1) \min\limits_{\bm{C}\in \butterfly(L,k)}\|\bm{A} - \bm{C}\|_\F^2 + (1+\varepsilon) \sum\limits_{i=0}^{L/2} (2+\varepsilon)^i \sum\limits_{c=1}^{C_{2i}}\left[\|\bm{E}^{(2i,c)}\|_\F^2 + \|\bm{D}^{(2i,c)}\|_\F^2\right].
\end{equation*}
\end{lemma}
Before proceeding with the proof, we need a simple lemma. 
\begin{lemma}\label[lemma]{lemma:gaussians}
    Let $\alpha > 1$ and let $\bm{\Psi}_1 \sim \gaussian(\alpha k,k)$ let $\bm{\Psi}_2 \sim \gaussian(\alpha k,n)$ be independent and let $\bm{D}$ be a fixed matrix. Then, for $t \geq 0$ and $s \geq 1$, with probability at least $1-\exp\left(-\frac{t^2}{4(1+t)}\right) - s^{(1-\alpha)k}$ we have $\|\bm{\Psi}_1^{\dagger} \bm{\Psi}_2 \bm{D}\|_\F^2 \leq \frac{3s^2(1+t)}{\alpha-1}\|\bm{D}\|_\F^2$.
    In particular, $\|\bm{\Psi}_1^{\dagger} \bm{\Psi}_2 \bm{D}\|_\F^2 \leq \varepsilon^2 \|\bm{D}\|_\F^2$ with probability $ \geq 1-\delta \geq 1/2$ if $\alpha = O\left(\frac{\log(1/\delta)}{\varepsilon^2}\right)$.
\end{lemma}
\begin{proof}
    Conditioned on $\bm{\Psi}_1$, we have by \cite[Theorem 1]{kressnercortinovis} with probability at least $1-\exp\left(-\frac{t^2}{4(1+t)}\right)$ 
    \begin{equation*}
        \|\bm{\Psi}_1^{\dagger} \bm{\Psi}_2 \bm{D}\|_\F^2 = \|(\bm{D}^\T\otimes\bm{\Psi}_1^{\dagger}) \mathrm{vec}(\bm{\Psi}_2)\|_\F^2 \leq (1+t)\|\bm{D}^\T \otimes \bm{\Psi}_1^{\dagger}\|_\F^2 = (1+t)\|\bm{D}\|_\F^2 \|\bm{\Psi}_1^{\dagger}\|_\F^2.
    \end{equation*}
    The first result is then obtained by applying \cite[Proposition 10.4]{HalkoMartinssonTropp:2011} and the union bound. The second result is obtained by equating $\exp\left(-\frac{t^2}{4(1+t)}\right) = s^{(1-\alpha)k} = \delta/2$ to obtain $s^2 = (2/\delta)^{\frac{2}{(\alpha-1)k}}$, $1+t \leq 4(1+\log(2/\delta))$, and $\frac{3s^2(1+t)}{\alpha-1} \leq \frac{12}{\alpha-1}(1+\log(2/\delta))(2/\delta)^{\frac{2}{(\alpha-1)k}}$. The final term is upper bounded by $\varepsilon^2$ if $\alpha = 1 + \frac{12e(1+\log(2/\delta))}{\varepsilon^2}$. 
\end{proof}
We now proceed with the proof of \Cref{lemma:preparation}. 
\begin{proof}[Proof of \Cref{lemma:preparation}]
    The proof is by induction on $h := L/2$ and will be split into 4 steps.
    
    \paragraph{Step 1 (Base case):} Assume $h = 0$. We prove a perturbation bound for a special case of the generalized Nyström method. We apply the Pythagorean theorem to obtain
    \begin{align*}
        \|\bm{A} - \bm{U}\bm{X} \bm{V}^\T\|_\F^2 & = \|\bm{A} - \bm{U}\bm{U}^\T \bm{A}\|_\F^2 + \|\bm{U}^\T \bm{A} - \bm{M}\|_\F^2,
    \end{align*}
    where $\bm {M}$ is defined in the $L= 0$ case of~\cref{alg:butterfly_greedy_matvec}.
    Define
    \begin{equation*}
        \bm{D} = \begin{bmatrix} \bm{D}^{(0,1)} \\ \vdots \\ \bm{D}^{(0,C_0)}\end{bmatrix}, \quad \widetilde{\bm{\Psi}} = \begin{bmatrix} \widetilde{\bm{\Psi}}^{(0,1)} \\ \vdots \\ \widetilde{\bm{\Psi}}^{(0,C_0)} \end{bmatrix} \quad \Rightarrow \quad \bm{D}^\T \widetilde{\bm{\Psi}} = \sum\limits_{c=1}^{C_0} \bm{D}^{(0,c)\T} \bm{\Psi}^{(0,c)}.
    \end{equation*}
    Then,
    \begin{equation*}
        \bm{M} = (\bm{\Psi}^{(0)\T} \bm{U})^{\dagger} \bm{Z}^{(0)\T} = (\bm{\Psi}^{(0)\T} \bm{U})^{\dagger} \left[\bm{\Psi}^{(0)\T} \bm{A} + \widetilde{\bm{\Psi}}^\T \bm{D}\right].
    \end{equation*}
    Let $\bm{U}_{\bot}$ be an orthonormal basis for the orthogonal complement of $\range(\bm{U})$ and define $\bm{\Psi}_1 = \bm{\Psi}^{(0)\T} \bm{U}$ and $\bm{\Psi}_2 = \bm{\Psi}^{(0)\T} \bm{U}_{\bot}$. Note that, conditioned on $\bm{U}$, $\bm{\Psi}_1$ and $\bm{\Psi}_2$ are independent Gaussian random matrices, which are also independent of $\widetilde{\bm{\Psi}}$.
    By \cite[Proof of Lemma 5.1]{ChenDumanKelesHalikiasMuscoMuscoPersson:2025}  and the inequality $(x+y)^2 \leq 2x^2 + 2y^2$ for $x,y \geq 0$ we have
    \begin{equation}\label{eq2}
        \|\bm{U}^\T \bm{A} - \bm{M}\|_\F^2 = \|\bm{\Psi}_1^{\dagger} \bm{\Psi}_2 \bm{U}_{\bot}^\T\bm{A} + \bm{\Psi}_1^{\dagger} \widetilde{\bm{\Psi}}^\T\bm{D}\|_\F^2 \leq 2\|\bm{\Psi}_1^{\dagger} \bm{\Psi}_2 \bm{U}_{\bot}^\T\bm{A}\|_\F^2 + 2\|\bm{\Psi}_1^{\dagger} \widetilde{\bm{\Psi}}^\T\bm{D}\|_\F^2
    \end{equation}
    Since $\bm{\Psi}_1$ and $\bm{\Psi}_2$ have $\alpha k = O\left(\frac{k \log(1/\delta)}{\varepsilon^2}\right) = O\left(\frac{k \log(3/\delta)}{(0.4 \varepsilon)^2}\right)$, by the independence of all involved matrices, \Cref{lemma:gaussians}, and a union bound we have, conditioned on $\bm{U}$, that
    \begin{equation}\label{eq3}
    \begin{aligned}
        2\|\bm{\Psi}_1^{\dagger} \widetilde{\bm{\Psi}}^\T\bm{D}\|_\F^2 &\leq 2 \cdot 0.4^2 \cdot \varepsilon^2\|\bm{D}\|_\F^2 = 0.32\varepsilon^2 \sum\limits_{c=1}^{C_0}\|\bm{D}^{(0,c)}\|_\F^2
        \\
        2\|\bm{\Psi}_1^{\dagger} \bm{\Psi}_2 \bm{U}_{\bot}^\T\bm{A}\|_\F^2 &\leq 2 \cdot 0.4^2 \cdot \varepsilon^2\|\bm{U}_{\bot}^\T\bm{A}\|_\F^2  = 0.32\varepsilon^2 \|\bm{A} - \bm{U}\bm{U}^\T \bm{A}\|_\F^2, 
    \end{aligned}
    \end{equation}
    with probability at least $1-2\delta/3$.
    Define
    \begin{equation}\label{eq:perturbdef}
        \bm{E} = \begin{bmatrix} \bm{E}^{(0,1)} & \cdots & \bm{E}^{(0,C_0)} \end{bmatrix}, \quad \widetilde{\bm{\Omega}} = \begin{bmatrix} \widetilde{\bm{\Omega}}^{(0,1)} \\ \vdots \\ \widetilde{\bm{\Omega}}^{(0,C_0)} \end{bmatrix} \quad \Rightarrow \quad \bm{E} \widetilde{\bm{\Omega}} = \sum\limits_{c=1}^{C_0} \bm{E}^{(0,c)} \widetilde{\bm{\Omega}}^{(0,c)}.
    \end{equation}
    By \Cref{thm:perturbation1} with $\bm{E}$ and $\widetilde{\bm{\Omega}}$ as defined in \cref{eq:perturbdef}, we have with probability at least $1-\delta/3$
    \begin{equation}\label{eq1}
        \|\bm{A} - \bm{U}\bm{U}^\T \bm{A}\|_\F^2 \leq (1+0.4\varepsilon) \,\mathcal{E}_k(\bm{A})^2 + (1+0.4\varepsilon) \sum\limits_{c=1}^{C_0} \|\bm{E}^{(0,c)}\|_\F^2.
    \end{equation}
    Hence, by a union bound, combining \cref{eq1}, \cref{eq2}, and \cref{eq3} yields
    \begin{align*}
        \|\bm{A} - \bm{U} \bm{X} \bm{V}^\T \|_\F^2 &\leq (1+0.4\varepsilon)(1+0.32\varepsilon^2) \mathcal{E}_k(\bm{A})^2 + (1+0.4\varepsilon)(1+0.32\varepsilon^2)\sum\limits_{c=1}^{C_0}\|\bm{E}^{(0,c)}\|_\F^2 + 0.32\varepsilon^2 \sum\limits_{c=1}^{C_0}\|\bm{D}^{(0,c)}\|_\F^2,\\
        & \leq (1+\varepsilon) \mathcal{E}_k^2(\bm{A}) + (1+\varepsilon) \sum\limits_{c=1}^{C_0}\left[\|\bm{E}^{(0,c)}\|_\F^2 + \varepsilon \|\bm{D}^{(0,c)}\|_\F^2\right]\\
        & \leq 2(1+\varepsilon) \mathcal{E}_k^2(\bm{A}) + (1+\varepsilon) \sum\limits_{c=1}^{C_0}\left[\|\bm{E}^{(0,c)}\|_\F^2 + \varepsilon \|\bm{D}^{(0,c)}\|_\F^2\right]
    \end{align*}
    with probability at least $1-3\delta / 3 = 1-\delta$. 
    
    \paragraph{Step 2 (Inductive step -- Additive error of sketches):} Now assume that the results holds for $h = L/2 - 1$. We will prove it also holds for $h = L/2$.  
    The key observation is that the additional error introduced when passing sketches to the next level of the hierarchy has exactly the same form as the perturbation already appearing in the statement of the lemma. 
    As a result, the recursive sketches fit under the same error model, allowing the inductive hypothesis to be applied at the next level.
    We now verify this claim.

    Let $\bm{U}$ and $\bm{V}$ be the blockdiagonal bases computed in \Cref{alg:butterfly_greedy_matvec}. 
    Partition 
\begin{equation*}
    \bm{A} = \begin{bmatrix} \bm{A}_{1,1} & \bm{A}_{1,2} \\ \bm{A}_{2,1} & \bm{A}_{2,2} \end{bmatrix}, \quad \widehat{\bm{X}}:=\bm{U}^\T \bm{A} \bm{V} = \begin{bmatrix} \widehat{\bm{X}}_{1,1} & \widehat{\bm{X}}_{1,2} \\ 
                \widehat{\bm{X}}_{2,1} & \widehat{\bm{X}}_{2,2} \end{bmatrix}, \quad \bm{U}_{[s]}^\T \bm{A}_{s,t} \bm{V}_{[t]} = \widehat{\bm{X}}_{s,t} \in \mathbb{R}^{2^{L-1}k \times 2^{L-1}k}.
\end{equation*}
For $\ell = L-2,L-4,\ldots,0, c = 1,\ldots,C_{\ell}$ partition
\begin{equation*}
    \bm{E}^{(\ell,c)} = \begin{bmatrix} \bm{E}_{1,1}^{(\ell,c)} & \bm{E}_{1,2}^{(\ell,c)} \\
    \bm{E}_{2,1}^{(\ell,c)} & \bm{E}_{2,2}^{(\ell,c)} \end{bmatrix}, \quad \bm{D}^{(\ell,c)} = \begin{bmatrix} \bm{D}_{1,1}^{(\ell,c)} & \bm{D}_{1,2}^{(\ell,c)} \\
    \bm{D}_{2,1}^{(\ell,c)} & \bm{D}_{2,2}^{(\ell,c)} \end{bmatrix}, \quad \bm{Y}^{(\ell)} = \begin{bmatrix} \bm{Y}_{1,1}^{(\ell)} & \bm{Y}_{1,2}^{(\ell)} \\
    \bm{Y}_{2,1}^{(\ell)} & \bm{Y}_{2,2}^{(\ell)} \end{bmatrix}, \quad \bm{Z}^{(\ell)} = \begin{bmatrix} \bm{Z}_{1,1}^{(\ell)} & \bm{Z}_{2,1}^{(\ell)} \\
    \bm{Z}_{1,2}^{(\ell)} & \bm{Z}_{2,2}^{(\ell)} \end{bmatrix}.
\end{equation*}
Recall that  $\bm{U} = \blockdiag(\bm{U}_{[1]}, \bm{U}_{[2]})$, $\bm{V} = \blockdiag(\bm{V}_{[1]}, \bm{V}_{[2]})$, $\bm{\Omega}^{(\ell)} = \blockdiag(\bm{\Omega}_{[1]}^{(\ell)},\bm{\Omega}_{[2]}^{(\ell)})$, and $\widetilde{\bm{\Omega}}^{(\ell,c)} = \blockdiag(\widetilde{\bm{\Omega}}_{[1]}^{(\ell,c)},\widetilde{\bm{\Omega}}_{[2]}^{(\ell,c)})$, where the split is exactly in half. 
Let $\bm{V}_{\bot} = \blockdiag(\bm{V}_{1,\bot},\ldots,\bm{V}_{2^L,\bot}) = \blockdiag(\bm{V}_{[1],\bot}, \bm{V}_{[2],\bot})$ be an orthonormal basis for the orthogonal complement of $\bm{V}$. For $s, t \in \{1, 2\}$, $\widehat{\bm{Y}}_{s,t}$ from \Cref{alg:butterfly_greedy_matvec} has the form:
\begin{align*}
    \widehat{\bm{Y}}^{(\ell)}_{s,t}= \bm{U}_{[s]}^\T \bm{Y}_{s,t}^{(\ell)}
    &= \bm{U}_{[s]}^\T\bm{A}_{s,t} \bm{\Omega}_{[t]}^{(\ell)} + \sum\limits_{c=1}^{C_{\ell}}\bm{U}_{[s]}^\T\bm{E}^{(\ell,c)}_{s,t} \widetilde{\bm{\Omega}}^{(\ell,c)}_{[t]}\\
    &=\bm{U}_{[s]}^\T \bm{A}_{s,t} \bm{V}_{[t]}\cdot \bm{V}_{[t]}^\T \bm{\Omega}_{[t]}^{(\ell)} + \bm{U}_{[s]}^\T \bm{A}_{s,t} \bm{V}_{[t],\bot} \cdot \bm{V}_{[t],\bot}^\T \bm{\Omega}_{[t]}^{(\ell)} + \sum\limits_{c=1}^{C_{\ell}}\left[\bm{U}_{[s]}^\T\bm{E}^{(\ell,c)}_{s,t}\bm{V}_{[t]} \cdot \bm{V}_{[t]}^\T \widetilde{\bm{\Omega}}_{[t]}^{(\ell,c)} \right.\\
    &\hspace{15em}\left.+\bm{U}_{[s]}^\T\bm{E}^{(\ell,c)}_{s,t}\bm{V}_{[t],\bot} \cdot \bm{V}_{[t],\bot}^\T \widetilde{\bm{\Omega}}_{[t]}^{(\ell,c)}\right].
\end{align*}
We note that, conditioned on $\bm{\Omega}^{(L)}$ and $\bm{\Psi}^{(L)}$, $\bm{U}_{[s]}^\T \bm{A}_{s,t} \bm{V}_{[t]}\cdot \bm{V}_{[t]}^\T \bm{\Omega}_{[t]}^{(\ell)}$ is a sketch of $\bm{X}_{s,t}$ with a blockdiagonal Gaussian random matrix $\widehat{\bm{\Omega}}_{[t]}^{(\ell)}:=\bm{V}_{[t]}^\T \bm{\Omega}_{[t]}^{(\ell)}$. Furthermore, $\bm{U}_{[s]}^\T \bm{A}_{s,t} \bm{V}_{[t],\bot}, \bm{U}_{[s]}^\T\bm{E}^{(\ell,c)}_{s,t}\bm{V}_{[t]}$, and $\bm{U}_{[s]}^\T\bm{E}^{(\ell,c)}_{s,t}\bm{V}_{[t],\bot}$ are fixed matrices of size $2^{L-1}k \times 2^{L-1}k$. Moreover, $\widehat{\bm{\Omega}}_{[t]}^{(\ell)}, \bm{V}_{[t],\bot}^\T\bm{\Omega}_{[t]}^{(\ell)}, \bm{V}_{[t]}^\T \widetilde{\bm{\Omega}}_{[t]}^{(\ell,c)}$, and $\bm{V}_{[t],\bot}^\T \widetilde{\bm{\Omega}}_{[t]}^{(\ell,c)}$ are all \emph{independent} copies of blockdiagonal Gaussian matrices.
Hence, we can write
\begin{equation*}
    \widehat{\bm{Y}}_{s,t}^{(\ell)} = \widehat{\bm{X}}_{s,t}\widehat{\bm{\Omega}}_{[t]}^{(\ell)} + \sum\limits_{c=1}^{2C_{\ell}+1} \widehat{\bm{E}}_{s,t}^{(\ell,c)} \widetilde{\widehat{\bm{\Omega}}}_{[t]}^{(\ell,c)},
\end{equation*}
where, conditioned on $\bm{\Omega}^{(L)}$ and $\bm{\Psi}^{(L)}$, $\widehat{\bm{E}}_{s,t}^{(\ell,c)}$ are fixed matrices of size $2^{L-1}k \times 2^{L-1}k$, and the sketch matrices are \emph{independent} copies of blockdiagonal Gaussian matrices.
This illustrates that the additive structure of the errors are preserved in the recursion. 
Another important property that we will need is
\begin{align}
\begin{split}\label{eq:bound1}
    \sum\limits_{s,t=1}^{2} \sum\limits_{c=1}^{2C_{\ell}+1} \|\widehat{\bm{E}}_{s,t}^{(\ell,c)}\|_\F^2 &= \sum\limits_{s,t=1}^{2}\left[\|\bm{U}_{[s]}^\T \bm{A}_{s,t} \bm{V}_{[t],\bot}\|_\F^2+\sum\limits_{c=1}^{C_{\ell}} \left[\|\bm{U}_{[s]}^\T \bm{E}_{s,t}^{(\ell,c)} \bm{V}_{[t],\bot}\|_\F^2 + \|\bm{U}_{[s]}^\T \bm{E}_{s,t}^{(\ell,c)} \bm{V}_{[t]}\|_\F^2\right]\right]\\
    &= \|\bm{U}^\T \bm{A}(\bm{I} - \bm{V} \bm{V}^\T)\|_\F^2 + \sum\limits_{c=1}^{C_{\ell}}\|\bm{U}^\T\bm{E}^{(\ell,c)}\|_\F^2\\
    & \leq \|\bm{U}^\T \bm{A}(\bm{I} - \bm{V} \bm{V}^\T)\|_\F^2 + \sum\limits_{c=1}^{C_{\ell}}\|\bm{E}^{(\ell,c)}\|_\F^2.
\end{split}
\end{align}
By an identical argument, by interchanging the roles of $\bm{U}$ and $\bm{V}$, we can show that $\widehat{\bm{Z}}_{s,t}$ from \Cref{alg:butterfly_greedy_matvec} satisfies
\begin{equation*}
    \widehat{\bm{Z}}_{s,t} = \bm{V}_{[t]}^\T \bm{Z}^{(\ell)}_{s,t} = \widehat{\bm{X}}_{s,t}^\T \widehat{\bm{\Psi}}_{[s]}^{(\ell)} + \sum\limits_{c=1}^{2C_{\ell}+1} \widehat{\bm{D}}_{s,t}^{(\ell,c)\T} \widetilde{\widehat{\bm{\Psi}}}_{[s]}^{(\ell,c)},
\end{equation*}
where, conditioned on $\bm{\Omega}^{(L)}$ and $\bm{\Psi}^{(L)}$, $\widehat{\bm{E}}_{s,t}^{(\ell,c)}$ are fixed matrices of size $2^{L-1}k \times 2^{L-1}k$, and the sketch matrices are \emph{independent} copies of blockdiagonal Gaussian matrices. Moreover, the error terms satisfy
\begin{align}
\begin{split}\label{eq:bound2}
    \sum\limits_{s,t=1}^{2} \sum\limits_{c=1}^{2C_{\ell}+1} \|\widehat{\bm{D}}_{s,t}^{(\ell,c)}\|_\F^2 &\leq  \|(\bm{I}-\bm{U}\bm{U}^\T)\bm{A}\bm{V}\|_\F^2 + \sum\limits_{c=1}^{C_{\ell}}\|\bm{D}^{(\ell,c)}\|_\F^2.
\end{split}
\end{align}

\paragraph{Step 3 (Inductive step -- Recursive formula for approximation error):} Let $E(\bm{A},h,\{\bm{E}^{(\ell,c)}\},\{\bm{D}^{(\ell,c)}\}) = \|\bm{A} - \bm{B}\|_\F^2$ is the error of $\Cref{alg:butterfly_greedy_matvec}$ with $L = 2h$ levels and with additive errors in the sketches $\bm{Y}^{(\ell)}, \bm{Z}^{(\ell)}$ defined by the matrices $\bm{E}^{(\ell,c)}, \bm{D}^{(\ell,c)}$ for $\ell = 0,\ldots,L-2,L$ and $c = 1,\ldots,C_{\ell}$. 
We have
\begin{align*}
    E(\bm{A},h,\{\bm{E}^{(\ell,c)}\},\{\bm{D}^{(\ell,c)}\}) &= \|\bm{A} - \bm{U}\bm{X} \bm{V}^\T\|_\F^2 \\
    &= \|\bm{A} - \bm{U} \widehat{\bm{X}} \bm{V}^\T \|_\F^2 + \|\widehat{\bm{X}} - \bm{X}\|_\F^2 \tag{Pythagorean theorem}\\
    & =\|\bm{A} - \bm{U} \widehat{\bm{X}} \bm{V}^\T \|_\F^2+ \sum\limits_{s,t=1}^2\|\widehat{\bm{X}}_{s,t} - \bm{X}_{s,t}\|_\F^2\\
    &= \|\bm{A} - \bm{U} \widehat{\bm{X}} \bm{V}^\T \|_\F^2+ \sum\limits_{s,t=1}^2 E(\widehat{\bm{X}}_{s,t},h-1,\{\widehat{\bm{E}}_{s,t}^{(\ell,c)}\},\{\widehat{\bm{D}}_{s,t}^{(\ell,c)}\}), \numberthis \label{eq:errorbound}
\end{align*}
where $\widehat{\bm{E}}_{s,t}^{(\ell,c)}$ and $\widehat{\bm{D}}_{s,t}^{(\ell,c)}$ are as in \emph{Step 2}.
\paragraph{Step 4 (Inductive step -- Solving the recursion):} Define  $\OPT^2(\bm{A},h) := \min\limits_{\bm{C} \in \butterfly(2h,k)}\|\bm{A}-\bm{C}\|_\F^2$. Let $a_{h-1} = 2((2+\varepsilon)^{h}-1)$ and $b_i = (1+\varepsilon)(2+\varepsilon)^i$. Since $\alpha = O\left(\frac{L+\log(1/\delta)}{\varepsilon^2}\right) = O\left(\frac{L-2+ \log(4/\delta)}{\varepsilon^2}\right)$, conditioned on $\bm{\Omega}^{(L)}$ and $\bm{\Psi}^{(L)}$, by our inductive hypothesis we have with probability at least $1-\frac{L}{2}\frac{\delta}{4}$
\begin{equation}\label{eq:lowererrorbound}
     E(\widehat{\bm{X}}_{s,t},h-1,\{\widehat{\bm{E}}^{(\ell,c)}_{s,t}\}, \{\widehat{\bm{D}}^{(\ell,c)}_{s,t}\}) 
     \leq a_{h-1} \OPT^2(\widehat{\bm{X}}_{s,t},h-1) + \sum\limits_{i = 0}^{h-1} b_i\sum\limits_{c=1}^{2C_{2i}+1}\left[\|\widehat{\bm{E}}_{s,t}^{(2i,c)}\|_\F^2 + \|\widehat{\bm{D}}_{s,t}^{(2i,c)}\|_\F^2\right].
\end{equation}
By \eqref{eq:bound1} and \eqref{eq:bound2} we have
\begin{align}
\begin{split}\label{eq:errorsum}
\hspace{5em}&\hspace{-5em}\sum_{s,t=1}^2\sum\limits_{i = 0}^{h-1} b_i\sum\limits_{c=1}^{2C_{2i}+1}\left[\|\widehat{\bm{E}}_{s,t}^{(2i,c)}\|_\F^2 + \|\widehat{\bm{D}}_{s,t}^{(2i,c)}\|_\F^2\right] \\
& \leq \sum\limits_{i=0}^{h-1} b_i \left[\|\bm{U}^\T \bm{A} (\bm{I} - \bm{V}\bm{V}^\T)\|_\F^2 + \|(\bm{I} - \bm{U}\bm{U}^\T) \bm{A} \bm{V}\|_\F^2 + \sum\limits_{c=1}^{C_{2i}} \left[\| \bm{E}^{(2i,c)}\|_\F^2 + \|\bm{D}^{(2i,c)}\|_\F^2\right]\right]\\
& \leq \|\bm{A} - \bm{U}\bm{U}^\T \bm{A} \bm{V}\bm{V}^\T\|_\F^2 \sum\limits_{i=0}^{h-1} b_i +\sum\limits_{i=0}^{h-1} b_i\sum\limits_{c=1}^{C_{\ell}} \left[\| \bm{E}^{(2i,c)}\|_\F^2 + \|\bm{D}^{(\ell,c)}\|_\F^2\right]\\
& \leq \|\bm{A} - \bm{U}\widehat{\bm{X}}\bm{V}^\T\|_\F^2 \sum\limits_{i=0}^{h-1} b_i +\sum\limits_{i=0}^{h-1} b_i\sum\limits_{c=1}^{C_{2i}} \left[\| \bm{E}^{(2i,c)}\|_\F^2 + \|\bm{D}^{(2i,c)}\|_\F^2\right],
\end{split}
\end{align}
where we used $\|\bm{A} - \bm{U}\bm{U}^\T \bm{A}\bm{V}\bm{V}^\T\|_\F^2- \left(\|\bm{U}^\T \bm{A} (\bm{I} - \bm{V}\bm{V}^\T)\|_\F^2 + \|(\bm{I} - \bm{U}\bm{U}^\T) \bm{A} \bm{V}\|_\F^2\right) = \|(\bm{I}-\bm{U}\bm{U}^\T) \bm{A} (\bm{I} - \bm{V}\bm{V}^\T)\|_\F^2 \geq 0$.
Noting that by \Cref{theorem:contraction} we have 
\begin{equation}\label{eq:optsum}
    \sum\limits_{s,t=1}^2 \OPT^2(\widehat{\bm{X}}_{s,t},h-1)\leq \OPT^2(\bm{A},h).
\end{equation}
Conditioned on $\bm{\Omega}^{(L)}$ and $\bm{\Psi}^{(L)}$, by \Cref{eq:errorbound}, \Cref{eq:lowererrorbound}, \Cref{eq:errorsum}, and \Cref{eq:optsum} we have, by a union bound, with probability at least $1-4\frac{L}{2} \frac{\delta}{4} = 1-\frac{L}{2} \delta$
\begin{align*}
    E(\bm{A},h,\{\bm{E}^{(\ell,c)}\},\{\bm{D}^{(\ell,c)}\})&=\|\bm{A} - \bm{U} \widehat{\bm{X}} \bm{V}^\T \|_\F^2+ \sum\limits_{s,t=1}^2 E(\widehat{\bm{X}}_{s,t},h-1,\{\widehat{\bm{E}}_{s,t}^{(\ell,c)}\},\{\widehat{\bm{D}}_{s,t}^{(\ell,c)}\})\\
    &\leq \|\bm{A} - \bm{U} \widehat{\bm{X}} \bm{V}^\T \|_\F^2+ \sum\limits_{s,t=1}^2 \left[a_{h-1} \OPT^2(\widehat{\bm{X}}_{s,t},h-1) + \sum\limits_{i = 0}^{h-1} b_i\sum\limits_{c=1}^{2C_{2i}+1}\left[\|\widehat{\bm{E}}_{s,t}^{(2i,c)}\|_\F^2 + \|\widehat{\bm{D}}_{s,t}^{(2i,c)}\|_\F^2\right]\right]\\
    & \leq \left(1+\sum\limits_{i=0}^{h-1}b_i\right)\|\bm{A} - \bm{U} \widehat{\bm{X}} \bm{V}^\T \|_\F^2 + a_{h-1} \OPT^2(\bm{A},h) + \sum\limits_{i=0}^{h-1} b_i\sum\limits_{c=1}^{C_{2i}} \left[\| \bm{E}^{(2i,c)}\|_\F^2 + \|\bm{D}^{(2i,c)}\|_\F^2\right].
\end{align*}
By \Cref{lemma:perturbation2}, we have with probability $1-\delta$ that
    \begin{align*}
        \|\bm{A} - \bm{U}\widehat{\bm{X}}\bm{V}^\T\|_\F^2 \leq &2(1+\varepsilon) \min\limits_{\bm{C} \in \bsep(L,k)}\|\bm{A} - \bm{C}\|_\F^2 + (1+\varepsilon)\sum\limits_{c=1}^{C_L}\left[\|\bm{E}^{(L,c)}\|_\F^2 + \|\bm{D}^{(L,c)}\|_\F^2\right]\\
        \leq &2(1+\varepsilon) \OPT^2(\bm{A},h) + (1+\varepsilon)\sum\limits_{c=1}^{C_L}\left[\|\bm{E}^{(L,c)}\|_\F^2 + \|\bm{D}^{(L,c)}\|_\F^2\right]. \tag{$\butterfly(L,k) \subseteq \bsep(L,k)$}
    \end{align*}
Hence, by another union bound we have with probability $1-\delta - \frac{L}{2}\delta = 1-\frac{L+2}{2}\delta$
\begin{align*}
     E(\bm{A},h,\{\bm{E}^{(\ell,c)}\},\{\bm{D}^{(\ell,c)}\}) & \leq \left[2(1+\varepsilon)\left[1+\sum\limits_{i=0}^{h-1} b_i\right] + a_{h-1}\right]\OPT^2(\bm{A},h)\\
    &\hspace{2em}+ (1+\varepsilon)\left[1+\sum\limits_{i = 0}^{h-1} b_i\right]\sum\limits_{c=1}^{C_L}\left[\|\bm{E}^{(L,c)}\|_\F^2 + \|\bm{D}^{(L,c)}\|_\F^2\right] \\
    &\hspace{4em}+ \sum\limits_{i = 0}^{h-1} b_i \sum\limits_{c=1}^{C_{\ell}} \left(\| \bm{E}^{(\ell,c)}\|_\F^2 + \|\bm{D}^{(\ell,c)}\|_\F^2\right).
\end{align*}
Hence,
\begin{equation*}
    b_h = (1+\varepsilon) \left[1+\sum\limits_{i=0}^{h-1} b_i\right], \quad a_h = a_{h-1} + 2(1+\varepsilon)\left[1+\sum\limits_{i=0}^{h-1} b_i\right] = a_{h-1} + 2b_h.
\end{equation*}
Resolving the recursion with $b_0 = 1+\varepsilon$ and $a_0 = 2b_0$ yields
\begin{equation*}
    b_h = (1+\varepsilon)(2+\varepsilon)^h,
\end{equation*}
and
\begin{equation*}
    a_h = 2((2+\varepsilon)^h - 1) + 2(1+\varepsilon) (2+\varepsilon)^h = 2((2+\varepsilon)^{h+1} - 1),
\end{equation*}
as required.
\end{proof}
With \Cref{lemma:preparation} at hand, we are ready to prove our main result as a simple corollary. 
\begin{theorem}\label{theorem:matvecmain}
    Let $\bm{A} \in \mathbb{R}^{2^{L+1}k \times 2^{L+1}k}$, where $L$ is even. Fix $\varepsilon \in (0,1)$ and $\delta \in (0,1/2)$. Let $\bm{\Omega}^{(\ell)}, \bm{\Psi}^{(\ell)}$ be defined as in \cref{eqn:sketch_def} with $\alpha = O\left(\frac{L+\log(1/\delta)}{\varepsilon^2}\right)$. Then the output of \Cref{alg:butterfly_greedy_matvec} satisfies with probability at least $1-\delta$
\begin{equation*}
    \|\bm{A} - \bm{B}\|_\F^2 \leq 2((2+\varepsilon)^{\frac{L+2}{2}}-1) \min\limits_{\bm{C}\in \butterfly(L,k)}\|\bm{A} - \bm{C}\|_\F^2.
\end{equation*}
In particular, if $N = 2^{L+1}k$, \Cref{alg:butterfly_greedy_matvec} requires $O\left(\frac{\sqrt{Nk} \left(\log(N/k) + \log(1/\delta)\right)}{\varepsilon^2}\right)$ matvecs with $\bm{A}$ and $\bm{A}^\T$, and returns a matrix $\bm{B} \in \butterfly(L,k)$ satisfying with probability at least $1-\delta$
\begin{equation}\label{eq:upperbound}
    \|\bm{A} - \bm{B}\|_\F \leq O\left(\left(\frac{N}{k}\right)^{\frac{1}{4}(1 + \varepsilon)}\right) \min\limits_{\bm{C} \in \butterfly(L,k)}\|\bm{A} - \bm{C}\|_\F.
\end{equation}
\end{theorem}
\begin{proof}
By \Cref{lemma:preparation} with $\bm{E}^{(\ell,c)} = \bm{0}, \bm{D}^{(\ell,c)} = \bm{0}$ and $\alpha = O\left(\frac{L+\log(1/\widetilde{\delta})}{\varepsilon^2}\right)$ we have with probability at least $1-\frac{L+2}{2}\widetilde{\delta}$
\begin{equation*}
    \|\bm{A} - \bm{B}\|_\F^2 \leq 2((2+\varepsilon)^{\frac{L+2}{2}}-1) \min\limits_{\bm{C}\in \butterfly(L,k)}\|\bm{A} - \bm{C}\|_\F^2.
\end{equation*}
Letting $\widetilde{\delta} = \frac{2}{L+2} \delta$ and noting that $O(L+\log(1/\widetilde{\delta})) = O(L + \log(1/\delta))$ yields the desired result. The second statement follows from the fact that $2((2+\varepsilon)^{\frac{L+2}{2}}-1) = O\left(\left(\frac{N}{k}\right)^{\frac{1}{2}(1 + \varepsilon)}\right)$.
\end{proof}
Consequently, for $\varepsilon = O\left(1/L\right)$ the quasi-optimality constant becomes $O\left((N/k)^{1/4}\right)$, which is larger than the $\sqrt{\log_2(N/k)}$ quasi-optimality constants appearing in \Cref{theorem:meta_butterfly}.
However, for reasonable $N$ \eqref{eq:upperbound} still provides a strong guarantee. 
To illustrate, for $N = 10^7, k = 1, \varepsilon = 0.3$, the upper bound in \Cref{theorem:matvecmain} shows that the approximation error $\|\bm{A} - \bm{B}\|_\F \leq 221 \min\limits_{\bm{C} \in \butterfly(L,k)}\|\bm{A} - \bm{C}\|_\F$. 
Hence, \Cref{alg:butterfly_greedy_matvec} returns an approximation whose approximation error is guaranteed to be at most roughly 2 orders of magnitude larger than the best possible approximation. 

%% file: imgs/sketch_ex.tex
\newcommand{\drawOmega}[1]{%
  \pgfmathtruncatemacro{\N}{round(2^(#1))}
  \pgfmathtruncatemacro{\Nm}{\N-1}%
  \def\W{\N}%
  \def\H{32}%
  \pgfmathsetmacro{\bh}{\H/\N}%
  \begin{tikzpicture}[scale=.13]
    \ifnum\N>1\relax
      \foreach \i in {1,...,\Nm}{
        \draw[line width=.2px,black!20] (0,{\i*\bh}) -- (\W,{\i*\bh});
        \draw[line width=.2px,black!20] (\i,0) -- (\i,\H);
      }
    \fi
    \draw (0,0) rectangle (\W,\H);
    \foreach \i in {1,...,\N}{
      \draw[fill=black!10] (\i-1,{\H-(\i-1)*\bh}) rectangle (\i,{\H-\i*\bh});
    }
    \ifnum\N>1\relax
      \draw (0,{\H/2}) -- (\W,{\H/2});
      \draw ({\W/2},0) -- ({\W/2},\H);
    \fi
    \node[scale=.9,above] at ({\W/2},{\H+0.2}){$\bm{\Omega}^{(#1)}$};
  \end{tikzpicture}}

\newcommand{\drawOmegaSplit}[1]{%
  \pgfmathtruncatemacro{\N}{round(2^(#1))}%
  \pgfmathtruncatemacro{\Nm}{\N-1}%
  \def\W{\N}%
  \def\H{32}%
  \pgfmathsetmacro{\bh}{\H/\N}%
  \begin{tikzpicture}[scale=.13]
    \ifnum\N>1\relax
      \foreach \i in {1,...,\Nm}{
        \draw[line width=.2px,black!20] (0,{\i*\bh}) -- (\W,{\i*\bh});
        \draw[line width=.2px,black!20] (\i,0) -- (\i,\H);
      }
    \fi
    \draw (0,0) rectangle (\W,\H);
    \foreach \i in {1,...,\N}{
      \draw[fill=black!10] (\i-1,{\H-(\i-1)*\bh}) rectangle (\i,{\H-\i*\bh});
    }
    \ifnum\N>1\relax
      \draw (0,{\H/2}) -- (\W,{\H/2});
      \draw ({\W/2},0) -- ({\W/2},\H);
      \draw[fill=black!20,fill opacity=.9] (0,\H) rectangle ({\W/2},{\H/2})
        node[midway] {$\bm{\Omega}_{[1]}^{(#1)}$};
      \draw[fill=black!20,fill opacity=.9] ({\W/2},{\H/2}) rectangle (\W,0)
        node[midway] {$\bm{\Omega}_{[2]}^{(#1)}$};
      \draw[fill=white,fill opacity=.9] (0,{\H/2}) rectangle ({\W/2},0);
      \draw[fill=white,fill opacity=.9] ({\W/2},\H) rectangle (\W,{\H/2});
    \fi
    \node[scale=.9,above] at ({\W/2},{\H+0.2}){%
      \ifnum\N>1\relax
        $\blockdiag(\bm{\Omega}_{[1]}^{(#1)},\bm{\Omega}_{[2]}^{(#1)})%
         \vphantom{\left(\bm{\Omega}_1^{(#1)},\ldots,\bm{\Omega}^{(#1)}_{2^{#1}}\right)}$%
      \else
        $\bm{\Omega}_1^{(#1)}$%
      \fi};
  \end{tikzpicture}}
  
\drawOmega{0}
\hspace{.3cm}
\drawOmega{1}
\hspace{.3cm}
\drawOmega{2}
\hspace{.3cm}
\drawOmega{3}
\hspace{.3cm}
\drawOmega{4}
\hspace{1cm}
\drawOmegaSplit{4}

%% file: numerical.tex
\section{Numerical experiments} \label{sec:experiments}

We now provide computational examples to contextualize our theoretical results. In addition, we describe several extensions of the proposed algorithms from square matrices and dyadic index trees to general rectangular matrices and arbitrary index trees. These extensions are implemented in our open source Julia package (see \cref{fn:software}).

\subsection{Illustration of theoretical results}

First, we consider the case of an exactly butterfly matrix $\bm{A}_0 \in \R^{N \times N}$ which is corrupted by additive noise. To generate such a matrix, we sample all basis matrices $\bm{U}_i, \bm{V}_i \in \mathbb{R}^{2k \times k}$ in \Cref{theorem:recursive} uniformly from the set of matrices with orthonormal columns. At level $\ell = 0$, the core matrices $\bm{X} \in \mathbb{R}^{k \times k}$ are sampled from $\gaussian(k,k)$. In addition, we sample an additive error term $\bm{E} \sim \gaussian(N,N)$. We then define $\bm{A} = \frac{1}{\|\bm{A}_0\|_\F} \bm{A}_0 + \frac{\sigma}{\|\bm{E}\|_\F} \bm{E}$ with noise level $\sigma = 10^{-4}$. Finally, we compute butterfly approximations $\bm{B} \in \butterfly(L,k)$ to $\bm{A}$ using both the entry evaluation \cref{alg:butterfly_greedy_meta} and by computing matrix-vector products with $\bm{A}$ and using the randomized \cref{alg:butterfly_greedy_matvec} with various oversampling rates $\alpha$. The resulting recovery error $\|\bm{A} - \bm{B}\|_\F$ for various matrix sizes $N$ are shown in \cref{fig:butterfly-plus-noise}.

Several aspects of \cref{fig:butterfly-plus-noise} are noteworthy. First, we see in \cref{fig:random-errors-alpha} that as $\alpha$ increases, the recovery error in the randomized matvec algorithm converges rapidly to the error achieved by the entry evaluation algorithm. Furthermore, from \cref{fig:random-errors-N} we note that the error achieved by the entry evaluation algorithm approaches the lower bound on the optimal butterfly approximation error obtained by applying \cref{theorem:bsep} at each level $\ell=0,\dots,L$ and taking the maximum. Therefore, the entry evaluation algorithm and the matvec algorithm with sufficiently large $\alpha$ are both close to optimal for this family of test matrices. In addition, we note that by choosing the oversampling rate to be $\alpha = \big(L + \log(1 / \delta)\big) / \varepsilon^2$ as suggested by \cref{theorem:matvecmain}, the errors incurred by the matvec recovery algorithm are within a small factor of those incurred by the entry evaluation algorithm, suggesting that increasing the oversampling rate logarithmically with $N$ is indeed sufficient to control the recovery error. Finally, we see that the growth rates of the upper bounds \cref{theorem:meta_butterfly} and \cref{theorem:matvecmain} are both highly pessimistic for these particular test matrices. However, we emphasize that there may exist adversarial families of matrices for which these bounds are much tighter, or for which these asymptotic growth rates are observed.

\begin{figure}
    \centering
    \begin{subfigure}{0.28\textwidth}
        \centering
        \includegraphics[width=0.8\linewidth]{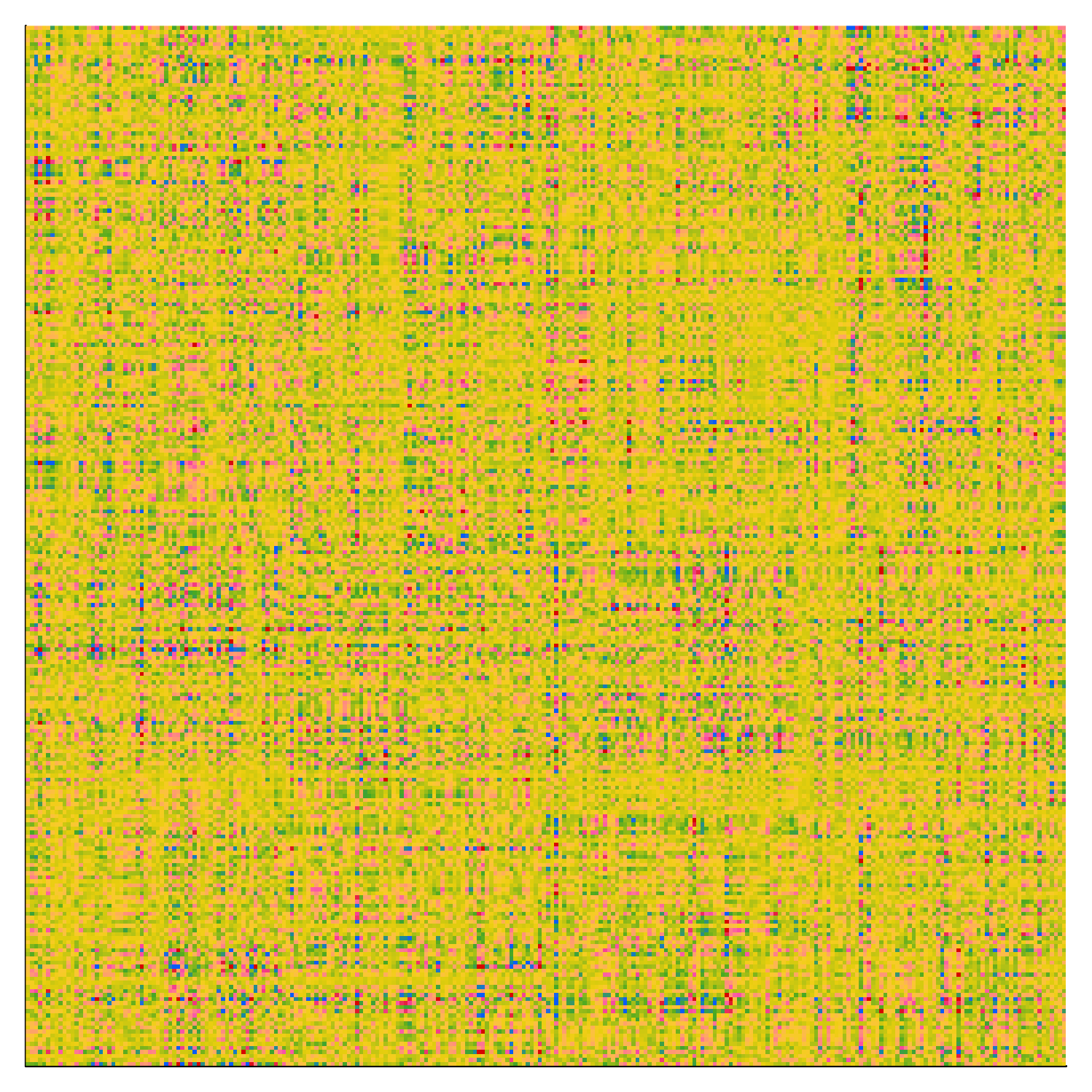}
        \caption{Random noisy butterfly matrix}
        \label{fig:random-matrix}
    \end{subfigure}
    \hfill
    \begin{subfigure}{0.35\textwidth}
        \centering
        \begin{overpic}[width=\linewidth]{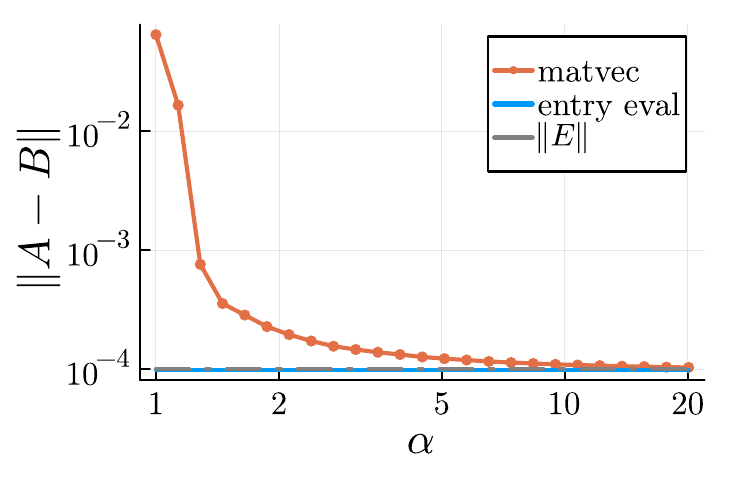}
        \put(-2,20){
            \rotatebox{90}{
                \colorbox{white}{$\|\bm{A}-\bm{B}\|_{\F}$}
            }
        }
        \end{overpic}
        \caption{Recovery error vs. $\alpha$ for $N = 1024$}
        \label{fig:random-errors-alpha}
    \end{subfigure}
    \hfill
    \begin{subfigure}{0.35\textwidth}
        \centering
        \begin{overpic}[width=\linewidth, trim={0 0.8cm 0 0}, clip]{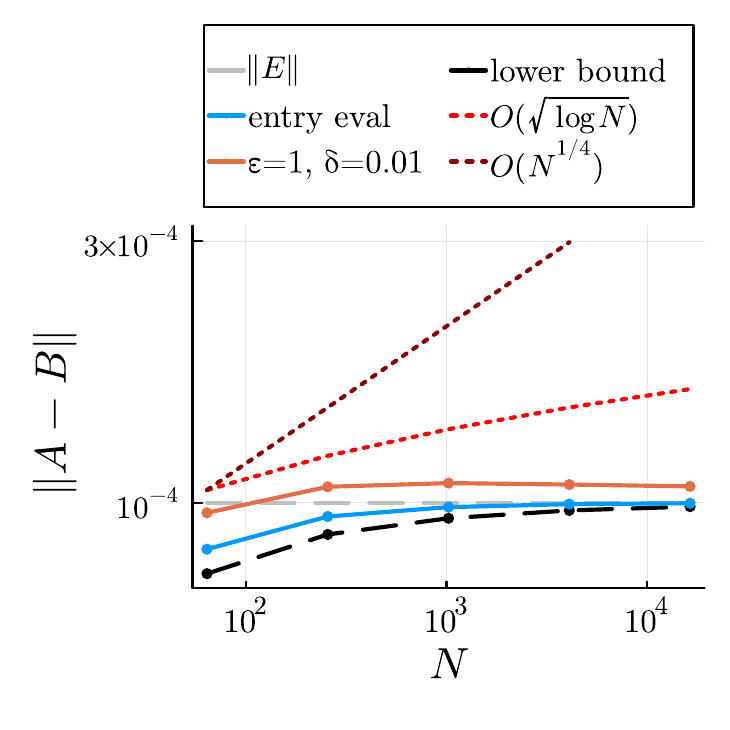}
        \put(0,20){
            \rotatebox{90}{
                \colorbox{white}{$\|\bm{A}-\bm{B}\|_{\F}$}
            }
        }
        \end{overpic}
        \caption{Recovery error vs. $N$}
        \label{fig:random-errors-N}
    \end{subfigure}

    \caption{Performance of \cref{alg:butterfly_greedy_meta} (labeled ``entry eval'') and \cref{alg:butterfly_greedy_matvec} (labeled ``matvec'') for random butterfly matrices with $k=8$ and additive noise $\sigma = 10^{-4}$. (\textbf{a}) Example random noisy butterfly matrix $\bm{A}$ with $L=4$. (\textbf{b}) Log-log plot of recovery error of \cref{alg:butterfly_greedy_matvec} for various oversampling rates $\alpha$ with fixed $N = 1024$. Error for \cref{alg:butterfly_greedy_meta} is also plotted for reference. (\textbf{c}) Log-log plot of recovery error for both algorithms. For \cref{alg:butterfly_greedy_matvec}, we take $\alpha = \big(L + \log(1 / \delta)\big) / \epsilon^2$ with $\delta = 0.01$ and $\epsilon = 1$. In addition, we display $O\big(\sqrt{\log N}\big)$ and $O(N^{1/4})$ reference lines corresponding to the upper bound growth rates in \cref{theorem:meta_butterfly} and \cref{theorem:matvecmain} respectively. Finally, we plot a lower bound on the optimal butterfly approximation error given by applying \cref{theorem:bsep} at each level $\ell=0,\dots,L$ and taking the maximum.}
    \label{fig:butterfly-plus-noise}
\end{figure}

\subsection{Generalizations and applications}

In order to apply Algorithms~\ref{alg:butterfly_greedy_meta} and~\ref{alg:butterfly_greedy_matvec} to more general operators appearing in applications, we describe several extensions which are standard practice in existing algorithms such as~\cite{oneil2010algorithm,liu2021sparse}. First, we note that by replacing transposes with Hermitian transposes, both algorithms can be used to compress complex matrices $\bm{A} \in \C^{N \times N}$.

We can also relax the assumption that $N = 2^{L+1}k$ for some even $L$ and fixed $k$. In practice, all that is required to run our recovery algorithms is a tree structure on the row and column indices for which $\rank\!\big(\bm{A}(I_{i,\ell},J_{j,L-\ell})\big) \leq k$ for all level-$\ell$ row index sets $I_{i,\ell}$ and level-$(L-\ell)$ column index sets $J_{j,L-\ell}$ for $\ell=0,\dots,L$. This tree structure should be constructed according to the properties of $\bm{A}$, and need not be a perfect dyadic tree. For example, if $\bm{A}_{pq} = K(\bm{x}_p, \bm{\omega}_q)$ is a kernel matrix defined by $\{\bm{x}_p\}_{p=1}^M \subset D_x$ and $\{\bm{\omega}_q\}_{q=1}^N \subset D_\omega$, then the index trees can be those induced by a partitioning of the domains $D_x$ and $D_\omega$. In particular, two different index trees of the same depth can be used for the rows and columns, which allows us to compress rectangular matrices $\bm{A} \in \C^{M \times N}$.

Furthermore, we do not have to assume the rank $k$ is fixed and known. One can instead determine $k$ separately for each block at each level to be the number of singular values above some tolerance $\epsilon$. In the entry evaluation context this can be achieved by computing a column pivoted QR factorization or truncated SVD of the corresponding block of $\bm{A}$. In the matvec setting, randomized error estimators have been developed to adaptively select $k$~\cite[Algorithm 4.2]{HalkoMartinssonTropp:2011}. Selecting block ranks adaptively in this way leads to basis matrices $\bm{U}_i \in \R^{(k_1 + k_2) \times k}$, where $k_1$ and $k_2$ are the ranks of the children of the current block in the row index tree. An analogous claim holds for $\bm{V}_i$ involving the ranks of children in the column index tree. The block sizes of $\bm{U}$ and $\bm{V}$ in the recursive format are thus given by summing adjacent pairs of ranks from the previous level. The accompanying software provides the option to specify either a fixed rank $k$ or a tolerance $\epsilon$.

To demonstrate the applicability of our algorithms to practical problems, we start by computing a recursive butterfly factorization of the nonuniform discrete Fourier transform (NUDFT) matrix $\bm{A} \in \C^{M \times N}$ with entries $\bm{A}_{pq} = e^{i\omega_q x_p}$ with $x_p \sim \text{Unif}(0,2\pi)$ for $p = 1, \dots, M = N/2$ and $\omega_q = q-1$ for $q = 1, \dots, N$. For each $N$, we take the level $L$ to be the largest even integer such that $2^L \leq \min(M,N)$. We then define trees on the row and columns index sets (with the $\{x_q\}$ sorted), each with depth $L$ and leaves of approximately equal sizes. We use these trees for both algorithms. For each basis $\bm{U}_i$ and $\bm{V}_i$, we retain the singular vectors corresponding to singular values greater than tolerance $\epsilon = 10^{-4}$. In the matvec algorithm, we take the sketch matrices $\bm{\Omega}^{(\ell)}_i$ and $\bm{\Psi}^{(\ell)}_i$ at all levels to have 40 columns, and employ an analysis-based nonuniform fast Fourier transform (NUFFT)~\cite{barnett2019parallel} which computes matvecs with $\bm{A}$ in $O(M + N\log(N))$ time for $\max x_q = O(1)$ and $\max \omega_q = O(N)$. The runtimes of both algorithms on a 3.2 GHz Apple M1 Pro CPU, as well as the final factorization sizes in memory are shown in \cref{fig:nufft}. We observe the expected $O(N^2)$ and $O(N^{3/2})$ scaling in the runtimes of \cref{alg:butterfly_greedy_meta} and \cref{alg:butterfly_greedy_matvec} respectively, as well as the $O(N\log(N))$ size of the butterfly factorization in memory. The \texttt{FINUFFT} library~\cite{barnett2019parallel} is highly optimized, and therefore in \cref{fig:nufft-runtime} we note that sketching is generally less computationally intensive than factorizing the matrix. In order to verify the accuracy of the factorizations, we sample a random matrix $\bm{\Omega} \sim \gaussian(N, 10)$ and compute $\|\bm{A}\bm{\Omega} - \bm{B}\bm{\Omega}\|_\F / \|\bm{A}\bm{\Omega}\|_\F$, which was at most $1.5 \times 10^{-4}$ for all tested matrices.

Finally, we perform the same experiments for a Hankel transform matrix $\bm{A} \in \R^{M \times N}$ with entries $\bm{A}_{pq} = J_0(\omega_q x_p)$, where $J_0$ is the Bessel function of the first kind of order zero. We take $x_p = (p-1)^2/(M-1)^2$ for $p = 1, \dots, M = 2N$ and $\omega_q = q-1$ for $q = 1, \dots, N$ as before. As above, we form approximately dyadic trees on the row and column index sets of depth $L$, for $2^L \leq \min(M,N)$. Due to the rank deficiency of $\bm{A}$, we require sketches $\bm{\Omega}^{(\ell)}_i$ and $\bm{\Psi}^{(\ell)}_i$ with only 20 columns. As before we use tolerance $\epsilon = 10^{-4}$ when computing basis matrices. We then use an analysis-based nonuniform fast Hankel transform (NUFHT)~\cite{beckman2026nonuniform} to compute matvecs with $\bm{A}$ in quasilinear time in $N$ and $M$. \cref{fig:nufht} shows the resulting runtimes and factorization sizes, which again agree with the asymptotic analysis. However, because of the relative difficulty of the Hankel transform compared to the Fourier transform and the fact that the \texttt{FastHankelTransform.jl} library~\cite{beckman2026nonuniform} is less optimized than \texttt{FINUFFT}, sketching takes up significantly more computational effort than recovering the factorization from the sketches. In such settings where matvecs are relatively expensive, minimizing the number of sketches required to recover a factorization of a given accuracy $\epsilon$ becomes a high priority in practice. The estimated error $\|\bm{A}\bm{\Omega} - \bm{B}\bm{\Omega}\|_\F / \|\bm{A}\bm{\Omega}\|_\F$ for $\bm{\Omega} \in \gaussian(N, 10)$ was at most $1.8 \times 10^{-4}$ for all tested matrices.

\begin{figure}
    \centering
    \begin{subfigure}{0.28\textwidth}
        \centering
        \includegraphics[width=\linewidth, trim={0 -15cm 0 0}]{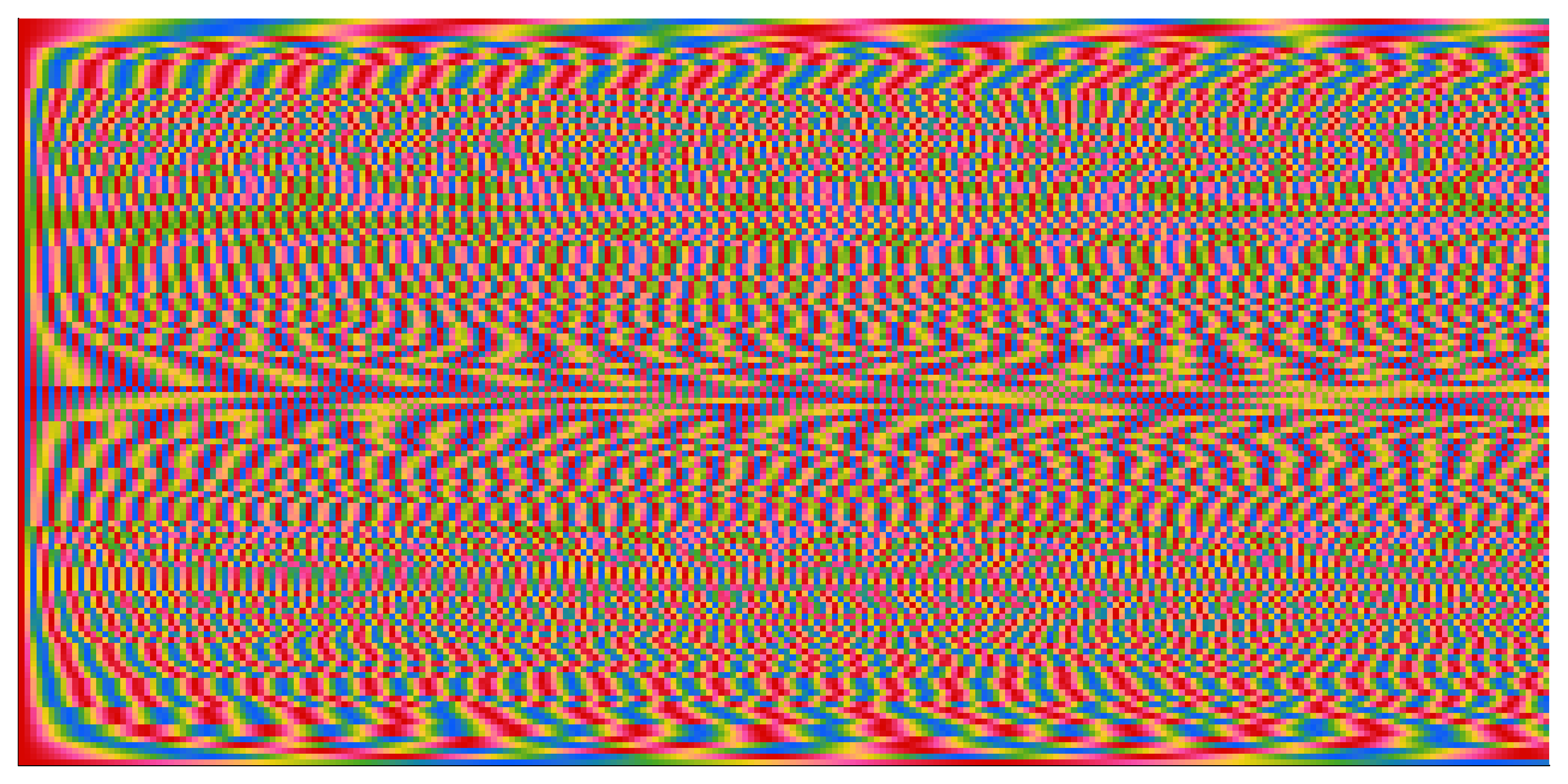}
        \caption{Real part of NUDFT matrix}
        \label{fig:nufft-matrix}
    \end{subfigure}
    \hfill
    \begin{subfigure}{0.35\textwidth}
        \centering
        \includegraphics[width=0.8\linewidth, trim={0 0.5cm 0 0}, clip]{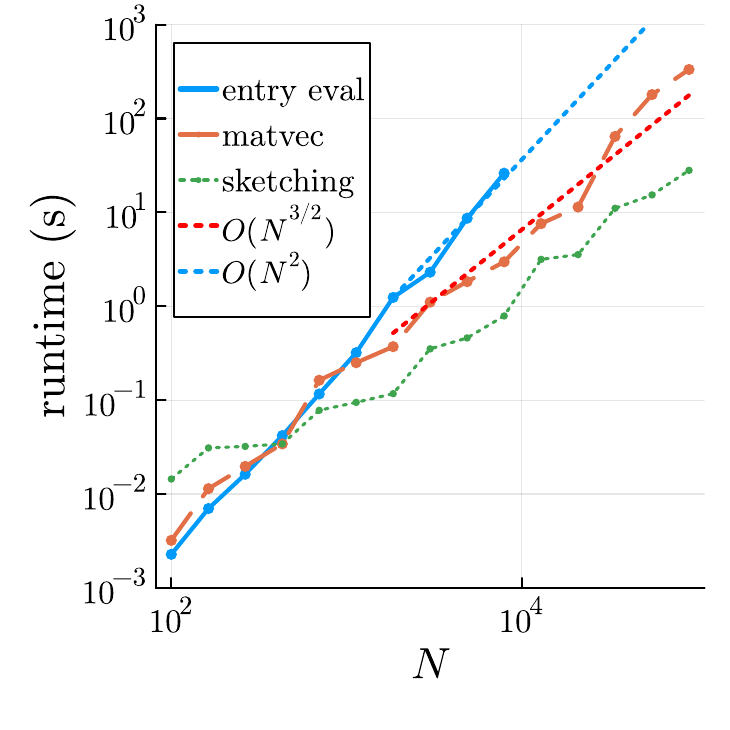}
        \caption{Runtime vs. $N$}
        \label{fig:nufft-runtime}
    \end{subfigure}
    \hfill
    \begin{subfigure}{0.35\textwidth}
        \centering
        \includegraphics[width=0.8\linewidth, trim={0 0.5cm 0 0}, clip]{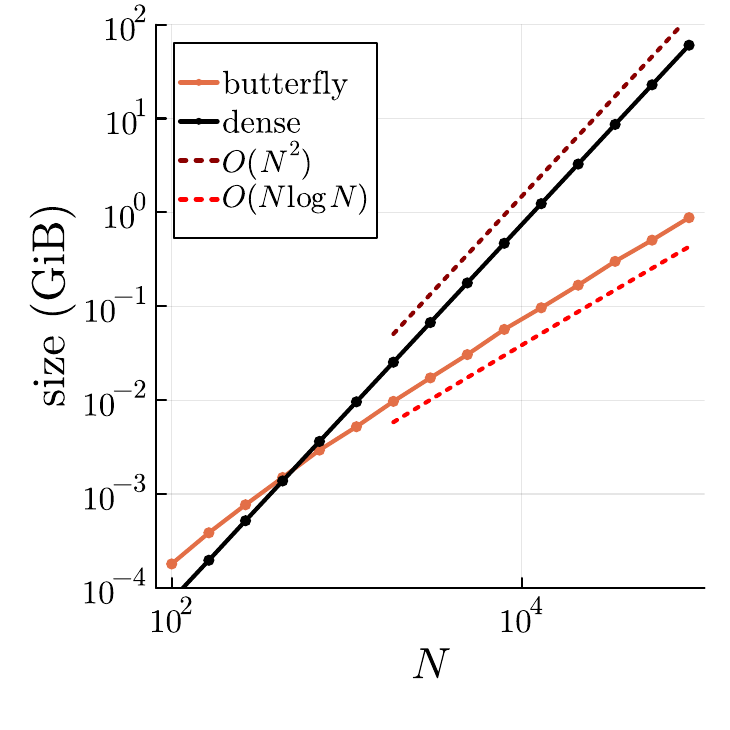}
        \caption{Factorization size vs. $N$}
        \label{fig:nufft-memory}
    \end{subfigure}
    \caption{Performance of \cref{alg:butterfly_greedy_meta} (labeled ``entry eval'') and \cref{alg:butterfly_greedy_matvec} (labeled ``matvec'') for NUDFT matrices using tolerance $\epsilon = 10^{-4}$ and 40 sketch vectors to recover each basis. (\textbf{a}) Example NUDFT matrix with $N = 256$. (\textbf{b}) Log-log plot of runtime of both algorithms as well as the time required to compute the sketches using the NUFFT. (\textbf{c}) Log-log plot of size of the butterfly factorization recovered by \cref{alg:butterfly_greedy_matvec} in memory compared to dense matrix.}
    \label{fig:nufft}
\end{figure}

\begin{figure}
    \centering
    \begin{subfigure}{0.28\textwidth}
        \centering
        \includegraphics[width=0.5\linewidth, trim={0 0cm 0 0}]{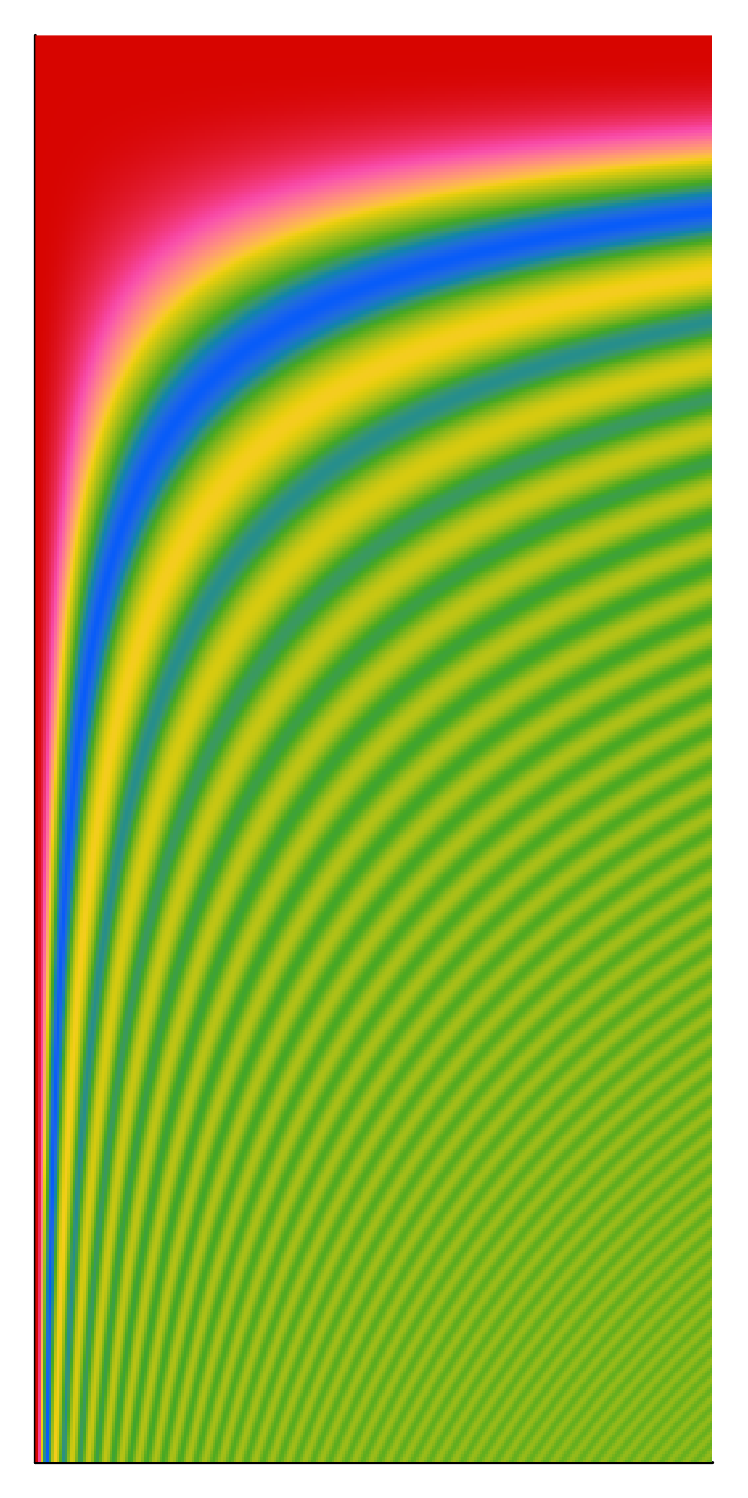}
        \caption{Hankel transform matrix}
        \label{fig:hankel-matrix}
    \end{subfigure}
    \hfill
    \begin{subfigure}{0.35\textwidth}
        \centering
        \includegraphics[width=0.8\linewidth, trim={0 0.5cm 0 0}, clip]{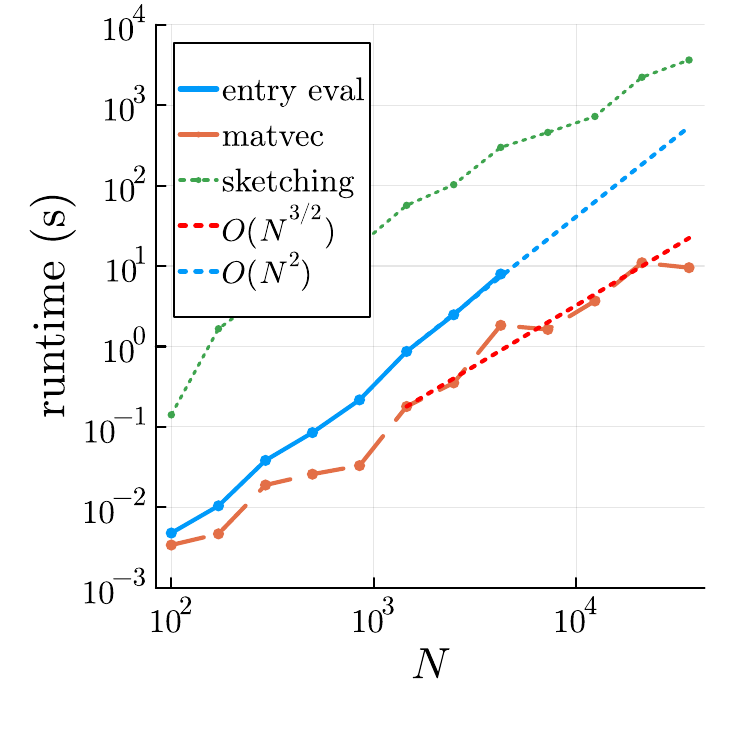}
        \caption{Runtime vs. $N$}
        \label{fig:hankel-runtime-N}
    \end{subfigure}
    \hfill
    \begin{subfigure}{0.35\textwidth}
        \centering
        \includegraphics[width=0.8\linewidth, trim={0 0.5cm 0 0}, clip]{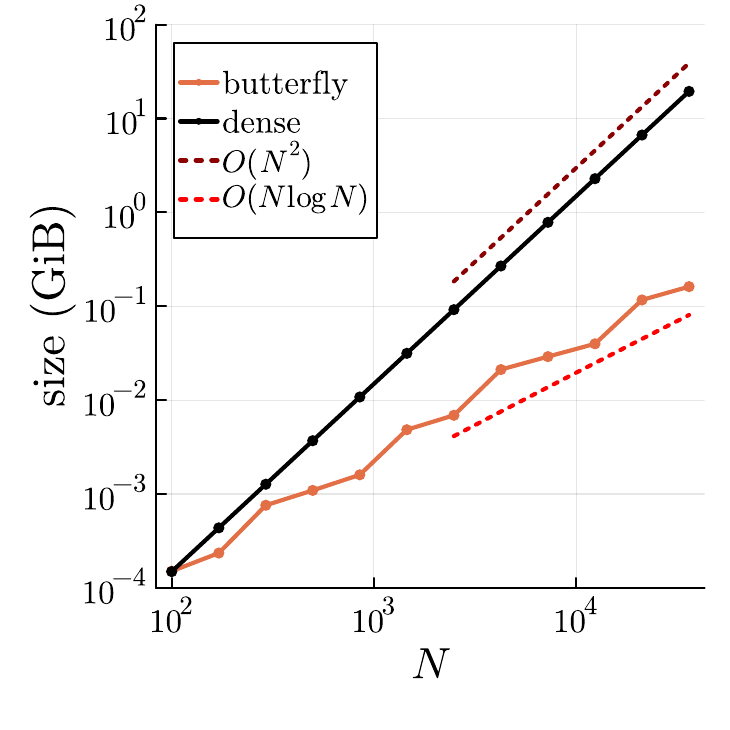}
        \caption{Factorization size vs. $N$}
        \label{fig:hankel-size-N}
    \end{subfigure}
    \caption{Performance of \cref{alg:butterfly_greedy_meta} (labeled ``entry eval'') and \cref{alg:butterfly_greedy_matvec} (labeled ``matvec'') for Hankel transform matrices using tolerance $\epsilon = 10^{-4}$ and 20 sketch vectors to recover each basis. (\textbf{a}) Example Hankel transform matrix with $N = 256$. (\textbf{b}) Log-log plot of runtime of both algorithms as well as the time required to compute the sketches using the NUFHT. (\textbf{c}) Log-log plot of size of the butterfly factorization recovered by \cref{alg:butterfly_greedy_matvec} in memory compared to dense matrix.}
    \label{fig:nufht}
\end{figure}

%% file: conclusion.tex
\section{Conclusion}

By formalizing the recursive structure of butterfly matrices suggested in~\cite{Pang2020}, we provide an easy-to-implement reformulation of the class of butterfly matrices and prove that it is equivalent to the complementary low-rank property employed in existing constructions. We use this recursive format to prove a near-optimality guarantee for an entry evaluation algorithm which matches the results of~\cite{Le2025}, and prove a novel near-optimality guarantee for the matvec query algorithm proposed in~\cite{LiuXingGuo:2021}. In addition, we demonstrate the effectiveness of the proposed algorithms in realistic applications.

%% file: ai.tex
\section*{Declaration of generative AI and AI-assisted technologies in the manuscript preparation process.}

During the preparation of this work the authors used a Large Language Model in order to assist with generating figures, editing, and searching for typos.
After using this tool, the authors reviewed and edited the content as needed and take full responsibility for the content of the published article.

%% file: recursive.tex
\section{Proof of \Cref{theorem:recursive}}\label{section:recursive}
This section is dedicated to the proof of \Cref{theorem:recursive}.  
To begin, we first prove that any $\bm{B}\in \butterfly(L,k)$ admits a factorization of the form $\bm{B} = \bm{U} \bm{X} \bm{V}^\T$, where $\bm{U}$ and $\bm{V}$ are blockdiagonal. 

\begin{lemma}\label[lemma]{lemma:highestlevel}
    Let $\bm{B} \in \butterfly(L,k)$. Then $\bm{B}$ admits a factorization of the form $\bm{B} = \bm{U} \bm{X} \bm{V}^\T$ where 
        \begin{align*}
                \bm{U} &= \blockdiag(\bm{U}_1,\ldots,\bm{U}_{2^L}), \quad \bm{U}_i \in \mathbb{R}^{2k \times k} \text{ has orthonormal columns},\\
            \bm{V} &= \blockdiag(\bm{V}_1,\ldots,\bm{V}_{2^L}), \quad \bm{V}_i \in \mathbb{R}^{2k \times k} \text{ has orthonormal columns},
            \end{align*}
\end{lemma}
\begin{remark}
    The result remains true if $\bm{B} \in \bsep(L,k)$.
\end{remark}
\begin{proof}
    For $i = 1,\ldots,2^L$, let $I_{i,L}$ be the leaf nodes of a perfect binary partition tree $\mathcal{T} = \{I_{i,L}\}$ (\Cref{def:index_sets}). Then, by \Cref{def:butterfly} and noting that $I_{1,0} = [2^{L+1}k]$ we have $\rank(\bm{B}(I_{i,L},I_{1,0})) = \rank(\bm{B}(I_{i,L},:)) \leq k$. Hence, $\bm{B}(I_{i,L},:) = \bm{U}_i \bm{U}_i^\T \bm{B}(I_{i,L},:)$ for some  matrix $\bm{U}_{i} \in \mathbb{R}^{2k \times k}$ with orthonormal columns. We stack these matrices into a block-diagonal matrix to define $\bm{U} := \blockdiag(\bm{U}_1,\ldots,\bm{U}_{2^L})$. Then, we can write
    \begin{equation}\label{eq:decomp1}
        \bm{B} = \bm{U}\bm{U}^\T \bm{B}.
    \end{equation}
    By the analogous argument for $\bm{B} (:, I_{j, L})$ for $j = 1, \dots, 2^L$, there exist matrices $\bm{V}_j \in \mathbb{R}^{2k \times k}$  whose columns form an orthonormal basis  so that given $\bm{V} := \blockdiag(\bm{V}_1,\ldots,\bm{V}_{2^L})$, we can write
    \begin{equation}\label{eq:decomp2}
        \bm{B} = \bm{B}\bm{V}\bm{V}^\T .
    \end{equation}
    Combining \eqref{eq:decomp1} and \eqref{eq:decomp2} yields
    \begin{equation*}
        \bm{B} = \bm{U}\bm{U}^\T \bm{B} = \bm{U}\bm{U}^\T \bm{B}\bm{V}\bm{V}^\T = \bm{U} \bm{X} \bm{V}^\T,
    \end{equation*}
    where $\bm{X} = \bm{U}^\T \bm{B}\bm{V}$, as required.
\end{proof}
It now remains to show that if
\begin{equation*}
    \bm{X} = \begin{bmatrix} \bm{X}_{1,1} & \bm{X}_{1,2} \\ 
                \bm{X}_{2,1} & \bm{X}_{2,2} \end{bmatrix},
\end{equation*}
then $\bm{X}_{s,t} \in \butterfly(L-2,k)$ for $s,t = 1,2$.
To do this, we must introduce two operations on the nodes of a binary partition tree. 
Assume $L$ is even and let $\mathcal{T} = \{I_{i,L}\}$ be a level $L$ perfect dyadic binary partition tree of $[2^{L+1}k]$ (\Cref{def:index_sets}). Define 
\begin{equation*}
    \domain(\leftfold) := \{I_{i,\ell} \in \mathcal{T} : \ell = 1,\ldots,L-1, i = 1,\ldots,2^{\ell}\} \subset \mathcal{T}
\end{equation*}
and the following operations
\begin{align*}
    \leftfold: \domain(\leftfold) \to \range(\leftfold) \subset \mathcal{T}, &\quad \leftfold(I_{i,\ell}) = \begin{cases} I_{i,\ell} \text{ if } i \leq 2^{\ell-1},\\
    I_{i-2^{\ell-1},\ell} \text{ otherwise}, \end{cases}\\
    \side: \domain(\leftfold) \to \{1,2\}, &\quad \side(I_{i,\ell}) = \begin{cases} 1 \text{ if } i \leq 2^{\ell-1},\\
    2 \text{ otherwise.} \end{cases}
\end{align*}
The $\leftfold$ operation is defined on all non-leaf and non-root nodes of $\mathcal{T}$. 
Intuitively, $\leftfold$ collapses the left and right halves of each level of the tree onto the left half, while $\side$ records from which half the node originated. 
Consequently, every node $I_{i,\ell} \in \domain(\leftfold)$ can be represented uniquely by the pair $(\leftfold(I_{i,\ell}),\side(I_{i,\ell}))$. 
\Cref{fig:treeops} provides a visual explanation for these operations.

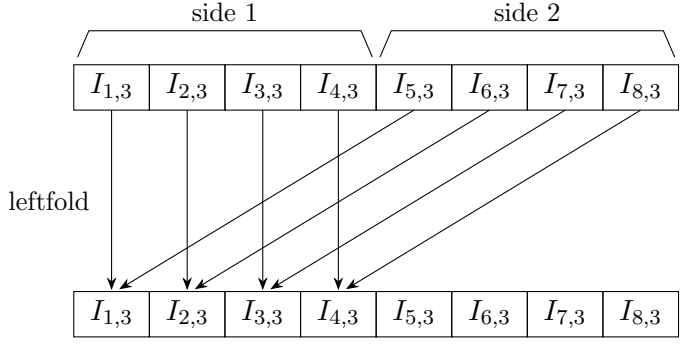
\begin{figure}
    \centering
    \input{leftfold}
    \caption{Visualization of $\leftfold$ and $\side$ operations for $\ell = 3$.}
    \label{fig:treeops}
\end{figure}

The reason for introducing $\leftfold$ and $\side$ is that they formalize the bookkeeping behind the recursive structure illustrated in \Cref{fig:recursive_def}.
By \Cref{lemma:highestlevel}, we can write any $\bm{B} \in \butterfly(L,k)$ as $\bm{B} = \bm{U} \bm{X} \bm{V}^\T$, where $\bm{U}$ and $\bm{V}$ are blockdiagonal matrices with orthonormal columns. 
Partition $\bm{U} = \blockdiag(\bm{U}_{[1]},\bm{U}_{[2]})$ and $\bm{V} = \blockdiag(\bm{V}_{[1]},\bm{V}_{[2]})$, where the split is exactly in half. 
Then for any complementary pair $(I_{i,\ell},I_{j,L-\ell})$, if 
\begin{equation*}
    s = \side(I_{i,\ell}), \quad t = \side(I_{j,L-\ell}),
\end{equation*}
and
\begin{equation*}
    I_{i',\ell} = \leftfold(I_{i,\ell}), \quad I_{j',L-\ell} = \leftfold(I_{j,L-\ell}),
\end{equation*}
we have
\begin{equation*}
    \bm{B}(I_{i,\ell},I_{j,L-\ell}) = \bm{U}_{[s]}(I_{i',\ell},:) \bm{X} \bm{V}_{[t]}(I_{j',L-\ell},:)^\T.
\end{equation*}
The key observation is that after folding the row and column nodes onto the left half of the tree, the complementary block $\bm{B}(I_{i,\ell}, I_{j,L-\ell})$ depends only on the block $\bm{X}_{s,t}$ of $\bm{X}$, where $s = \side(I_{i,\ell})$ and $t = \side(I_{j,L-\ell})$. 
Consequently, every complementary pair in $\domain(\leftfold)$ is associated with exactly one of the four matrices $\bm{X}_{1,1}, \bm{X}_{1,2}, \bm{X}_{2,1},\bm{X}_{2,2}$, as illustrated in \Cref{fig:recursive_def}.

To prove that $\bm{X}_{s,t} \in \butterfly(L-2,k)$, we must also identify the corresponding complementary pair in a binary partition tree $\mathcal{T}'$ with $L-2$ levels.
To make this precise, we introduce another operation. 
Let $\mathcal{T}'$ be a level $L-2$ perfect dyadic binary partition tree of $[2^{L-1}k]$ with nodes $I'_{i,\ell}$. 
Define
\begin{equation*}
    \halving: \range(\leftfold) \to \mathcal{T}', \quad \halving(I_{i,\ell}) = I'_{i,\ell-1},
\end{equation*}
Intuitively, after folding to the left, the resulting indices can be relabeled as nodes of a tree that is two levels shallower. 
The map $\halving$ performs exactly this relabeling.

Consequently, every node $I_{i,\ell} \in \domain(\leftfold)$ is uniquely identified by the tuple
\begin{equation*}
    (\halving(\leftfold(I_{i,\ell})), \side(I_{i,\ell})).
\end{equation*}
Conversely, every node of the smaller tree $\mathcal{T}'$ together with a side label $s \in \{1,2\}$ determines a unique node in $\mathcal{T}$.
Thus, the maps $\leftfold$, $\side$, and $\halving$ identify a one-to-one correspondence between nodes in $\domain(\leftfold)$ and nodes of $\mathcal{T}'$ together with a side label.

The next lemma is a bookkeeping result. 
Its purpose is to relate the block-diagonal factors $\bm{U}$ and $\bm{V}$ to this correspondence between $\domain(\leftfold)$ and $\mathcal{T}'$.

\begin{lemma}\label[lemma]{lemma:indices}
    Consider
    \begin{equation*}
        \bm{W} = \blockdiag(\bm{W}_1,\ldots,\bm{W}_{2^L}) = \blockdiag(\bm{W}_{[1]},\bm{W}_{[2]}) \in \mathbb{R}^{2^{L+1}k \times 2^L k},
    \end{equation*}
    where $\bm{W}_{1},\ldots,\bm{W}_{2^{L}} \in \mathbb{R}^{2k \times k}$ and $\bm{W}_{[1]},\bm{W}_{[2]} \in \mathbb{R}^{2^{L}k \times 2^{L-1}k}$. For any $I_{i,\ell} \in \domain(\leftfold)$ define
    \begin{equation*}
        I_{i',\ell} = \leftfold(I_{i,\ell}), \quad s = \side(I_{i,\ell}), \quad I'_{i',\ell-1} = \halving(I_{i',\ell}).
    \end{equation*}
    Then, 
    \begin{equation*}
        \bm{W}(I_{i,\ell},:) = \bm{W}_{[s]}(I_{i',\ell},I'_{i',\ell-1}) \bm{I}_{2^{L-1}k}(:,I'_{i',\ell-1})^\T \bm{D}_s,
    \end{equation*}
    where 
    \begin{equation*}
        \bm{D}_s = \begin{cases} \begin{bmatrix} \bm{I}_{2^{L-1}k} & \bm{0}_{2^{L-1}k \times 2^{L-1}k} \end{bmatrix} \text{ if } s = 1, \\ \\
        \begin{bmatrix}  \bm{0}_{2^{L-1}k \times 2^{L-1}k} & \bm{I}_{2^{L-1}k} \end{bmatrix} \text{otherwise.}
        \end{cases}
    \end{equation*}
    Furthermore, if $\bm{W}$ has orthonormal columns, then so has $\bm{W}_{[s]}(I_{i',\ell},I'_{i',\ell-1})$.
\end{lemma}
\begin{proof}
    Note that
    \begin{align*}
        \bm{W}(I_{i,\ell},:) &=\begin{cases} \begin{bmatrix} \bm{W}_{[1]}(I_{i,\ell},:) & \bm{0}_{|I_{i,\ell}| \times 2^{L-1}k} \end{bmatrix} \text{ if } s = 1,\\\\
        \begin{bmatrix} \bm{0}_{|I_{i,\ell}| \times 2^{L-1}k } & \bm{W}_{[2]}(I_{i-2^{\ell-1},\ell},:) \end{bmatrix} \text{ otherwise.}
        \end{cases} \\
        &= \bm{W}_{[s]}(I_{i',\ell},:) \bm{D}_s.
    \end{align*}
    Let $\bm{W}_{[s]} = \blockdiag(\bm{W}_{[s],1},\ldots,\bm{W}_{[s],2^{L-1}})$ be the blockdiagonal structure of $\bm{W}_{[s]}$, where each diagonal block is in $\mathbb{R}^{2k \times k}$. 
    Note that
    \[
        \bm{W}_{[s]}(I_{i',\ell},I'_{i',\ell-1})=\blockdiag(\bm{W}_{[s],(i'-1)2^{L-\ell} + 1}, \ldots, \bm{W}_{[s],i'2^{L-\ell}}),
    \]
    using that $I'_{i,\ell-1} = [(i'-1)2^{L-\ell}k+1;i'2^{L-\ell}k]$ (see \Cref{def:index_sets}). Now note that
    \begin{align*}
        \bm{W}_{[s]}(I_{i',\ell},:) &= \begin{bmatrix} \bm{0}_{|I_{i',\ell}| \times (i'-1)2^{L-\ell} k } &
         \bm{W}_{[s]}(I_{i',\ell},I'_{i',\ell-1}) & 
        \bm{0}_{|I_{i',\ell}| \times (2^{L-1} - i'2^{L-\ell})k} \end{bmatrix}\\
        &=  \bm{W}_{[s]}(I_{i',\ell},I'_{i',\ell-1}) \begin{bmatrix} \bm{0}_{2^{L-\ell}k \times (i'-1)2^{L-\ell}k} & \bm{I}_{2^{L-\ell} k} & \bm{0}_{2^{L-\ell}k \times (2^{L-1} - i'2^{L-\ell})k} \end{bmatrix}\\
        &= \bm{W}_{[s]}(I_{i',\ell},I'_{i',\ell-1}) \bm{I}_{2^{L-1}k}(:,I'_{i',\ell-1})^\T,
    \end{align*}
    Hence,
    \begin{equation*}
        \bm{W}(I_{i,\ell},:) = \bm{W}_{[s]}(I_{i',\ell},I'_{i',\ell-1}) \bm{I}_{2^{L-1}k}(:,I'_{i',\ell-1})^\T \bm{D}_s.
    \end{equation*}
    The statement about orthonormal columns is immediate from 
    \begin{equation*}
        \bm{W}_{[s]}(I_{i',\ell},I'_{i',\ell-1}) =\blockdiag(\bm{W}_{[s],(i'-1)2^{L-\ell} + 1}, \ldots, \bm{W}_{[s],i'2^{L-\ell}}), 
    \end{equation*}
    and the fact that each diagonal block has orthonormal columns. 
\end{proof}
The significance of $\Cref{lemma:indices}$ is that it allows us to identify the rows and columns of the child matrix $\bm{X}_{s,t}$ associated with a complementary pair in $\mathcal{T}$.
We can therefore relate complementary blocks of $\bm{B}$ to complementary blocks of the matrices $\bm{X}_{s,t}$.
\begin{lemma}\label[lemma]{lemma:direction1}
   Let $\bm{B} \in \butterfly(L,k)$ with decomposition $\bm{B} = \bm{U} \bm{X} \bm{V}^\T$ as in \Cref{lemma:highestlevel}. Partition $\bm{X}$ into a $2 \times 2$ block-matrix 
    \begin{equation*}
        \bm{X} = \begin{bmatrix} \bm{X}_{1,1} & \bm{X}_{1,2} \\
        \bm{X}_{2,1} & \bm{X}_{2,2}
        \end{bmatrix}, 
        \quad \bm{X}_{s,t} \in \mathbb{R}^{2^{L-1}k \times 2^{L-1}k}.
    \end{equation*}
    Partition $\bm{U} = \blockdiag(\bm{U}_{[1]},\bm{U}_{[2]})$ where $\bm{U}_{[s]} \in \mathbb{R}^{2^{L}k \times 2^{L-1} k}$ and partition $\bm{V} = \blockdiag(\bm{V}_{[1]},\bm{V}_{[2]})$ analogously. 
    
    Assume $L$ is even and let $\mathcal{T}$ be a level $L$ perfect dyadic binary partition tree of $[2^{L+1}k]$. Let $\mathcal{T}'$ be a level $L-2$ perfect dyadic binary partition tree of $[2^{L-1}k]$. Denote by $I_{i,\ell}$ and $I'_{i,\ell}$ the $i^{\text{th}}$ node at level $\ell$ of $\mathcal{T}$ and $\mathcal{T}'$ respectively. 
    
    For any complementary pairs $I_{i,\ell}, I_{j,L-\ell} \in \domain(\leftfold) \subset \mathcal{T}$, define
    \begin{align*}
        I_{i',\ell} &= \leftfold(I_{i,\ell}), \quad &&s = \side(I_{i,\ell}), \quad &&&I'_{i',\ell-1} = \halving(I_{i',\ell}),\\
        I_{j',L-\ell} &= \leftfold(I_{j,L-\ell}), \quad &&t = \side(I_{j,L-\ell}), \quad &&&I'_{j',L-\ell-1}= \halving(I_{j',L-\ell}).
    \end{align*}
    Then, we have
    \begin{equation*}
        \bm{B}(I_{i,\ell},I_{j,L-\ell}) = \bm{U}_{[s]}(I_{i',\ell},I'_{i',\ell-1}) \bm{X}_{s,t}(I'_{i',\ell-1},I'_{j',L-\ell-1}) (\bm{V}_{[t]}(I_{j',L-\ell},I'_{j',L-\ell-1}))^\T.
    \end{equation*}
\end{lemma}
\begin{proof}
    For notational brevity, define
    \[
        \widetilde{\bm{U}} := \bm{U}_{[s]}(I_{i',\ell},I'_{i',\ell-1})
        ,\quad
        \widetilde{\bm{V}} := \bm{V}_{[t]}(I_{j',L-\ell},I'_{j',L-\ell-1}).
    \]
    By \Cref{lemma:indices} we have
    \begin{align*}
        \bm{B}(I_{i,\ell},I_{j,L-\ell}) &= \bm{U}(I_{i,\ell},:) \bm{X} \bm{V}(I_{j,L-\ell},:)^\T \\
        &= \widetilde{\bm{U}} \bm{I}_{2^{L-1}k}(:,I'_{i',\ell-1})^\T\bm{D}_s \bm{X} \bm{D}_t^\T \bm{I}_{2^{L-1}k}(:,I'_{j',L-\ell-1}) \widetilde{\bm{V}}^\T\\
        &=\widetilde{\bm{U}} \bm{I}_{2^{L-1}k}(:,I'_{i',\ell-1})^\T\bm{X}_{s,t} \bm{I}_{2^{L-1}k}(:,I'_{j',L-\ell-1})\widetilde{\bm{V}}^\T\\
        &=\widetilde{\bm{U}} \bm{X}_{s,t}(I'_{i',\ell-1},I'_{j',L-\ell-1}) \widetilde{\bm{V}}^\T,
    \end{align*}
    as required.
\end{proof}
\Cref{lemma:direction1} formalizes the forward direction suggested by \Cref{fig:recursive_def}: every complementary low-rank block can be expressed in terms of a complementary low-rank block of exactly one of $\bm{X}_{s,t}$ for $s,t=1,2$.
To prove the equivalence of the recursive definition of butterfly, we also require the converse statement.

\begin{lemma}\label[lemma]{lemma:direction2}
    Consider the setting of \Cref{lemma:direction1}. 
    Let $(I_{i',\ell},I_{j',L-2-\ell})$ be any complementary pair in $\mathcal{T}'$, and let $s,t \in \{1,2\}$. 
    Let  $I_{i,\ell+1}$ and $I_{j,L-\ell-1}$ denote the unique nodes of $\mathcal{T}$ satisfying
    \begin{align*}
        \halving(\leftfold(I_{i,\ell+1})) &= I'_{i',\ell}, \quad &&\side(I_{i,\ell+1}) = s,\\
        \halving(\leftfold(I_{j,L-\ell-1})) &= I'_{j',L-2-\ell}, \quad &&\side(I_{j,L-\ell-1}) = t.
    \end{align*}
    Then, 
    \begin{equation*}
        \bm{X}_{s,t}(I'_{i',\ell},I'_{j',L-2-\ell}) = (\bm{U}_{[s]}(I_{i,\ell+1},I'_{i',\ell}))^\T\bm{B}(I_{i,\ell+1},I_{j,L-\ell-1}) \bm{V}_{[t]}(I_{j,L-\ell-1},I'_{j',L-2-\ell}).
    \end{equation*}
\end{lemma}
\begin{proof}
    The statement follows by applying \Cref{lemma:direction1} to the complementary pair  $(I_{i,\ell+1}, I_{j,L-1-\ell})$ in $\mathcal{T}$ and then multiplying on the left by $\bm{U}_{[s]}(I_{i,\ell+1},I'_{i',\ell})^\T$ and on the right by $\bm{V}_{[t]}(I_{j,L-\ell-1},I'_{j',L-2-\ell})$, and using the fact that they have orthonormal columns; see \Cref{lemma:indices}. 
\end{proof}

Together, \Cref{lemma:direction1} and \Cref{lemma:direction2} make the correspondence between blocks illustrated in \Cref{fig:recursive_def} precise.
\Cref{theorem:recursive} can now be proved.

\begin{proof}[Proof of \Cref{theorem:recursive}]
     Since $L$ is even, we may write $L = 2h$ for some $h \in \mathbb{N}$. 

    $(\Rightarrow)$ Suppose that $\bm{B} \in \butterfly(L,k)$. 
    By \Cref{lemma:highestlevel}, $\bm{B}$ admits a factorization of the form
    \begin{equation*}
        \bm{B} = \bm{U} \bm{X} \bm{V}^\T,
    \end{equation*}
    where $\bm{U}$ and $\bm{V}$ are block-diagonal with $2^{L}$ blocks of size $2k \times k$ with orthonormal columns and $\bm{X} = \bm{U}^\T \bm{B} \bm{V}$.
    Partition $\bm{X}$ as a $2\times 2$ block-matrix with blocks of equal size:
    \begin{equation*}
        \bm{X} = \begin{bmatrix} \bm{X}_{1,1} & \bm{X}_{1,2} \\ \bm{X}_{2,1} & \bm{X}_{2,2} \end{bmatrix}, \quad \bm{X}_{s,t} \in \mathbb{R}^{2^{L-1} k \times 2^{L-1}k}.
    \end{equation*}
    We must now show that $\bm{X}_{s,t} \in \butterfly(L-2,k)$ for $s,t = 1,2$. Let $\mathcal{T}'$ be a perfect dyadic binary partition tree of $[2^{L-1}k]$ whose $i^{'\text{th}}$ node on level $\ell$ is denoted by $I'_{i',\ell}$ (see \Cref{def:index_sets}). Pick any complementary pair $(I'_{i',\ell}, I'_{j',L-2-\ell})$ in $\mathcal{T}'$ and define $(I_{i,\ell+1},I_{j,L-\ell-1})$ as in \Cref{lemma:direction2}. 
    By \Cref{lemma:direction2} we have 
    \begin{equation*}
        \rank(\bm{X}_{s,t}(I'_{i',\ell},I'_{j',L-2-\ell})) \leq \rank(\bm{B}(I_{i,\ell+1},I_{j,L-\ell-1})) \leq k,
    \end{equation*}
    since $(I_{i,\ell+1},I_{j,L-\ell-1})$ are complementary pairs. 
    Hence, since the complementary pair $I'_{i',\ell}$ and $I'_{j',L-2-\ell}$ was arbitrary, by \Cref{def:butterfly} $\bm{X}_{s,t} \in \butterfly(L-2,k)$, as required. 

    $(\Leftarrow)$ Suppose $\bm{B}$ has the recursive structure outlined in \Cref{theorem:recursive}. We must now show that $\bm{B} \in \butterfly(L,k)$. Note that $\bm{B}(I_{i,L},I_{1,0}) = \bm{B}(I_{i,L},:) = \bm{U}_i \bm{C}_i$ for some matrix $\bm{C}_i \in \mathbb{R}^{k \times 2^{L+1}k}$. Hence, $\rank(\bm{B}(I_{i,L},I_{1,0})) \leq k$. By an analogous argument, $\rank(\bm{B}(I_{1,0},I_{j,L})) \leq k$. Now choose any complementary pair $(I_{i,\ell}, I_{j,L-\ell})$ where $\ell \in \{1,\ldots,L-1\}$. Then, $I_{i,\ell}, I_{j,L-\ell} \in \domain(\leftfold)$. Define 
    \begin{align*}
        I_{i',\ell-1} &= \halving(\leftfold(I_{i,\ell})), \quad  &&s = \side(I_{i,\ell}),\\
        I_{j',L-\ell-1} &= \halving(\leftfold(I_{j,L-\ell})), \quad  && t= \side(I_{j,L-\ell})
    \end{align*}
    By assumption, $\bm{X}_{s,t} \in \butterfly(L-2,k)$ so $\rank(\bm{X}_{s,t}(I'_{i',\ell-1},I'_{j',L-\ell-1})) \leq k$. Then, by \Cref{lemma:direction1} we have
    \begin{equation*}
        \rank(\bm{B}(I_{i,\ell},I_{j,L-\ell})) = \rank(\bm{X}_{s,t}(I'_{i',\ell-1},I'_{j',L-\ell-1})) \leq k.
    \end{equation*}
    Hence, $\bm{B} \in \butterfly(L,k)$, as required. 
\end{proof}

%% file: leftfold.tex
\begin{tikzpicture}[>=Stealth]
\def\L{3}
\def\boxheight{.6}
\def\rowsep{3.0}
\def\bracketheight{0.35}
\def\bracketgap{0.1}
\pgfmathtruncatemacro{\N}{2^\L}
\pgfmathtruncatemacro{\Nhalf}{\N/2}

\foreach \i in {1,...,\N} {
    \pgfmathsetmacro{\xl}{\i-1}
    \pgfmathsetmacro{\xr}{\i}
    \draw (\xl,0) rectangle (\xr,\boxheight);
    \node at ({(\xl+\xr)/2},{0.5*\boxheight}) {$I_{\i,\L}$};
}

\foreach \i in {1,...,\N} {
    \pgfmathsetmacro{\xl}{\i-1}
    \pgfmathsetmacro{\xr}{\i}
    \draw (\xl,\rowsep) rectangle (\xr,\rowsep+\boxheight);
    \node at ({(\xl+\xr)/2},{\rowsep+0.5*\boxheight}) {$I_{\i,\L}$};
}

\pgfmathsetmacro{\bracy}{\rowsep+\boxheight+\bracketgap}
\pgfmathsetmacro{\bracytop}{\bracy+\bracketheight}
\draw ({0+.05},\bracy) -- (0.2,\bracytop) -- (\Nhalf-0.2,\bracytop) -- ({\Nhalf-.05},\bracy);
\node[above] at ({\Nhalf/2},\bracytop) {\small side 1};
\draw ({\Nhalf+.05},\bracy) -- (\Nhalf+0.2,\bracytop) -- (\N-0.2,\bracytop) -- ({\N-.05},\bracy);
\node[above] at ({(\Nhalf+\N)/2},\bracytop) {\small side 2};

\foreach \i in {1,...,\Nhalf} {
    \pgfmathsetmacro{\xmid}{\i - 0.5}
    \draw[->] (\xmid, \rowsep) -- (\xmid, {\boxheight+.02});
}

\foreach \i in {1,...,\Nhalf} {
    \pgfmathtruncatemacro{\src}{\i + \Nhalf}
    \pgfmathsetmacro{\xsrc}{\src - 0.5}
    \pgfmathsetmacro{\xdst}{\i - 0.5+.1}
    \draw[->] (\xsrc, \rowsep) -- (\xdst, {\boxheight+.03});
}

\node[anchor=west] at (-1,{\boxheight/2+\rowsep/2}) {\small leftfold};
\node[anchor=east] at (\N+1,1) {};

\end{tikzpicture}

%% file: absolute.tex
\section{Proof of \Cref{theorem:absolute}}\label{section:absolute}
In this section, we provide the proof of \Cref{theorem:absolute}. In spirit, it is similar to the proof of \Cref{theorem:meta_butterfly}. However, it relies on notation and results introduced in \Cref{section:recursive}. We begin with the following simple theorem, which allows us to obtain an absolute error bound for the approximation quality returned by \Cref{alg:bsep_greedy_meta}. 

\begin{theorem}\label{theorem:blockblrabsolute}
     Consider a matrix $\bm{A} \in \mathbb{R}^{2^{L+1} k \times 2^{L+1} k}$. For $i = 1,\ldots,2^{L}$, define
     \begin{align*}
         &\varepsilon_{i,\mathrm{r}} := \|\bm{A}(I_{i,L},:)-\llbracket \bm{A}(I_{i,L},:)\rrbracket_k\|_\F , \quad \text{and} \\
        &\varepsilon_{i,\mathrm{c}} := \|\bm{A}(:,I_{i,L})-\llbracket \bm{A}(:,I_{i,L})\rrbracket_k\|_\F
     \end{align*}
     as the optimal low-rank approximation errors for the complementary blocks at the extreme levels. Then under the assumptions in \Cref{theorem:bsep}, the output of \Cref{alg:bsep_greedy_meta} we have 
    \begin{align*}
        \frac{1}{2}\sum\limits_{i=1}^{2^L} \left[\varepsilon_{i,\mathrm{r}}^2 + \varepsilon_{i,\mathrm{c}}^2\right] \leq \min\limits_{\bm{C} \in \bsep(L,k)} \|\bm{A} - \bm{C}\|_\F^2 \leq \|\bm{A} - \bm{B}\|_\F^2 \leq \Gamma\sum\limits_{i=1}^{2^L} \left[\varepsilon_{i,\mathrm{r}}^2 + \varepsilon_{i,\mathrm{c}}^2\right].
    \end{align*}
\end{theorem}
\begin{proof}
    Let $\bm{B}^{\star} \in \argmin\limits_{\bm{C} \in \bsep(L,k)}\|\bm{A} - \bm{C}\|_\F$.  Then,
    \begin{align*}
        \frac{1}{2} \sum\limits_{i=1}^{2^L} \left[\varepsilon_{i,\mathrm{r}}^2 + \varepsilon_{i,\mathrm{c}}^2\right] & \leq \frac{1}{2}\sum\limits_{i=1}^{2^L} \left[\|\bm{A}(I_{i,L},:) - \bm{B}^{\star}(I_{i,L},:)\|_\F^2 + \|\bm{A}(:,I_{i,L}) - \bm{B}^{\star}(:,I_{i,L})\|_\F^2\right]\\
        &= \|\bm{A} - \bm{B}^{\star}\|_\F^2\\
        &= \min\limits_{\bm{C} \in \bsep(L,k)}\|\bm{A} - \bm{C}\|_\F^2\\
        & \leq \|\bm{A} - \bm{B}\|_\F^2.
    \end{align*}
    To obtain final inequality, we apply the Pythagorean theorem:
    \begin{align*}
        \|\bm{A} - \bm{B}\|_\F^2 &= \|\bm{A} - \bm{U}\bm{U}^\T \bm{A}\|_\F^2 + \|\bm{U}\bm{U}^\T \bm{A} -  \bm{U}\bm{U}^\T \bm{A}\bm{V}\bm{V}^\T \|_\F^2 \\
        &\leq \|\bm{A} - \bm{U}\bm{U}^\T \bm{A}\|_\F^2 + \| \bm{A} -   \bm{A}\bm{V}\bm{V}^\T\|_\F^2\\
        &\leq \Gamma \sum\limits_{i=1}^{2^L} \left[\varepsilon_{i,\mathrm{r}}^2 + \varepsilon_{i,\mathrm{c}}^2\right],
    \end{align*}
    as required.
\end{proof}

We can now proceed with the proof of \Cref{theorem:absolute}.
\begin{proof}[Proof of \Cref{theorem:absolute}]
    We begin with proving \cref{eq:absolute_main} by induction on $h$, where $L = 2h$. The proof for $h = 0$ is immediate. Assume the result holds for $h = \frac{L}{2}-1$.
    We need to show the result holds for $h = \frac{L}{2}$. 
    We have
    \begin{align*}
        \|\bm{A} - \bm{B}\|_\F^2 &= \|\bm{A} - \bm{U}\bm{U}^\T \bm{A} \bm{V}\bm{V}^\T \|_\F^2 + \sum\limits_{s,t = 1}^{2}\|\widehat{\bm{X}}_{s,t} - \bm{X}_{s,t}\|_\F^2 \tag{Pythagorean theorem}\\
        & \leq \Gamma\sum\limits_{i=1}^{2^L} \left[\varepsilon(I_{i,L},I_{1,0})^2 + \varepsilon(I_{1,0},I_{i,L})^2\right] + \sum\limits_{s,t = 1}^{2}\|\widehat{\bm{X}}_{s,t} - \bm{X}_{s,t}\|_\F^2. \tag{\Cref{theorem:blockblrabsolute}}
    \end{align*}
    Let $\mathcal{T} = \{I_{i,\ell}\}$ and $\mathcal{T}' = \{I'_{i',\ell}\}$ be perfect dyadic binary partition tree of $[2^{L+1}k]$ and $[2^{L-1}k]$, respectively. 
    Choose any complementary pair $(I'_{i',\ell},I_{j',L-2-\ell})$ in $\mathcal{T}'$.
    By \Cref{lemma:direction2,lemma:indices}, $\widehat{\bm{X}}_{s,t}(I'_{i',\ell},I'_{j',L-2 - \ell})$ is a contraction\footnote{A contraction $\bm{C}$ of a matrix $\bm{D}$ is a transformation $\bm{C} = \bm{Q} \bm{D} \bm{W}$ where $\|\bm{Q}\|_2,\|\bm{W}\|_2 \leq 1$. The singular values of a contraction must only get smaller since $\sigma_i(\bm{C}) \leq \|\bm{Q}\|_2 \|\bm{W}\|_2 \sigma_i(\bm{D}) \leq \sigma_i(\bm{D})$.} of $\bm{A}(I_{i,\ell+1},I_{j,L-\ell-1})$, where $I_{i,\ell+1},I_{j,L-\ell-1} \in \mathcal{T}$ are as in  \Cref{lemma:direction2}.
    So the singular values can only decrease and 
    \begin{equation}\label{eq:rankineq}
        \|\widehat{\bm{X}}_{s,t}(I'_{i',\ell},I'_{j',L-2 - \ell}) -\llbracket \widehat{\bm{X}}_{s,t}(I'_{i',\ell},I'_{j',L-2 - \ell})\rrbracket_k\|_\F \leq \varepsilon(I_{i,\ell+1},I_{j,L-\ell-1}). 
    \end{equation}
    Hence, inductively we have\footnote{As we sum over all $s,t = 1,2$ and all complementary pairs $(I'_{i',\ell},I'_{j',L-2-\ell})$ in $\mathcal{T}'$ we also sum over all complementary pairs $(I_{i,\ell},I_{j,L-\ell})$ in $\mathcal{T}$ where $\ell= 1,\ldots,L/2 - 1$.}
    \begin{align*}
        \sum\limits_{s,t = 1}^{2}\|\widehat{\bm{X}}_{s,t} - \bm{X}_{s,t}\|_\F^2 \leq \Gamma\sum\limits_{\ell = 1}^{L/2 -1} \sum\limits_{i=1}^{2^{\ell}} \sum\limits_{j=1}^{2^{L-\ell}} \left[\varepsilon(I_{i,\ell},I_{j,L-\ell})^2 + \varepsilon(I_{j,L-\ell},I_{i,\ell})^2\right].
    \end{align*}
    Hence, we obtain \cref{eq:absolute_main}
    \begin{align*}
        \|\bm{A} - \bm{B}\|_\F^2 &= \|\bm{A} - \bm{U}\bm{U}^\T \bm{A} \bm{V}\bm{V}^\T \|_\F^2 + \sum\limits_{s,t = 1}^{2}\|\widehat{\bm{X}}_{s,t} - \bm{X}_{s,t}\|_\F^2\\
        &\leq\Gamma\sum\limits_{i=1}^{2^L} \left[\varepsilon(I_{i,L},I_{1,0})^2 + \varepsilon(I_{1,0},I_{i,L})^2\right] + \Gamma\sum\limits_{\ell = 1}^{L/2 -1} \sum\limits_{i=1}^{2^{\ell}} \sum\limits_{j=1}^{2^{L-\ell}} \left[\varepsilon(I_{i,\ell},I_{j,L-\ell})^2 + \varepsilon(I_{j,L-\ell},I_{i,\ell})^2\right]\\
        &=\Gamma \sum\limits_{\ell = 0}^{L/2} \sum\limits_{i=1}^{2^{\ell}} \sum\limits_{j=1}^{2^{L-\ell}} \left[\varepsilon(I_{i,\ell},I_{j,L-\ell})^2 + \varepsilon(I_{j,L-\ell},I_{i,\ell})^2\right],
    \end{align*}
    as required. To obtain \cref{eq:absolute1} we bound
    \begin{align*}
        \|\bm{A} - \bm{B}\|_\F^2 &\leq\Gamma \sum\limits_{\ell = 0}^{L/2} \sum\limits_{i=1}^{2^{\ell}} \sum\limits_{j=1}^{2^{L-\ell}} \left[\varepsilon(I_{i,\ell},I_{j,L-\ell})^2 + \varepsilon(I_{j,L-\ell},I_{i,\ell})^2\right]\\
        &\leq 2 \Gamma\varepsilon^2 \sum\limits_{\ell = 0}^{L/2} \sum\limits_{i=1}^{2^{\ell}} \sum\limits_{j=1}^{2^{L-\ell}}1 \\
        &= \Gamma(L+2)2^L\varepsilon^2.
    \end{align*}
    \cref{eq:absolute1} is obtained by noting $2^L = \frac{N}{2k}$. To obtain \cref{eq:absolute2} we note that 
    \begin{align*}
        \sum\limits_{\ell = 0}^{L/2}\sum\limits_{i=1}^{2^{\ell}} \sum\limits_{j=1}^{2^{L-\ell}} \left[\varepsilon(I_{i,\ell},I_{j,L-\ell})^2 + \varepsilon(I_{j,L-\ell},I_{i,\ell})^2\right]
        &\leq \varepsilon^2 \sum\limits_{\ell = 0}^{L/2}\sum\limits_{i=1}^{2^{\ell}} \sum\limits_{j=1}^{2^{L-\ell}} \left[\|\bm{A}(I_{i,\ell},I_{j,L-\ell})\|_\F^2 + \|\bm{A}(I_{j,L-\ell},I_{i,\ell})\|_\F^2\right]\\
        &\leq2 \varepsilon^2 \|\bm{A}\|_\F^2 \sum\limits_{\ell=0}^{L/2} 1\\
        &\leq (L+2) \varepsilon^2\|\bm{A}\|_\F^2,
    \end{align*}
    as required.
\end{proof}

%% file: memory_efficient_implementation.tex
\section{How to make \Cref{alg:butterfly_greedy_matvec} memory efficient}\label{section:memoryeff} We now describe a memory-efficient implementation of \Cref{alg:butterfly_greedy_matvec}. Recall that the memory bottleneck is the storage of the sketches $\{\bm{Y}^{(\ell)}, \bm{Z}^{(\ell)}\}_{\ell = 0,2,\ldots,L}$. When all sketches are computed in advance, many of them remain unused for a long period because they are only needed locally within a small part of the recursion.  Instead, one can form the relevant parts of these sketches just before they are needed in the recursion. Consider $\bm{\Omega}^{(\ell)}$ as defined in \cref{eqn:sketch_def}, and let 
\begin{equation*}
    \bm{\Omega}^{(\ell)}_{i,\text{col}} = \bm{e}_i \otimes \bm{\Omega}^{(\ell)}_i
\end{equation*}
be the $i^{\text{th}}$ block column of $\bm{\Omega}^{(\ell)}$, where $\bm{e}_i \in \mathbb{R}^{2^{\frac{L-\ell}{2}}}$ is the $i^{\text{th}}$ canonical basis vector, so that
\begin{equation*}
    \bm{\Omega}^{(\ell)} = \begin{bmatrix} \bm{\Omega}^{(\ell)}_{1,\text{col}} & \cdots & \bm{\Omega}^{(\ell)}_{2^\frac{L-\ell}{2},\text{col}} \end{bmatrix}.
\end{equation*}
Hence, if $\bm{Y}^{(\ell)}_{i,\text{col}} = \bm{A} \bm{\Omega}^{(\ell)}_{i,\text{col}}$ then $\bm{Y}^{(\ell)}$ in \cref{eq:matvecs} takes the form
\begin{equation}\label{eq:blockmatvecs}
    \bm{Y}^{(\ell)} = \begin{bmatrix} \bm{Y}^{(\ell)}_{1,\text{col}} & \cdots & \bm{Y}^{(\ell)}_{2^\frac{L-\ell}{2},\text{col}} \end{bmatrix}.
\end{equation}
The memory-efficient implementation never forms $\bm{Y}^{(\ell)}$ explicitly. 
Instead, it computes a single block column $\bm{Y}^{(\ell)}_{i,\text{col}}$, uses it to construct every basis in the recursion that depends on that block column, and then immediately discards it.
The same idea applies to the transpose sketches. Writing
\begin{equation*}
    \bm{\Psi}^{(\ell)} = \begin{bmatrix} \bm{\Psi}^{(\ell)}_{1,\text{col}} & \cdots & \bm{\Psi}^{(\ell)}_{2^\frac{L-\ell}{2},\text{col}} \end{bmatrix}
\end{equation*}
and defining
\begin{equation}\label{eq:blockmatvecs_transpose}
    \bm{Z}_{i,\text{col}}^{(\ell)} =\bm{A}^\T \bm{\Psi}_{i,\text{col}}^{(\ell)},
\end{equation}
we obtain a similar block-column format of $\bm{Z}^{(\ell)}$. 
The memory-efficient implementation is given in \Cref{alg:butterfly_greedy_matvec_memory}. The memory complexity is improved to $O(Nk\log(N/k))$, while the output and approximation guarantees remain unchanged. 
The algorithm loops over the levels $\ell$ and block-columns indices $i$, computes a single sketch $\bm{Y}^{(\ell)}_{i,\text{col}}$ or $\bm{Z}^{(\ell)}_{i,\text{col}}$, and immediately passes it to \Cref{alg:recursefun}, which propagates this sketch through the butterfly hierarchy and perform all updates in the butterfly structure associated with that sketch. 
\Cref{alg:recursefun} takes the following inputs:
\begin{enumerate}
    \item[i)] A butterfly struct $\bm{B}$. Initially this is an empty butterfly struct. Each call to \Cref{alg:recursefun} updates the portion of the butterfly factorization associated with the current sketch.
    \item[ii)] Sketch $\bm{S}$. This is one of the sketch blocks $\bm{Y}^{(\ell)}_{i,\text{col}}$ or $\bm{Z}^{(\ell)}_{i,\text{col}}$, which is propagated through the butterfly hierarchy. 
    \item[iii)] The leaf level $L$, level $\ell$, and index $i$. These inputs determine the depth of the recursion and identify the subtree of the butterfly structure which will be updated by the sketch. 
    \item[iv)] $\text{Direction} \in \{\text{left},\text{right}\}$. This records whether the sketch is with $\bm{A}$ ($\text{Direction} = \text{right}$) or $\bm{A}^\T$ ($\text{Direction} = \text{left}$).
    \item[v)] Sketch matrix $\bm{\Psi}$ (only if $\ell = 0$). At the lowest level, the generalized Nyström approximation requires access to the sketch matrix itself. Consequently, when $\ell =0$, the corresponding block of $\bm{\Psi}$ is propagated through the recursion together with the sketch.
\end{enumerate}

\begin{algorithm}
\caption{Greedy $\butterfly$ approximation (memory-efficient matvec algorithm)}
\label{alg:butterfly_greedy_matvec_memory}
\textbf{input:} Matvec-access to $\bm{A}$. Rank parameter $k$. Level $L$.\\
\textbf{output:} An approximation $\bm{B} \in \butterfly(L,k)$ with factorization as in \Cref{theorem:recursive}.
\begin{algorithmic}[1]
\State Initialize an empty struct $\bm{B} \in \butterfly(L,k)$. 
\For{$\ell = L,L-2,\ldots,2,0$}
    \For{$i = 1,\ldots,2^{\frac{L-\ell}{2}}$}
        \State Compute $\bm{Y}^{(\ell)}_{i,\text{col}} = \bm{A} \bm{\Omega}^{(\ell)}_{i,\text{col}}$ as in \cref{eq:blockmatvecs}.
        \State Call \Cref{alg:recursefun} with inputs\vspace{-0.5\baselineskip}
        \[
        \begin{array}{rl@{\qquad\qquad}rl@{\qquad\qquad}rl}
            \text{Butterfly struct:} & \bm{B} & \text{Sketch:} & \bm{Y}^{(\ell)}_{i,\text{col}} & \text{Leaf level:} & L \\
            \text{Level:} & \ell & \text{Index:} & i & \text{Direction:} & \text{right} \\
        \end{array}
        \]\vspace{-0.5\baselineskip}
        \State Discard $\bm{Y}^{(\ell)_{i,\text{col}}}$.
    \EndFor
\EndFor
\For{$\ell = L,L-2,\ldots,2,0$}
    \For{$i = 1,\ldots,2^{\frac{L-\ell}{2}}$}
        \State  Compute $\bm{Z}^{(\ell)}_{i,\text{col}} = \bm{A}^\T\bm{\Psi}^{(\ell)}_{i,\text{col}}$ as in \cref{eq:blockmatvecs_transpose}.
        \State Call \Cref{alg:recursefun} with inputs\vspace{-0.5\baselineskip}
        \[
        \begin{array}{rl@{\qquad\qquad}rl@{\qquad\qquad}rl}
            \text{Butterfly struct:} & \bm{B} & \text{Sketch:} & \bm{Z}^{(\ell)}_{i,\text{col}} & \text{Leaf level:} & L \\
            \text{Level:} & \ell & \text{Index:} & i & \text{Direction:} & \text{left} \\
            \text{Sketch matrix (only if $\ell = 0$):} & \bm{\Psi}^{(0)}_{i,\text{col}}
        \end{array}
        \]\vspace{-0.5\baselineskip}
        \State Discard $\bm{Z}^{(\ell)}_{i,\text{col}}$.
    \EndFor
\EndFor
\State \textbf{return} $\bm{B}$
\end{algorithmic}
\end{algorithm}

\begin{algorithm}
\caption{Recursive update from sketches}\label{alg:recursefun}
\textbf{input:} Butterfly struct $\bm{B}$. Sketch $\bm{S}$. Leaf level $L$. Level $\ell$. Index $i$. Direction $\in\{\text{left},\text{right}\}$. Sketch matrix: $\bm{\Psi}$ (only if $\ell = 0$ and $\text{Direction} = \text{left}$).\\
\textbf{output:} An updated struct $\bm{B}$.
\begin{algorithmic}[1]
\If{$\ell = L = 0$ and $\text{Direction} = \text{left}$}
    \State Obtain the left basis $\bm{U}$ from the struct $\bm{B}$.
    \State $\bm{M} = (\bm{\Psi}^\T \bm{U})^{\dagger} \bm{S}^\T$.
    \State Compute the rank reduced SVD of $\bm{M}$ to obtain $\bm{W} \bm{\Sigma} \bm{V}^\T$.
    \State Let $\bm{X} = \bm{W}\bm{\Sigma}$.
    \State Update the struct $\bm{B}$ by appending the core matrix $\bm{X}$ and right basis $\bm{V}$.
\ElsIf{$\ell = L$}
    \State Partition
    \begin{equation*}
        \bm{S} = \begin{bmatrix} \bm{S}^{(\cdot,1)} \\ \vdots \\ \bm{S}^{(\cdot,2^{L})} \end{bmatrix}, \quad \bm{S}^{(\cdot,i)} \in \mathbb{R}^{2k \times \alpha k}.
    \end{equation*}
    \For{$i = 1,\ldots,2^L$}
        \State Compute the top $k$ left singular vectors $\bm{W}_i$ for $\bm{S}^{(\cdot,i)}$.
    \EndFor
    \State $\bm{W} = \blockdiag(\bm{W}_1,\ldots,\bm{W}_{2^L})$. 
    \If{$\text{direction} = \text{right}$}
        \State Update the struct $\bm{B}$ by letting $\bm{U} = \bm{W}$ be the left block-diagonal basis. 
    \Else
        \State Update the struct $\bm{B}$ by letting $\bm{V} = \bm{W}$ be the right block-diagonal basis. 
    \EndIf
\Else
    \If{$\text{direction} = \text{right}$}
        \State Let $\bm{W} = \bm{U}$, where $\bm{U}$ is the left block-diagonal basis in the struct $\bm{B}$.
    \Else
        \State Let $\bm{W} = \bm{V}$, where $\bm{V}$ is the right block-diagonal basis in the struct $\bm{B}$.
    \EndIf
    \State Partition $\bm{W}^\T\bm{S} = \begin{bmatrix} \bm{S}_1 \\ \bm{S}_2 \end{bmatrix}$.
    \State Let
    \begin{equation*}
        \bm{X} = \begin{bmatrix} \bm{X}_{1,1} & \bm{X}_{1,2}\\\bm{X}_{2,1} & \bm{X}_{2,2} \end{bmatrix}
    \end{equation*}
    be the core matrix from $\bm{B}$, where each $\bm{X}_{s,t}$ for $s,t = 1,2$ are recursive butterfly structs. 
    \If{$\ell = 0$ and $\text{Direction} = \text{left}$}
        \State Partition $\bm{U}^\T \bm{\Psi} = \begin{bmatrix} \widehat{\bm{\Psi}}_{1} \\ \widehat{\bm{\Psi}}_{2} \end{bmatrix}$ \Comment{Since the upper or lower half of $\bm{\Psi}$ is zero, $\widehat{\bm{\Psi}}_{1}$ or $\widehat{\bm{\Psi}}_{2}$ is zero.}
    \EndIf
    \algstore{recurse}
\end{algorithmic}
\end{algorithm}
\begin{algorithm}
\caption{Recursive update from sketches (continued)}
\begin{algorithmic}
\algrestore{recurse}
    \State Conditioned on $i$ and $\text{Direction}$, define
    \begin{align*}
        i \leq 2^{\frac{L-\ell}{2} - 1}, \text{Direction}  =\text{right} &\Rightarrow (s_1,t_1) = (1,1), (s_2,t_2) = (2,1), j = i\\
        i > 2^{\frac{L-\ell}{2} - 1}, \text{Direction}  =\text{right} &\Rightarrow (s_1,t_1) = (1,2), (s_2,t_2) = (2,2), j = i - 2^{\frac{L-\ell}{2} - 1}\\
        i \leq 2^{\frac{L-\ell}{2} - 1}, \text{Direction}  =\text{left} &\Rightarrow (s_1,t_1) = (1,1), (s_2,t_2) = (1,2), j = i\\
        i > 2^{\frac{L-\ell}{2} - 1}, \text{Direction}  =\text{left} &\Rightarrow (s_1,t_1) = (2,1), (s_2,t_2) = (2,2), j = i - 2^{\frac{L-\ell}{2} - 1}
    \end{align*}
    \For{$n = 1,2$}
        \State Update the struct $\bm{X}_{s_n,t_n}$ by recursively applying \Cref{alg:recursefun} with inputs
        \[
        \begin{array}{rl@{\qquad\qquad}rl@{\qquad\qquad}rl}
            \text{Butterfly struct:} & \bm{X}_{s_n,t_n} & \text{Sketch:} & \bm{S}_n & \text{Leaf level:} & L-2 \\
            \text{Level:} & \ell & \text{Index:} & j & \text{Direction:} & \text{Direction} \\
            \text{Sketch matrix (only if given):} & \widehat{\bm{\Psi}}_{s_n}
        \end{array}
        \]
    \EndFor
\EndIf
\end{algorithmic}
\end{algorithm}

%% file: naive_matvec_algorithm.tex
\section{A matrix-vector product algorithm with stronger optimality guarantees}
\label{sec:matvec_strongeraccuracy}

The matvec algorithm in \cref{section:matvec} uses $\widetilde{O}(\sqrt{N})$ matvecs to achieve a $O(N^{1/4})$-quasi-optimality guarantee. In this section, we describe a simple matrix-vector algorithm that uses $O(\sqrt{N})$ matvecs to attain a $O(\sqrt{L})$-quasi-optimality guaranty. 
This algorithm is not as efficient in terms of working memory, but it is conceptually simple. 

First, we introduce a family of matrices which have low-rank blocks corresponding to the complimentary low-rank property at $\ell = L/2$.
\begin{definition}[$\butterflymid$]\label{def:butterflymid}
Let $L, k\in \mathbb{N}$ be fixed, and assume that $L$ is even. Consider a matrix $\bm{B} \in \mathbb{R}^{2^{L+1}k \times 2^{L+1}k}$. Using the notation from \Cref{def:index_sets}, we say that $\bm{B} \in \butterflymid(L,k)$ if for every $i,j = 1,\ldots,2^{L/2}$, $\rank(\bm{B}(I_{i,L/2},I_{j,L/2})) \leq k$.
\end{definition}
A near-optimal $\butterflymid$ approximation is obtained by computing near-optimal rank-$k$ approximation to each block $\bm{A}(I_{i,L/2},I_{j,L/2})$. 
The procedure described here will be essentially identical to the one described in \cite[Section 3.1]{li2015butterfly}.
Generate block-diagonal sketch matrices $\bm{\Omega}^{(L/2)}$ and $\bm{\Psi}^{(L/2)}$ as in \cref{eqn:sketch_def} with $\alpha = O(1)$ and compute
\begin{align*}
    \bm{A} \bm{\Omega}^{(L/2)} =  \begin{bmatrix} \bm{Y}_{1,1} & \cdots & \bm{Y}_{1,2^{L/2}} \\
    \vdots & \ddots & \vdots \\
    \bm{Y}_{2^{L/2},1} & \cdots & \bm{Y}_{2^{L/2},2^{L/2}}
    \end{bmatrix},\quad \bm{A}^\T \bm{\Psi}^{(L/2)} = \begin{bmatrix} \bm{Z}_{1,1} & \cdots & \bm{Z}_{2^{L/2},1} \\
    \vdots & \ddots & \vdots \\
    \bm{Z}_{1,2^{L/2}} & \cdots & \bm{Z}_{2^{L/2},2^{L/2}}
    \end{bmatrix},
\end{align*}
where $\bm{Y}_{i,j} = \bm{A}(I_{i,2^{L/2}},I_{j,2^{L/2}})\bm{\Omega}_j^{(L/2)}$ and $\bm{Z}_{i,j}=\bm{A}(I_{i,2^{L/2}},I_{j,2^{L/2}})^\T \bm{\Psi}_i^{(L/2)}$. 
From this we can compute a generalized Nyström approximation to each complementary low-rank block. 
Thus, we can sketch all $(2^{L/2}) \times (2^{L/2})$ blocks using just $O((2^{L/2})k)$ matvecs from each side. This means that, from $O(2^{L/2}k) = O(\sqrt{Nk})$ matvecs with $\bm{A}$ we can construct $\widehat{\bm{A}} \in \butterflymid(L,k)$ satisfying
\begin{equation}
    \label{eqn:A_to_Ahat_approx}
    \mathbb{E}\|\bm{A} - \widehat{\bm{A}}\|_\F \leq C \min\limits_{\bm{C} \in \butterflymid(L,k)}\|\bm{A} - \bm{C}\|_\F,
\end{equation}
for a universal constant $C = O(1)$. Once we have obtained $\widehat{\bm{A}}$, we can use \cref{alg:butterfly_greedy_meta} compute a near-optimal butterfly approximation $\bm{B} \in \butterfly(L,k)$ to $\widehat{\bm{A}}$ \Cref{alg:butterfly_greedy_meta}:
\begin{equation}
    \label{eqn:B_to_Ahat_approx}
    \|\widehat{\bm{A}} - \bm{B}\|_\F \leq \sqrt{\Gamma \cdot (L+2)} \min\limits_{\bm{C} \in \butterfly(L,k)}\|\widehat{\bm{A}} - \bm{C}\|_\F.
\end{equation}
This requires no additional matvecs with $\bm{A}$. By the triangle inequality
\begin{align*}
    \min\limits_{\bm{C} \in \butterfly(L,k)}\|\widehat{\bm{A}} - \bm{C}\|_\F
    & \leq  \| \bm{A} - \widehat{\bm{A}} \|_\F + \min\limits_{\bm{C} \in \butterfly(L,k)}\|\bm{A} - \bm{C}\|_\F.
\end{align*}
Hence, again by the triangle inequality, and by \cref{eqn:B_to_Ahat_approx},
\begin{align*}
    \|\bm{A} - \bm{B}\|_\F & \leq \|\bm{A} - \widehat{\bm{A}}\|_\F + \|\widehat{\bm{A}} - \bm{B}\|_\F 
    \\&\leq \left(1 + \sqrt{\Gamma\cdot(L+2)} \right) \| \bm{A} - \widehat{\bm{A}} \|_\F + \sqrt{\Gamma\cdot(L+2)}\min\limits_{\bm{C} \in \butterfly(L,k)}\|\bm{A} - \bm{C}\|_\F
\end{align*}
Now, since $\butterfly(L,k)\subseteq \butterflymid(L,k)$,
\[
\min\limits_{\bm{C} \in \butterflymid(L,k)}\|\bm{A} - \bm{C}\|_\F \leq \min\limits_{\bm{C} \in \butterfly(L,k)}\|\bm{A} - \bm{C}\|_\F.
\]
Thus, taking expectations, applying \cref{eqn:A_to_Ahat_approx}, and then plugging in the above bounds,
\begin{equation*}
    \mathbb{E} \|\bm{A} - \bm{B}\|_\F
    \leq \left( C + (1+C) \sqrt{\Gamma\cdot(L+2)} \right) \min\limits_{\bm{C} \in \butterfly(L,k)}\|\bm{A} - \bm{C}\|_\F.
    = O(\sqrt{L}) \min\limits_{\bm{C} \in \butterfly(L,k)}\|\bm{A} - \bm{C}\|_\F,
\end{equation*}
which yields the desired bound.

We emphasize that since this method requires storing a matrix $\widehat{\bm{A}} \in \butterflymid(L,k)$, requiring $O(N^{3/2})$ units of memory. Hence, this algorithm can never obtain the optimal memory complexity of $O(N\log(N))$. 
Moreover, if $\widehat{\bm{A}}$ is formed explicitly, then a direct implementation of \Cref{alg:bsep_greedy_meta} requires $O(N^2)$ operations. 
With some care, however, the low-rank representation of $\widehat{\bm{A}}$ can be exploited to obtain an implementation requiring only $O(N^{3/2})$ operations. We omit details. 

%% file: mapbetweenformats.tex
\section{Mapping between recursive and sparse format}
\label{sec:mapbetweenformats}

For ease of implementation, we describe the conversion between the hybrid butterfly factorization in~\cite{LiuXingGuo:2021} and the recursive factorization we present in~\cref{theorem:recursive}. The hybrid butterfly factorization  is defined as follows. Following our convention of a perfect binary tree with $L$ levels and $2k$-leaf size,  $\bm B_L \in \R^{2^{L+1}k \times 2^{L+1}k}$ satisfies

$$
\bm B_L  =\bm U^L \bm R^{L-1} \bm R^{L-2}\cdots\bm R^{L/2}\bm B^{L/2} \bm W^{L/2}\cdots\bm  W^{L-2} \bm W^{L-1} \bm V^0,
$$
where
the ``outer factors'' $\bm U^L \in \R^{2^{L+1}k \times 2^L k}$ and $\bm V^0 \in \R^{2^L k \times 2^{L+1}k}$ are block diagonal matrices of the form 
$$
\bm U^L = \operatorname{blockdiag}(\bm U_1, \dots, \bm U_{2^L}), \quad \bm V^0 = \operatorname{blockdiag}(\bm V_1^\T, \dots, \bm V_{2^L}^\T), $$
where  $\bm U_i,\bm  V_i \in \R^{2k \times k}$ have orthonormal columns. The so-called transfer matrices $\bm R^\ell, \bm W^\ell \in \R^{2^L k \times 2^L k}$ for each level $\ell =  L/2,\dots, L-1 $ are obtained by collecting the level-$\ell$ basis matrices of the complementary low-rank blocks. In particular, the outermost transfer matrices at level $\ell = L-1$ are
\begin{equation}\label{eq:hybrid_transfer_matrices}
\bm R^{L-1}  = \begin{bmatrix}
   \bm U_{11}^{L-1} & &\bm U_{12}^{L-1}&  \\
    &\bm U_{21}^{L-1} & &\bm U_{22}^{L-1}
\end{bmatrix}  \quad \text{and} \quad  \bm  W^{L-1} = \begin{bmatrix}
    \bm V_{11}^{{L-1}^\T} & \\
    & \bm V_{12}^{{L-1}^\T} \\
    \bm V_{21} ^{{L-1}^\T} & \\
    & \bm V_{22}^{{L-1}^\T}
\end{bmatrix},
\end{equation}
where $\bm U_{ij}^{L-1}, \bm V_{ij}^{L-1} \in \R^{2^{\ell}k \times 2^{\ell - 1} k}$ have orthonormal columns. At subsequent levels, each complementary low-rank block corresponds to local transfer matrices of the form~\cref{eq:hybrid_transfer_matrices}, which are then combined and permuted block-wise to yield the transfer matrices $\bm R^\ell$ and $\bm W^\ell$ with the sparsity patterns in~\cite[Figure 3(g)]{LiuXingGuo:2021}.  The middle matrix $\bm{B}^{L/2}$ is a block-permutation matrix as per~\cite[Figure 3(g)]{LiuXingGuo:2021}.

We  now describe how to  recover  the parameters of the  hybrid butterfly factorization from our recursive definition of butterfly matrices in~\cref{theorem:recursive}. Let $\bm B  \in \butterfly(L, k)$. If $L = 0$, the butterfly matrix $\bm B \in \R^{2k \times 2k}$ satisfies  $\bm B = \bm U \bm X \bm V^\T$, where $\bm U, \bm V \in \R^{2k \times k}$, $\bm X \in \R^{k \times k}$. In the corresponding hybrid butterfly factorization, there are no transfer matrices, and thus $\bm B  = \bm U^0 \bm B^0 \bm V^0 $, where $\bm U^0 = \bm U \in \R^{2k \times k}$ and $\bm V ^0 = \bm V \in \R^{2k \times k}$ are trivially block-diagonal matrices with blocks of size $2k \times k$, and $\bm B^0 = \bm X \in \R^{k \times k}$. 

The outermost block-diagonal factors $\bm U , \bm V \in \R^{2^{L+1}k \times 2^L k}$ in the recursive format are related to $\bm B$ by
$$
\bm{B} = \bm{U} \bm{X} \bm{V}^\T.
$$

Then, we convert to the outermost factors in the hybrid factorization simply by setting $\bm U^L := \bm U$ and $\bm V^0 := \bm V$. Because in the base case of $L = 0$, there are no transfer matrices, $\bm B^0$ is the inner matrix $\bm X$. Otherwise, for even $L > 0$, we find the outermost transfer matrices $\bm R^{L-1}$ and $\bm W^{L-1}$ as follows. By~\cref{theorem:recursive}: 
\begin{align*}
     \bm{X} = & \begin{bmatrix} \bm{U}_{\bm F} \bm{F}   \bm{V}_{\bm F}^\T &\bm{U}_{\bm G} \bm{G}   \bm{V}_{\bm G}^\T \\[0.5em]
    \bm{U}_{\bm H} \bm{H}  \bm{V}_{\bm H}^\T & \bm{U}_{\bm K} \bm{K}  \bm{V}_{\bm K}^\T \end{bmatrix}  = \underbrace{\begin{bmatrix} \bm{U}_{\bm F} & & \bm{U}_{\bm G} & \\ 
    & \bm{U}_{\bm H} & & \bm{U}_{\bm K} \end{bmatrix}}_{\bm R^{L-1}} \underbrace{ \begin{bmatrix} \bm{F}   & & & \\
    & & \bm{H}  & \\
    & \bm{G}   & &\\
    & & & \bm{K}  \end{bmatrix} }_{\bm B^{L-1}} \underbrace{\begin{bmatrix} \bm{V}_{\bm F}^\T & \\
    & \bm{V}_{\bm G}^\T \\
    \bm{V}_{\bm H}^\T & \\
    & \bm{V}_{\bm K}^\T \end{bmatrix}}_{\bm W^{L-1}}.
\end{align*}

Thus, the left and right matrices exactly match the definitions of the outer transfer matrices $\bm R^{L-1}$ and $\bm W^{L-1}$. 

To recover the next level's factors, we note that  the matrices $\bm F, \bm G, \bm H, \bm K \in \R^{2^{L-1}k \times 2^{L-1}k}$ have the block structure
$$
\bm Y = \begin{bmatrix}
    \bm U_{\bm Y_{11}} \bm Y_{11} \bm V_{\bm Y_{11}}^\T  & \bm U_{\bm Y_{12}}\bm Y_{12} \bm V_{\bm Y_{12}}^\T  \\
    \bm U_{\bm Y_{21}}\bm Y_{21} \bm V_{\bm Y_{21}}^\T  & \bm U_{\bm Y_{22}} \bm Y_{22} \bm V_{\bm Y_{22}}^\T 
\end{bmatrix}, \quad \bm Y \in \{\bm F, \bm G, \bm H, \bm K\},
$$
where each subblock of the four subblocks is in  $\butterfly(L-1, k)$.
We expand these blocks in $\bm B^{L-1}$ to form the factorization:
$$ \bm B^{L-1} = 
\begin{bmatrix} \bm{F}   & & & \\
    & & \bm{H}  & \\
    & \bm{G}   & &\\
    & & & \bm{K}  \end{bmatrix} = \bm R^{L-2} \bm B^{L-2} \bm W^{L-2}, 
$$
where 
\[
\scalebox{0.6}{$
\bm R^{L-2}
= \left [ \arrayrulecolor{gray!30}
\begin{array}{c|c|c|c|c|c|c|c!{\color{gray!80}\vrule}c|c|c|c|c|c|c|c} 
\bm{U}_{\bm{F}_{1, 1}}& & & & \bm{U}_{\bm{F}_{1, 2}} & & & & & & & & & & & \\
\hline
& \bm{U}_{\bm{F}_{2, 1}} & & & & \bm{U}_{\bm{F}_{2, 2}} & & & & & & & & & & \\
\hline
& & \bm{U}_{\bm{H}_{1, 1}} & & & & \bm{U}_{\bm{H}_{1, 2}} & & & & & & & & & \\
\hline
& & & \bm{U}_{\bm{H}_{2, 1}} & & & & \bm{U}_{\bm{H}_{2, 2}} & & & & & & & & \\
\arrayrulecolor{gray!80} \hline \arrayrulecolor{gray!30}
& & & & & & & & \bm{U}_{\bm{G}_{1, 1}} & & & & \bm{U}_{\bm{G}_{1, 2}} & & & \\
\hline
& & & & & & & & & \bm{U}_{\bm{G}_{2, 1}} & & & & \bm{U}_{\bm{G}_{2, 2}} & & \\
\hline
& & & & & & & & & & \bm{U}_{\bm{K}_{1, 1}} & & & & \bm{U}_{\bm{K}_{1, 2}} & \\
\hline
& & & & & & & & & & & \bm{U}_{\bm{K}_{2, 1}} & & & & \bm{U}_{\bm{K}_{2, 2}} \\
\end{array}
\right ]
$},
\]

$$
\scalebox{0.6}{$
\bm B^{L-2} = \left[ \arrayrulecolor{gray!30} \begin{array}{c|c|c|c|c|c|c|c!{\color{gray!80}\vrule}c|c|c|c|c|c|c}
\bm{F}_{1,1} & & & & & & & & & & & & & & \\
\hline
& & & & \bm{F}_{2,1} & & & & & & & & & & \\
\hline
& & & & & & & & \bm{H}_{1,1} & & & & & & \\
\hline
& & & & & & & & & & & \bm{H}_{2,1} & & & \\
\hline
& \bm{F}_{1,2} & & & & & & & & & & & & & \\
\hline
& & & & & \bm{F}_{2,2} & & & & & & & & & \\
\hline
& & & & & & & & & \bm{H}_{1,2} & & & & & \\
\hline
& & & & & & & & & & & & \bm{H}_{2,2} & & \\
\arrayrulecolor{gray!80} \hline \arrayrulecolor{gray!30}
& & \bm{G}_{1,1} & & & & & & & & & & & & \\
\hline
& & & & & & \bm{G}_{2,1} & & & & & & & \\
\hline
& & & & & & & & & & \bm{K}_{1,1} & & & \\
\hline
& & & & & & & & & & & & & \bm{K}_{2,1} & \\
\hline
& & & \bm{G}_{1,2} & & & & & & & & & & \\
\hline
& & & & & & & \bm{G}_{2,2} & & & & & & & \\
\hline
& & & & & & & & & & & \bm{K}_{1,2} & & & \\
\hline
& & & & & & & & & & & & & & \bm{K}_{2,2} \\
\end{array}\right]$},
$$
and
\[
\scalebox{0.6}{$
\bm W^{L-2}
=
\left[
\arrayrulecolor{gray!30}
\begin{array}{c|c|c|c!{\color{gray!80}\vrule}c|c|c|c}
\bm{V}_{\bm{F}_{1,1}}^\T & & & & & & & \\
\hline
& \bm{V}_{\bm{F}_{1,2}}^\T & & & & & & \\
\hline
& & \bm{V}_{\bm{G}_{1,1}}^\T & & & & & \\
\hline
& & & \bm{V}_{\bm{G}_{1,2}}^\T & & & & \\
\hline
\bm{V}_{\bm{F}_{2,1}}^\T & & & & & & & \\
\hline
& \bm{V}_{\bm{F}_{2,2}}^\T & & & & & & \\
\hline
& & \bm{V}_{\bm{G}_{2,1}}^\T & & & & & \\
\hline
& & & \bm{V}_{\bm{G}_{2,2}}^\T & & & & \\
\arrayrulecolor{gray!80}\hline
\arrayrulecolor{gray!30}
& & & & \bm{V}_{\bm{H}_{1,1}}^\T & & & \\
\hline
& & & & & \bm{V}_{\bm{H}_{1,2}}^\T & & \\
\hline
& & & & & & \bm{V}_{\bm{K}_{1,1}}^\T & \\
\hline
& & & & & & & \bm{V}_{\bm{K}_{1,2}}^\T \\
\hline
& & & & \bm{V}_{\bm{H}_{2,1}}^\T & & & \\
\hline
& & & & & \bm{V}_{\bm{H}_{2,2}}^\T & & \\
\hline
& & & & & & \bm{V}_{\bm{K}_{2,1}}^\T & \\
\hline
& & & & & & & \bm{V}_{\bm{K}_{2,2}}^\T
\end{array}
\right]
$}
\]
Using the recursive factorization, one can repeatedly rewrite the blocks of $\bm B^{\ell}$ to obtain the transfer matrices at subsequent levels.